\documentclass[aos,preprint]{imsart}

\RequirePackage{amsthm,amsmath}
\RequirePackage[numbers]{natbib}
\RequirePackage[colorlinks,citecolor=blue,urlcolor=blue]{hyperref}
\usepackage{amssymb}
\usepackage{graphicx}
\usepackage{amsfonts}
\usepackage{amsbsy}
\usepackage{mathrsfs}
\usepackage{lscape}
\usepackage{float}
\theoremstyle{plain}

\usepackage{multirow}
\usepackage{tablefootnote} 
\usepackage{threeparttable}
\usepackage{pdflscape}
\usepackage{enumerate}
\usepackage{pdfpages}
\usepackage{soul}
\usepackage{algorithm,algpseudocode}
\usepackage{footnote}
\usepackage{subcaption}
\usepackage{wrapfig}
\usepackage{verbatim}
\usepackage{upgreek}
\usepackage{mathabx}
\usepackage[shortlabels]{enumitem}

\startlocaldefs
\numberwithin{equation}{section}
\theoremstyle{plain}
\newtheorem{theorem}{Theorem}
\usepackage[english]{babel}
\newtheorem{lemma}{Lemma}
\newtheorem{remark}{Remark}
\newtheorem{proposition}{Proposition} 
\newtheorem{corollary}{Corollary}
\endlocaldefs

\newcommand{\bx}{\mathbf{x}}
\newcommand{\bX}{\mathbf{X}}

\newcommand{\by}{\mathbf{y}}
\newcommand{\bz}{\mathbf{z}}

\newcommand{\bS}{\mathbf{S}}
\newcommand{\bt}{\mathbf{t}}

\newcommand{\br}{\mathbf{r}}
\newcommand{\bg}{\mathbf{g}}

\newcommand{\bV}{\mathbf{V}}
\newcommand{\bM}{\mathbf{M}}

\newcommand{\btheta}{\boldsymbol{\theta}}

\newcommand{\bSigma}{\boldsymbol{\Sigma}}

\newcommand{\EE}{\mathbb{E}}
\newcommand{\Var}{\mathrm{Var}}

\graphicspath{{../Figures/}}

\usepackage{xcolor}

\definecolor{brightgreen}{rgb}{0.4, 1.0, 0.0}

\begin{document}

\begin{frontmatter}

\title{Pairwise interaction function estimation of stationary Gibbs point processes using basis expansion}
\runtitle{Orthogonal series estimation for Gibbs models}

\begin{aug}

\author[A]{\fnms{Isma\"ila} \snm{Ba}\thanksref{t1,t2}\ead[label=e1]{ba.ismaila@courrier.uqam.ca}},
\author[B]{\fnms{Jean-Fran\c cois} \snm{Coeurjolly}\thanksref{t1,t3}\ead[label=e2]{jean-francois.coeurjolly@univ-grenoble.fr}}
\and 
\author[C,D]{\fnms{Francisco} \snm{Cuevas-Pacheco}\thanksref{t4}}
\ead[label=e3]{francisco.cuevas@usm.cl}
\ead[label=e4]{francisco.cuevas.pacheco@gmail.com}

\address[A]{Department of Mathematics,
Universit\'e du Qu\'ebec \`a Montr\'eal (UQAM), \\ Canada,
\printead{e1}}

\address[B]{Laboratory Jean Kuntzmann, Department DATA,
Universit\'e Grenoble Alpes, \\ France,
\printead{e2}}

\address[C]{Departamento de Matem{\'a}tica,
Universidad T{\'e}cnica Federico Santa Mar{\'i}a, Valparaiso, \\ Chile,
\printead{e3}}

\address[D]{Advanced Center for Electrical and Electronic Engineering,
Universidad T{\'e}cnica Federico Santa Mar{\'i}a, \\ Valparaiso, Chile
}

\thankstext{t1}{Natural Sciences and Engineering Research Council of Canada.}
\thankstext{t2}{Institut des Sciences Mathématiques (ISM).}
\thankstext{t3}{French National Research Agency in the framework of the ``Investissements d’avenir” program (ANR-15-IDEX-02).}
\thankstext{t4}{National Agency for Research and Development of Chile, through grant ANID/FONDECYT/POSTDOCTORADO/No. $3210453$ and by the AC3E, UTFSM, under grant FB-$0008$.}
\runauthor{I. BA, J-F. COEURJOLLY AND F. CUEVAS-PACHECO}

\end{aug}

\begin{abstract}
The class of Gibbs point processes (GPP) is a large class of spatial point processes able to model both clustered and repulsive point patterns. They are specified by their conditional intensity, which for a point pattern $\mathbf{x}$ and a location $u$, is roughly speaking the probability that an event occurs in an infinitesimal ball around $u$ given the rest of the configuration is $\mathbf{x}$. The most simple and natural class of models is the class of pairwise interaction point processes where the conditional intensity depends on the number of points and pairwise distances between them. 
This paper is concerned with the problem of estimating  the pairwise interaction function non parametrically. We propose to estimate it using an orthogonal series expansion of its logarithm. Such an approach has numerous advantages compared to existing ones. The estimation procedure is simple, fast and completely data-driven. We provide asymptotic properties such as consistency and asymptotic normality and show the efficiency of the procedure through simulation experiments and illustrate it with several datasets.
\end{abstract}





\end{frontmatter}

\section{Introduction}
Spatial point processes are widely used to model data in the form of events or objects within a region (commonly called spatial domain). There are applications of spatial point process models in a broad range of disciplines for instance in forestry and plant ecology, epidemiology, seismology, astronomy, and others \cite{baddeley2006case,illian2008statistical}. Due to their ability to model a wide variety of spatial point patterns, the class of Gibbs point processes is an important class of models for spatial point pattern analysis~\cite[see e.g.][]{gates1986clustering,moller2003statistical,illian2008statistical,baddeley2015spatial,dereudre2019introduction}. Gibbs point processes are however well-known for their lack of tractability \cite[see e.g.][]{dereudre2019introduction}: even for simple  models such as the Strauss process \cite{strauss1975model}, intensity functions  as well as standard summary statistics used in point pattern analysis like the Ripley's $K$-function, the nearest-neighbour $G$ function or the empty space function $F$, are not explicit. To rephrase this comment, this means that even for simple parametric  models, it is in general impossible to set the parameters of a Gibbs model such that the average number of points in a given window has a prescribed value.

Therefore, Gibbs models are  not defined through intensity functions. Instead, in a bounded domain of $\mathbb R^d$, they are characterized by a density with respect to the (unit rate) Poisson point process, the reference process modelling random points without any interaction. The flexibility of the density makes the modelling simple and natural. One has a mechanistic interpretation of the model and produced patterns can be highly regular, clustered or exhibit a complex mixture of these two characteristics.
This paper is focused on the class of pairwise interaction point processes, for which the density depends (in the homogeneous case)  on the number of points and on pairs of points. 
From a statistical point of view, it is not  convenient to work with the density of a Gibbs point process, as the latter is, in general, defined up to a normalizing constant. The quantity of interest, which overcomes this problem, is the Papangelou conditional intensity function \cite[see again][]{dereudre2019introduction}. This function, which can be seen as a ratio of two densities, is interpreted as the probability to observe a point in the vicinity of a location, say $u$, given the rest of the configuration of points is $\mathbf x$.
In this paper, we focus on stationary and isotropic pairwise interaction Gibbs point processes which are defined (see Section~\ref{sec:gpps} for a more formal and precise definition) via their Papangelou conditional	 intensity given for $u\in \mathbb R^d$ and $\bx$  a configuration of points by
\begin{align} \label{pipp0}
\lambda(u,\bx) = \beta \prod_{v \in \bx} \phi (\Vert v - u \Vert)
\end{align}
and we investigate the problem of estimating simultaneously the non-negative parameter $\beta$ and the function $\phi$.

Methods of nonparametric curve estimation allow one to analyze and present data at hand without any prior about the data~\cite{wasserman2006all,efromovich2008nonparametric} i.e. these methods let the data speak for themselves. Kernel estimation and orthogonal series estimation \cite[see e.g.]{whittle1958smoothing,parzen1962estimation,watson1969density} are among some of them. In the context of spatial point processes, a considerable amount of research has addressed nonparametric estimation for  first and second order summary functions, i.e. the intensity function, the pair correlation function (which is a normalized version of the second-order intensity function), and summary functions like Ripley's K function or functions F, G, J, L~\cite[][]{stoyan1994fractals,moller2003statistical,illian2008statistical,baddeley2015spatial}. The current literature covers theoretical as well as practical aspects and all these nonparametric estimates are implemented in particular in \texttt{R} packages such as \texttt{spatstat,GET,stpp} \cite{baddeley2015spatial,myllymaki2019get,gabriel2013stpp}.

Nonparametric estimation of the pairwise interaction function of a Gibbs point process has also been investigated but has received  much less attention.
\begin{itemize}
\item \cite{takacs1986estimator,fiksel1988estimation} are frequent references when the pairwise interaction function estimation problem is evoked. These short notes suggest to model the pairwise interaction function by a step function and to estimate the steps using moment type methods. Both references contain neither discussion, nor simulation, nor data-driven procedure to calibrate the number of steps and  jump points. This approach is definitely not a purely nonparametric one. Actually, the ideas developed in these references have been popularized in the literature of parametric Gibbs models under the terminology Takacs-Fiksel method, see e.g.~\cite{coeurjolly2012takacs} and the references therein. Modelling a pairwise interaction function or higher-order interaction by a step function is a recurrent idea in parametric Gibbs modelling and estimation \cite[see e.g.][]{berthelsen2003likelihood,billiot2008maximum,raeisi2021spatio,iftimi2018multi}. \cite{billiot2001estimation} propose also a similar approach to model the pairwise interaction function of spherical point processes by using a fixed truncated Fourier series expansion. As outlined hereafter, our work should be seen as an appropriate treatment of these ideas in a nonparametric context, since we propose an orthogonal series expansion, which, by choosing for instance the Haar basis, see \cite{haar1909theorie} (see also Section~\ref{subsec:imp}), will lead to a piecewise constant estimate (with jump points, steps and number of breakpoints simultaneously estimated).
\item The work by~\cite{diggle1987nonparametric} is probably the only one in a purely nonparametric frequentist framework. The authors suggest to use an approximation from statistical physics due to  \cite{percus1964pair}, \cite[Chapter~5]{hansen1976theory} to derive an approximate integral equation relating the pair correlation function to the pairwise interaction function of a stationary Gibbs point process. It is worth pointing out that the quality of this approximation is not quantified nor empirically measured. The methodology results in plugging a standard nonparametric estimate of the pair correlation function into the integral equation and to invert this equation using several Fourier approximations. The approach proposed by~\cite{diggle1987nonparametric} suffers from several drawbacks. As noticed by the authors and also by \cite[p.519]{baddeley2015spatial}, the implementation is not straightforward yielding a sensitive and computationally expensive procedure. Due to the not well-understood integral equation approximation, it is impossible to simply transfer properties for the pair correlation function estimator \cite[see e.g.][]{david2010central} to that of the pairwise interaction function one. Thus, consistency, behaviour of the  integrated square error or asymptotic normality are not available.
\item Bayesian approaches have been considered by~\cite{heikkinen1999bayesian}, \cite{berthelsen2003likelihood}. \cite{bognar2004bayesian,berthelsen2008non} also investigated the situation where the parameter $\beta$ in~\eqref{pipp0} is replaced by a function. \cite{heikkinen1999bayesian} approximates model~\eqref{pipp0} by a sum of step functions while~\cite{berthelsen2003likelihood} assume that the pairwise interaction function is approximated by local linear functions. The difference between these two works lies essentially in the following facts. First, \cite{heikkinen1999bayesian} propose a procedure conditionally on the observed number of points and fix the number of steps in their procedure. Second, \cite{berthelsen2003likelihood} assume that the pairwise interaction function $\phi$ is increasing and bounded by~1. This assumption corresponds to the situation of a purely repulsive point pattern, see Corollary~\ref{cor:existence} for more detail. Third, both works propose a different prior distribution for the pairwise interaction function and require thus a different calibration of some tuning parameters. Both approaches require a lot of computations and simulations. On the contrary, the frequentist procedure requires by nature less assumptions on the function; in particular we do not assume $\phi$ to be an increasing step function bounded by $1$. Our methodology is able to fit  clustered, regular or more complex structured  pairwise interaction Gibbs models. In Section~\ref{sec:app} in the supplementary material, we propose to compare our results with these two Bayesian procedures on a standard dataset (available in the \texttt{R} package \texttt{spatstat}) which exhibits repulsion.
\end{itemize}

As a summary, there is no methodology which, simultaneously, is theoretically studied, easily implementable and computationally feasible in a very reasonable time. The aim of this paper is to fill this gap and to address nonparametric estimation of the pairwise interaction function of a stationary GPP via series expansion, an approach which has never been considered for pairwise interaction Gibbs point processes. In a recent study, \cite{jalilian2019orthogonal,coeurjolly2019second} have adapted orthogonal series density estimators~\cite[see e.g.][]{hall1987cross,efromovich2008nonparametric} to the estimation of the pair correlation function or its logarithm of a spatial point process. 
In the same vein, we  develop an orthogonal series estimation of the log pairwise interaction function, i.e. $\log \phi$, for stationary Gibbs point processes. Our approach, detailed in Section~\ref{sec:meth}, is to truncate the series expansion of $\log \phi$ and to estimate coefficients of the truncated series  by maximizing a truncated version of the log pseudolikelihood function~(\cite{besag1974spatial,jensen1991pseudolikelihood}). The pseudolikelihood method is an efficient and very popular  parametric method, alternative to the maximum likelihood method since MLE is first complex to implement (due to the fact that the density of the Gibbs point process is defined up to a normalizing constant depending on the parameters) and second very complex to study theoretically. Consistency for the MLE is only proved and under for some particular models~\cite{dereudre2017consistency}. The pseudolikelihood is roughly speaking a composite likelihood which involves only the conditional intensity function. Since its introduction in the context of spatial point processes by \cite{jensen1991pseudolikelihood}, the pseudolikelihood has gained in popularity. It has been theoretically studied in different contexts for finite range or infinite range models, i.e. when the function $\phi$ equals one or not after a certain distance~\cite[see][]{jensen1994asymptotic,billiot2008maximum,coeurjolly2017parametric}, or for high-dimensional inhomogeneous models \cite{bainference}. Numerically the pseudolikelihood can be very well approximated by a quasi-Poisson regression. Hence the method can be efficiently implemented using standard software and in particular within the \texttt{R} package \texttt{spatstat} \cite{baddeley2000practical} dedicated to spatial point pattern analysis.

Our theoretical contribution is to provide conditions on the homogeneous pairwise interaction Gibbs model (local stability, finite range property, etc) and on the orthonormal basis system to obtain the consistency and  asymptotic normality for the truncated pseudolikelihood estimator and the orthogonal series estimator. These results are established in an increasing domain setting i.e. a setting where the number of points in the observation domain grows with the volume of the observation domain. This asymptotic framework is common in spatial statistics~\cite[see e.g.][]{book:981816}. In particular, Theorem~\ref{THM:CONST} shows the expected compromise between the number of polynomials $K$ used in the series expansion and the amount of acquired data expressed by the volume of the observation domain which has to be considered to get a consistent estimator. Our work is inspired by~\cite{jalilian2019orthogonal} but is fundamentally different. First, the referred paper focuses on the pair correlation function. As already mentioned, the pair correlation function of a Gibbs point process is not explicit and cannot be directly related to its pairwise interaction function. Second,~\cite{jalilian2019orthogonal} have taken advantage of the fact that the kernel estimation of the pair correlation is explicit. This is the main reason why coefficients of the series expansion of the pair correlation can be estimated separately, without any optimization procedure making estimators easier to study from a theoretical point of view.

The choice to expand $\log \phi$ instead of $\phi$ combined with the truncated pseudolikelihood method used to estimate the coefficients brings back the problem to the maximization of a concave function of the parameters \cite[see][]{jensen1991pseudolikelihood}. Therefore the method is very fast and  numerically stable. In addition,
\cite{coeurjolly2013fast} have proposed  fast and consistent estimates of the sensitivity matrix and variance covariance matrix of the score of the pseudolikelihood. Such estimates are  implemented in the \texttt{spatstat R} package. Another advantage of our orthogonal series expansion approach is that we can take advantage of this work and propose asymptotic pointwise confidence intervals for $g(r)=\log \phi(r)$ (and $\phi(r)$) for any $r$. By passing, working with an orthonormal basis system is appealing  as we are able to prove that normalized versions of the sensitivity matrix and variance of the score of the pseudolikelihood are positive definite for $n$ sufficiently large (see Lemmas~\ref{lemma1}-\ref{lemma2} for more details). Such a result is essential to ensure the uniqueness maximum of the truncated pseudolikelihood and to evaluate estimates of the asymptotic variance of $\hat g(r)$ and $\hat \phi (r)$.

The orthogonal series expansion estimator we propose naturally depends on a smoothing parameter, here the number $K$ of orthonormal functions in the series expansion. Tuning this parameter or the bandwidth of a kernel is a standard and recurrent problem in functional estimation~\cite[see e.g.][for the density estimation problem]{hall1987cross,efromovich2008nonparametric}. We propose to select the smoothing parameter $K$ using a composite information criterion, namely the composite AIC \cite{gao2010composite,choiruddin2021information}. This information criterion has the advantage to be quickly evaluated making the whole methodology fully data-driven.

Several perspectives could be undertaken based on the present work. First, we assume that the process has a finite range. To say it quickly, see~\eqref{fr}, we assume that $\phi(r)=1$ for $r\ge \tilde R$ form some $0<\tilde R <\infty$. Extending the methodology to infinite-range models such as the Lennard-Jones model \cite[e.g.][]{coeurjolly2010asymptotic} could be of interest. One limitation of~\eqref{pipp0} is also that we focus on homogeneous models. Letting for instance the parameter $\beta$ to evolve with space and estimate non parametrically both functions $\beta$ and $\phi$ is definitely an interesting but challenging problem from a theoretical point of view.

The rest of the paper is organized as follows. In Section~\ref{sec:gpps}, we give a brief review on pairwise interaction point processes, recall the Georgii-Nguyen-Zessin formula and define some important concepts (translation invariance, finite range, local stability) which ensure the existence of a stationary spatial GPP. We also provide a few examples of pairwise interaction functions in this section.
The orthogonal series expansion  procedure we propose is presented in Section \ref{sec:meth}. 
Section~\ref{sec:asymp} deals with asymptotic properties and presents our main theoretical results. Section~\ref{sec:sim} contains numerical considerations. We discuss the implementation, the data-driven selection of $K$ and propose a large simulation study. We consider several pairwise interaction functions,  orthonormal basis functions and increasing observation domains. We also investigate the misspecification or estimation of irregular parameters such as the hard core and finite range parameters. An application  to real datasets as well as proofs of our main results are provided in the Supplementary Material.

\section{Pairwise interaction Gibbs point processes}
\label{sec:gpps}

The framework of this paper is concerned with stationary and isotropic Gibbs point processes $\bX$ defined on the infinite volume $\mathbb{R}^d$. We view $\bX$ as a locally finite subset of $\mathbb{R}^d$, that is the set $\bX_B:=\bX \cap B$ is finite almost surely for any bounded Borel set $B$ of $\mathbb{R}^d$. This means also that the number of points $N(B)=|\bX_B|$ is finite almost surely. We denote by ${N}_{lf}$, the space of locally finite configurations in $\mathbb{R}^d$, that is
\[
{N}_{lf} = \{ \bx, |\bx_B| < \infty \; \mbox{for any bounded domain} \; B \subset \mathbb{R}^d \}.
\]
The reader interested by a deep presentation of GPPs is referred to \cite{dereudre2019introduction}, \cite{daley2007introduction}, \cite{moller2003statistical} and \cite{georgii1979canonical}. We consider particularly the class of stationary and isotropic pairwise interaction point processes (PIPP) where the Papangelou conditional intensity is given by 
\begin{align} \label{pipp}
\lambda(u,\bx) = \beta \prod_{v \in \bx} \phi (\Vert v - u \Vert)
\end{align}
where $\beta > 0$ represents the activity parameter and $\phi$ is a real-valued interaction function. We are often interested in the log-linear form of the Papangelou conditional intensity which writes 
\begin{align}
 \label{condint}
\log \lambda(u,\bx) = \theta_0 + \sum_{v \in \bx} g (\Vert v - u \Vert)
\end{align}
where $\theta_0:=\log(\beta)$ and $g:=\log(\phi)$.    
For a configuration of points $\bx$ and a location $u \in \mathbb{R}^d$, the quantity $\lambda(u,\bx)$ can be viewed as the conditional probability to have a point in an infinitesimal ball around $u$ given the configuration elsewhere is $\bx$. 

The GNZ (for Georgii, Nguyen and Zessin) formula~\cite[see][]{xanh1979integral,georgii1979canonical,georgii2011gibbs} is a way of characterizing the GPP and also highlights the interest of the Papangelou conditional intensity. It states that, for any measurable function $h : \mathbb R^d \times {N}_{lf} \to \mathbb{R}^{+}$ such that the following expectations, which are denoted by $\EE$ and defined with respect to the distribution of $\bX$,  are finite
\begin{equation}
\label{gnz}
\EE \sum_{u \in \bX} h(u, \bX \setminus \{u\})  =  
 \EE \int    h(u ,\bX) \lambda(u, \bX)  \mathrm{d}u.
\end{equation}
Throughout the paper, we will often use the following three properties for a function $f : \mathbb{R}^d \times {N}_{lf} \to \mathbb{R}^{+}$,
\begin{enumerate}[(i)]
\item $f$ is translation invariant, if
\begin{equation}
\label{ti}
f(u,\bx)=f(\mathrm{0},\uptau_{u}\bx) \tag{TI}
\end{equation}
where $\uptau_{u}\bx = \{v-u: v \in \bx \}$ is the translation of the locations of $\bx$ by the vector $-u$.
\item $f$ has finite interaction range $\tilde R > 0$, if
\begin{equation}
\label{fr}
f(u,\bx) = f(u,\bx \cap B(u,\tilde R)) \tag{FR}
\end{equation}
where $B(u,\tilde R)$ is the euclidean ball centered at $u$ with radius $\tilde R$.
\item $f$ is locally stable, if there exists $\bar{f} < \infty$ such that 
\begin{equation}
\label{ls}
f(u,\bx) \leq \bar{f}. \tag{LS}
\end{equation}
\end{enumerate}
When the Papangelou conditional intensity plays the role of the function $f$, the finite range  property~\eqref{fr} means that the probability to insert  a point $u$ in $\bx$ depends only on the $\tilde R$-neighbors of $u$ to $\bx$. The local stability \eqref{ls} means that the GPP is stochastically dominated by a Poisson point process. And the translation invariance  property~\eqref{ti} leads to the stationarity of the point process. The following proposition proved by \cite{bertin1999existence} is central to this paper as it provides sufficient conditions to ensure the existence of at least one stationary GPP on $\mathbb{R}^d$.  

\begin{proposition} \label{prop0}
Let $\lambda :  \mathbb{R}^d \times {N}_{lf} \to \mathbb{R}^{+}$ be a function of the form (\ref{pipp}) satisfying \eqref{ti}, \eqref{fr} and \eqref{ls}. Then there exists at least one infinite volume stationary Gibbs measure, that is there exists at least one stationary Gibbs model $\bX$ with Papangelou conditional intensity $\lambda$ which in particular satisfies the GNZ formula (\ref{gnz}). 
\end{proposition}

More specifically, for pairwise interaction point processes, we have the following useful result.

\begin{corollary} \label{cor:existence}
Let $\lambda :  \mathbb{R}^d \times {N}_{lf} \to \mathbb{R}^{+}$ given by~\eqref{pipp}-\eqref{condint}. Assume 
\begin{itemize}
\item[(a)] either  $g(r)=-\infty$ (i.e. $\phi(r)=0$) for $r\in [0,\delta]$ (with $\delta>0$) and $g(r)=0$ (i.e. $\phi(r)=1$) for $r\ge \tilde R$ for some $0<\tilde R<\infty$.
\item[(b)] or $g(r)\ge0$ (i.e. $\phi(r)\le 1$) for  $r\in [0,\tilde R]$  and $g(r)=0$ (i.e. $\phi(r)=1$) for  $r\ge \tilde R$ for some $0<\tilde R<\infty$;
\end{itemize}
Then, the corresponding Papangelou conditional intensity satisfies \eqref{ti}, \eqref{fr} and \eqref{ls} and thus Proposition~\ref{prop0} applies.  	
\end{corollary}

In the situation (b), $g$ is non-negative which necessarily yields a repulsive PIPP. In the situation (a), assuming $g(r)=-\infty$ for $r\in[0,\delta]$  means that two points at distance smaller than $\delta$ are forbidden. This condition is known as hard core condition, see \cite{dereudre2019introduction} and the references therein. The hard core condition at small scales enables $g$ to eventually take negative values when $r\ge \delta$, which can result in attractive patterns. In both set of assumptions (a)-(b), we assume that the pairwise interaction has finite range. Infinite range PIPP such as the Lennard-Jones model \cite[see existence conditions in ][]{ruelle69,ruelle70} are not considered in this paper. Assuming~\eqref{fr} is relatively standard in statistical modelling and inference for GPP \cite[see e.g.][]{jensen1994asymptotic,billiot2008maximum,bainference}. We return to this limitation in Section~\ref{sec:conclusion}. 
Assumptions considered in Corollary~\ref{cor:existence} are satisfied by classical examples of PIPP like the Strauss model, the piecewise Strauss model and their hard core versions \cite{strauss1975model}, the annulus piecewise Strauss model proposed by~\cite{stucki2014bounds}, the Diggle-Graton model \cite{diggle1984monte}, the Diggle-Gates-Stibbard model \cite{diggle1987nonparametric}, the Fiksel model \cite{fiksel1984estimation}, etc. All these examples are in particular implemented in the \texttt{R} package \texttt{spatstat} \cite{baddeley2015spatial}.

Finally, let us add the following comment that the local stability property implies a stochastic domination of the GPP by a homogeneous Poisson point process. Hence, for all $c>0,p\ge 1$, $A \subset \mathbb{R}^d$, $A$ bounded,  $\EE \left[ N(A)^p \exp(cN(A))\right] < \infty$. Note that this upper-bound is trivial in the hard core situation as this constraint implies that $N(A)$ is an almost surely bounded random variable.

To illustrate this section and this paper, we consider four examples of pairwise interaction functions. To ease notation, it is sometimes useful to encompass both cases (a)-(b), and thus assume that $g(r)=-\infty$ (or $\phi(r)=0$) for $r\in[0,\delta]$ with $\delta\ge 0$ (thus $\delta=0$ means no hard core) and $g(r)=0$ (or $\phi(r)=1$) for $r\ge \tilde R:=R+\delta$ for some $R<\infty$. The set of functions $\phi_k$ for $k=1,\dots,4$ is shown in Figure~\ref{fig:Pif} (first column). The logarithm of these functions, denoted by $g_k=\log \phi_k$ is depicted in Figure~\ref{fig:Pif} (second column). 
\begin{enumerate}
\item Let $\delta=0$ and $R=0.08$
\[  \phi_1(r) =  \left\{
\begin{array}{ll}
      \alpha + \Big( 1-\alpha \Big) \Big(\frac{r}{R}\Big)^{2} & 0\le r \leq R \\
      1 & \mbox{otherwise} \\
\end{array} 
\right. \]
where $\alpha=0.05$.
\item Let $\delta=0$ and $R=0.08$
\[
\phi_2(r) =  \left\{ \begin{array}{ll} 0.5 \quad  \; \; r \leq 0.01 \\ 
        0.2  \quad \; \; 0.01 < r \leq 0.04 \\
        0.6  \quad \; \; 0.04 < r \leq 0.06 \\
        0.8  \quad \; \; 0.06 < r \leq R \\
        1 \quad \quad \;  r > R.
        \end{array} \right.
\]
\item Let $\delta=0.01$ and $R=0.07$. 
\[
\phi_3(r) = \left\{ \begin{array}{ll} 0 \qquad \qquad \qquad \quad \quad \quad \, r \leq \delta \\ 
        a r^3 + b r^2 + c r + d \quad \; \; \; \, \delta < r \leq R + \delta \\
        1 \qquad \qquad \qquad \quad \quad \quad \, r > R + \delta
        \end{array} \right.
\]
where the real numbers $a$, $b$, $c$ and $d$ are such that the following conditions are satisfied: $\phi_2(\delta) = 1.3$, $\phi_2(0.03) = 1.05$, $\phi_2(0.06) = 1.2$ and $\phi_2(R + \delta) = 1$. \\
\item Let $\delta=0.01$ and $R=0.07$
\[
\phi_4(r) =  \left\{ \begin{array}{ll} 0 \qquad \qquad \qquad  \quad \quad  \quad \qquad \qquad \qquad \qquad \qquad  \; \quad r \leq \delta \\ 
        \eta \Big(1 + \cos(\frac{3 \pi}{2} \cdot \frac{r-\delta}{R})\Big) + \Big(1-\eta \Big) \Big( (r-\delta)/R \Big)^\gamma \quad \; \; \delta < r \leq R + \delta \\
        1 \qquad \qquad \qquad \quad \qquad \qquad \qquad \qquad \qquad \qquad \quad r > R + \delta
        \end{array} \right.
\]
where  $\eta=2/3$ and $\gamma=1$.
\end{enumerate}

It is straightforwardly seen that the functions $\phi_k$ (or $g_k$) for $k=1,\dots,4$ satisfy the assumptions of Corollary~\ref{cor:existence}. The function $\phi_1$ is a slight modification of the Diggle-Gratton interaction \cite[see e.g.][]{baddeley2015spatial} where we ensure that $\phi_1(0)=0.05>0$ (which is required by our Condition~\ref{C:psi}). The second example $\phi_2$ is a piecewise Strauss process (with no hard core). The first two examples yield repulsive patterns as $\phi_k(r)\le 1$ (i.e. $g_k(r)\ge 0$) for $k=1,2$. For some values of $r$, $\phi_k(r) > 1$, for $k=3,4$. So, to ensure that processes with such interaction functions exist, we must include a hard core condition (see Corollary~\ref{cor:existence}(a)). The PIPP with interaction function $\phi_3$ yields clustering at all scales as $\phi_3(r) \geq 1$ for all distances $r$ while the interaction function $\phi_4$ generates point patterns with a mixture of clustering and inhibition. More precisely, we have clustering for $0.01 \leq r \leq 0.0275$ and inhibition for $0.0275 \leq r \leq 0.08$. The third column in Figure~\ref{fig:Pif} depicts one realization of each PIPP with interaction functions $\phi_k$ (for $k=1,\dots,4$) over the unit square. These simulations are performed using the Metropolis-Hastings algorithm applied to Gibbs point processes \cite{geyer1994simulation} and implemented in the \texttt{R} package \texttt{spatstat}.

\begin{figure}[!ht]
\begin{center}
\renewcommand{\arraystretch}{0}
\includegraphics[width=1\textwidth]{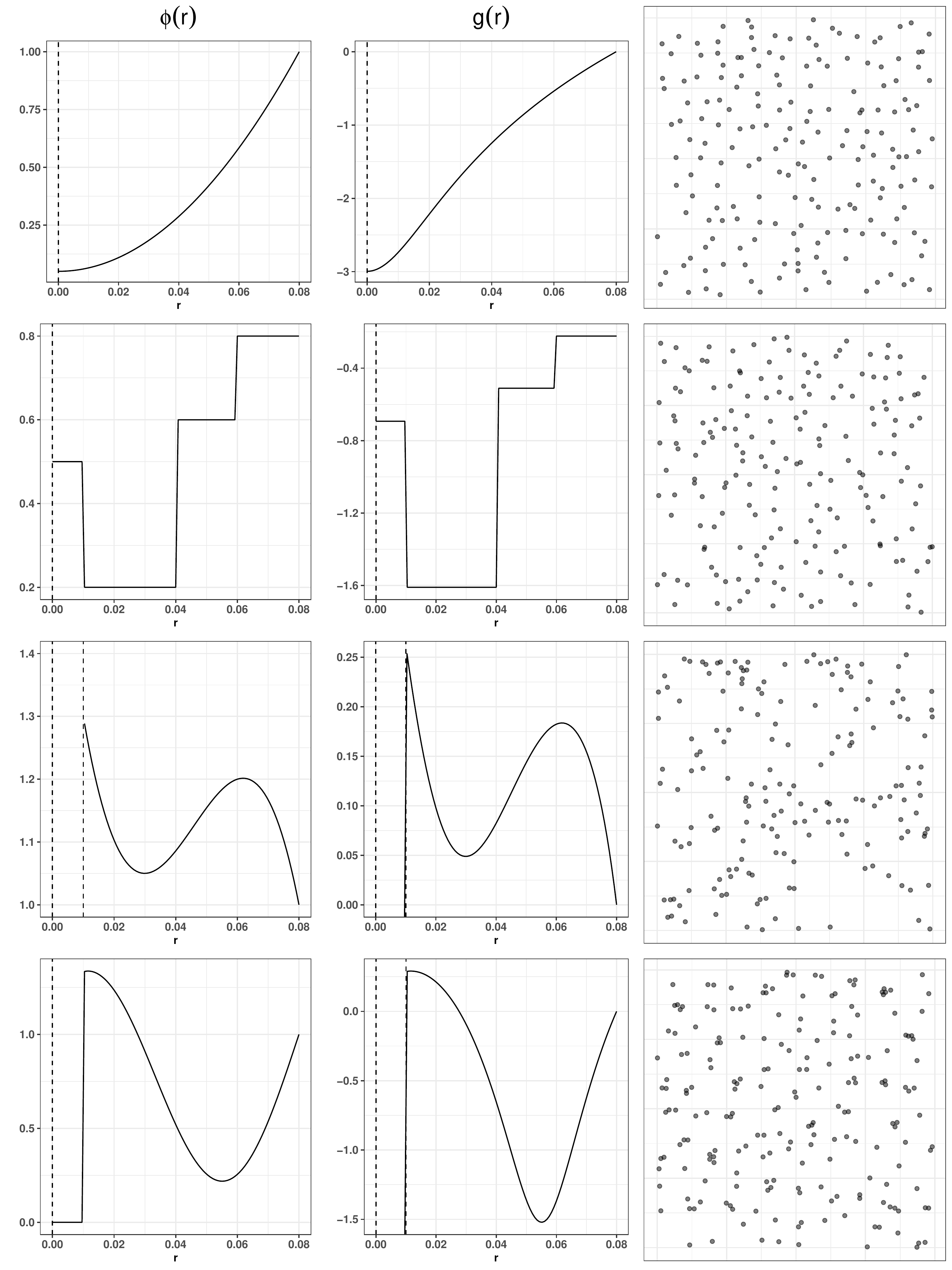} 
\caption{For $k=1,\dots,4$, true pairwise interaction functions (PIF) $\phi_k$ (first column), logarithm of PIFs $g_k=\log \phi_k$ (second column), and one simulated point pattern  over $W=[0,1]^2$ (third column). The activity parameter $\beta$ is set to $1827$, $1240$, $119$ and $1316$ respectively.}
\label{fig:Pif}
\end{center}
\end{figure}

\section{Methodology} \label{sec:meth}

\subsection{Orthogonal series estimation} \label{subsec:orth}

We assume that $g$ satisfies assumptions of Corollary~\ref{cor:existence}.  Our goal is  to estimate $g$ on the interval $[\delta,R+\delta]$.To this end, we decompose $g(\cdot+\delta)$ on a basis of orthonormal polynomials on $[0,R]$. 
Orthonormal polynomials defined on $[0,R]$ are classically built from an inner product of the form $<a,b>_w=\int_0^R a(r)b(r)w(r)\mathrm d r$, where $w$ is a non-negative continuous weight function, \cite[see e.g.][]{szeg1939orthogonal}.

We denote by $\{\varphi_k\}_{k \geq 1}$ the orthonormal basis of functions on $[0,R]$ with respect to some weight function $w$, thus satisfying $\int_0^R \varphi_k(r) \varphi_l(r) w(r) \mathrm{d}r = \mathbf{1}(k=l)$. 
Let $L^2_w([0,R])$ be the set of squared integrable functions on $[0,R]$ with respect to the inner product $<\cdot,\cdot>_w$.
Let us cite a few examples of orthonormal polynomials.
The first one is the cosine basis  $\varphi_1(r)=1/\surd{R}$, $\varphi_k(r)=\sqrt{2/R} \cos \{ (k-1) \pi r/ R \}, \, k \geq 2$ with $w(r)=1$. The second one is the Haar system, defined by $\varphi_1(t)=1/\sqrt{R}$ and  where for $k=2^m+\ell$ with $\ell=1,\dots,2^m$ and 
\[
\varphi_k(r)=\frac{\sqrt{2^m}\tilde \varphi_k(r/R)}{\sqrt{R}}   \text{ where } 
\tilde \varphi_k(r)  = 
\mathbf{1} \left(r\in I_{2k-2} \right)
-
\mathbf{1} \left(r\in I_{2k-1} \right)
\]
with $I_k=[k/2^{m+1},(k+1)/2^{m+1}]$ and $w(r)=1$. The third one with $w(r)=r^{d-1}$  is the Fourier-Bessel basis given by
\[
\varphi_k(r)=\frac{2^{1/2}}{R J_{\nu+1}(\alpha_{\nu,k})} \, J_{\nu}(r \alpha_{\nu,k} / R) r^{-\nu}, \, k \geq 1,
\]
where $\nu=(d-2)/2$, $J_{\nu}$ is the Bessel function of the first kind of order $\nu$ and $\{ \alpha_{\nu,k} \}_{k=1}^{\infty}$ is the sequence of successive positive roots of $J_{\nu}(r)$. Other examples of orthonormal polynomials include for example Legendre polynomials, Jacobi polynomials, Chebyshev polynomials of the first and second kind, Gegenbauer polynomials, etc.

Let $g\in L^2_w([\delta,R+\delta])$. Its orthogonal series expansion writes for any $r\in [\delta,R+\delta]$.
\begin{align} \label{orth}
g(r) = \sum_{k = 1}^{\infty} \theta_k \varphi_k(r - \delta).
\end{align}
The coefficients $\theta_k$ in the expansion are therefore given by $\theta_k=\int_0^R g(r+\delta) \varphi_k(r) w(r) \mathrm{d}r$ and satisfy $\sum_{k \geq 1} \theta_k^2 < \infty$. Now, with this series expansion of the logarithm of the pairwise interaction function, we rewrite equation $(\ref{condint})$ as follows:
\begin{align} \label{condintfty}
\log \lambda(u,\bx;\btheta) = \theta_0 + \sum_{k=1}^{\infty} \theta_k S_k(u,\bx)
\end{align}
where $\btheta = (\theta_0,\theta_1,\ldots) \in \ell^2(\mathbb N)=\{\by \in \mathbb{R}^{\mathbb N}: \|\by\|^2:=\sum_{k\ge 0 }y_k^2<\infty\}$ and\linebreak $S_k(u,\bx) = \sum_{v \in \bx} \varphi_k (\Vert v - u \Vert - \delta)$. After truncating the above series to a fixed $K$, Equation $(\ref{condintfty})$ becomes
\begin{align} \label{condintK}
\log \lambda(u,\bx;\btheta_{K}) = \btheta_K^\top \bS_K(u,\bx)
\end{align}
where $\btheta_K = (\theta_0,\theta_1,\ldots,\theta_K)^\top \in \mathbb{R}^{K+1}$ and $\bS_K(u,\bx) = (1,S_1(u,\bx),\ldots,S_K(u,\bx))^\top$. We stress on the fact that $\lambda$, as defined in $(\ref{condintK})$, does not correspond to the conditional intensity of the spatial Gibbs point process $\bX$. However, the interest is that this truncation brings back the model to an exponential family model. Therefore, standard methods including the pseudolikelihood function can be used to estimate $\btheta_K$.  We define 
\begin{align} \label{lpl}
\mathrm{LPL}(\bX;\btheta_K) = \sum_{u \in \bX \cap W} \log \lambda (u,\bX ; \btheta_{K}) - \int_{W} \lambda (u,\bX;\btheta_{K}) \mathrm{d}u.  
\end{align}
Note that taking $K=\infty$ in (\ref{lpl}) precisely leads to the log-pseudolikelihood function of the point process $\bX$. In the present paper, the function $\mathrm{LPL}(\bX;\btheta_K)$ will be referred to as the truncated log-pseudolikelihood of the point process $\bX$ and its derivative with respect to $\btheta_{K}$ is given by
\[
\boldsymbol{e}_{K}(\bX;\btheta_K)=  \sum_{u \in \bX \cap W} \bS_K(u, \bX \setminus u) -  \int_{W} \bS_K(u, \bX)\lambda (u,\bX;\btheta_{K}) \mathrm{d}u. 
\]
The estimate $\hat{\btheta}_K$ is obtained by maximizing the function $\mathrm{LPL}$ or equivalently by solving $\boldsymbol{e}_{K}(\bX;\btheta_K)=0$.
The maximization of~\eqref{lpl} is  fast and stable, mainly since $\mathrm{LPL}(\bX;\btheta_K)$ turns out to be a concave function of $\btheta_K $\cite[see][]{jensen1991pseudolikelihood}. 
We can even say more by slightly anticipating on Remark~\ref{remMat}: for $n$  large enough (that is when the amount of data increases) the sensitivity matrix derived from the score $\boldsymbol{e}_{K}(\bX;\btheta_K)$ is positive definite ensuring that the maximum is indeed unique. From a computational point of view, \cite{baddeley2000practical} show that~\eqref{lpl} can be  approximated by a quasi-Poisson regression. The methodology can be quickly and efficiently implemented using standard software for generalized linear models. We did not consider other composite likelihoods such as the logistic one \cite{baddeley2014logistic} or other estimating functions. Indeed, on the one hand, the logistic regression likelihood can be viewed as an approximation of the pseudolikelihood which can solve some numerical problems of the pseudolikelihood when discretizing the integral term in~\eqref{lpl} is problematic. We did not encounter such an issue in the problem considered in this paper.
On the other hand, \cite{coeurjolly2016towards} investigated the question of optimal estimating functions for parametric stationary  Gibbs point processes and showed that the procedure is more computationally intensive, less stable numerically than the pseudolikelihood for a little gain in terms of mean squared error. 

Let us add two more comments. First, the estimation procedure does not produce an explicit estimator of the coefficients $\theta_k$ as in \cite{jalilian2019orthogonal} since $\mathrm{LPL}(\bX;\btheta_K)$ is a nonlinear function of $\btheta_K$. Second, from GNZ formula~(\ref{gnz}), $\boldsymbol{e}_{K}(\bX;\btheta_K)$ is not necessarily an unbiased estimating function since its expectation does not vanish under the true underlying Gibbs model.  This would be true if $K=\infty$ or if $g$ has a finite series expression with $K^{\star}$ orthonormal functions such that $K^{\star} \leq K$.

An estimator of $g$ is then obtained by replacing $\theta_k$  by $\hat{\theta}_k$ in~\eqref{orth} and by truncating the infinite sum up to some integer $K \geq 1$
\begin{align} \label{orthest}
\hat{g}(r;K) =  \sum_{k=1}^K \hat{\theta}_k \varphi_k(r - \delta).
\end{align} 
A more general form could be given by 
\begin{align} \label{orthinfty}
\hat{g}(r;b) = \sum_{k=1}^{\infty} \tau_k \hat{\theta}_k \varphi_k(r - \delta)
\end{align}
where $\tau=\{\tau_k\}_{k=1}^{\infty}$ is a smoothing/truncation scheme, see for e.g.~\cite{wahba1981data}. To ease the presentation we do not consider this possibility in this paper and focus on the simplest one that is $\tau_k = \mathbf{1}(k \leq K)$. Finally, an estimator of the pairwise interaction function $\phi$ is given by 
\begin{align} \label{phiest}
\hat{\phi}(r;K) = \exp \left( \sum_{k=1}^{K} \hat{\theta}_k \varphi_k(r - \delta) \right), \quad r\in[\delta,R+\delta].
\end{align}
Using the orthonormality of the basis function, the integrated squared error (ISE) takes the simple form
\begin{align} \label{eq:ise}
\mathrm{ISE}(\hat{g}(\cdot \,;K)) &= \int_{\delta}^{R+\delta}  \{ \hat{g}(r;K) - g(r) \big \}^2 w(r-\delta) \mathrm{d}r 
=  \sum_{k=1}^{K} (\hat \theta_k-\theta_k)^2 + \sum_{k>K} \theta_k^2.
\end{align}



\section{Asymptotic properties} \label{sec:asymp}
Let $K_n$ be a sequence of integers such that $K_n \, \to \, \infty$ as $n \, \to \, \infty$. We denote by $\btheta^{\star}_{K_n}=(\theta_k^{\star})_{k=0,\ldots,K_n}$ the $(K_n+1)$-dimensional vector of true coefficients, i.e. the first $K_n+1$ elements of $\btheta^{\star}=(\theta_k^{\star})_{k=0,\ldots,\infty}$. Generally, we let $\btheta=(\theta_k)_{k=0,\ldots,\infty}$ and $\btheta_{K_n}=(\theta_k)_{k=0,\ldots,K_n}$. When there is no ambiguity, $\btheta_{K_n}$ will also denote the sequence of $\mathbb{R}^{\mathbb{N}}$ such that $\theta_k=0$ for all $k \geq K_n+1$.
 Let
\begin{align} \label{lpln}
\mathrm{LPL}_n(\bX;\btheta_{K_n}) = \sum_{u \in \bX \cap W_n} \log \lambda (u,\bX;\btheta_{K_n}) - \int_{W_n} \lambda (u,\bX;\btheta_{K_n}) \mathrm{d}u  
\end{align}
be the truncated log-pseudolikelihood function of $\bX$ observed on $W_n$. Its derivative with respect to $\btheta_{K_n}$ equals
\begin{align} \label{en}
\boldsymbol{e}_{K_n}(\bX;\btheta_{K_n})=  \sum_{u \in \bX \cap W_n} \bS_{K_n}(u, \bX \setminus u) -  \int_{W_n} \bS_{K_n}(u, \bX)\lambda (u,\bX;\btheta_{K_n}) \mathrm{d}u. 
\end{align}
We let $\hat{\btheta}_{K_n} = \underset{\btheta_{K_n} \in \mathcal{S}_m}{\text{argmax}} \, \mathrm{LPL}_n(\bX;\btheta_{K_n})$ where the set $\mathcal{S}_m$ will be made more precise in \ref{C:anKn}.
The estimator $\hat{g}_n$ of $g$ obtained from $\bX$ observed on $W_n$ is given, for any $r\in [\delta,R+\delta]$, by
\begin{align} \label{phin}
\hat{g}_n(r;K_n) = \sum_{k=1}^{K_n}  \hat{\theta}_{k,n} \, \varphi_k(r - \delta).
\end{align}
We define the integrated squared error (ISE) of $\hat{g}_n$ by
\begin{align} \label{ise}
 \mathrm{ISE}(\hat{g}_n(\cdot\,;K_n)) = \left \| \hat{g}_n - g \right \|^2_{\text{L}^2_w([\delta,R+\delta])}  = \int_{0}^{R}  \big \{ \hat{g}_n(r+\delta;K_n) - g(r+\delta) \big \}^2 w(r)\mathrm{d}r.
\end{align}
Using these notation for instance $\boldsymbol{e}_{K_n}(\bX;\btheta^{\star})$ corresponds to (\ref{en}) where $\btheta_{K_n}$ is replaced by $\btheta^{\star}$. And it is to be pointed out that $\boldsymbol{e}_{K_n}(\bX;\btheta^{\star})$ is an innovation type statistic~\cite[see][]{baddeley2005residual,coeurjolly2013fast} and in particular $ \EE \, [ \boldsymbol{e}_{K_n}(\bX;\btheta^{\star}) ] = 0$ from the GNZ formula.

\subsection{Conditions} \label{subsec:cond}
For $\by, \, \bz \in \ell^2(\mathbb{N})$, we let $\by^\top \bz = \sum_{k\ge 0} y_k \, z_k$. Let $\btheta \in \mathbb{R}^{\mathbb{N}}$.
Denote for $\gamma\ge 0$ and $\by\in \ell^2(\mathbb{N})$ by $\|\by\|_\gamma = |y_0|+\sum_{k\ge 1}|k|^\gamma |y_k|$ and we let $\mathcal S_\gamma:= \{\by \in \ell^2(\mathbb{N}: \|\by\|_\gamma<\infty\}$.
We define the following stochastic infinite matrices for $\btheta \in {\mathbb R}^{\mathbb N}$ 
\begin{align*} 
\mathbf{A} (\bX;\btheta)=&{\int_{W_n} \bS(u, \bX) \bS(u, \bX)^\top \lambda (u,\bX;\btheta) \mathrm{d}u}, \\
\mathbf{B}(\bX;\btheta)=&  {\int_{W_n}   \int_{W_n} \hspace{-.25cm}  \bS(u,\bX)\bS(v,\bX)^\top \lambda(u,\bX;\btheta)\lambda(v,\bX;\btheta)  
\left(
1-\phi_{\btheta}( \|v-u\|)
\right)
\mathrm{d}v \mathrm{d}u} \\
& + {\int_{W_n} \int_{W_n}  \widetilde{\boldsymbol{\varphi}}(\|v-u\|-\delta) \widetilde{\boldsymbol{\varphi}}(\|v-u\|-\delta)^\top  \lambda(\{u,v\},\bX;\btheta)\mathrm{d}v \mathrm{d}u} 
\end{align*}
where $\bS(u, \bX)=(1,S_1(u,\bX),S_2(u,\bX),\dots)^\top$, where for $r \in (\delta,R+\delta)$, 
$\log \phi_{\btheta}(r) = \btheta^\top_{-0}  \boldsymbol{\varphi}(r - \delta)$ with $\btheta_{-0}=(\theta_1,\dots)^\top$, $\boldsymbol{\varphi}(r)=(\varphi_1(r),\dots)^\top$ and finally where $\widetilde{\boldsymbol{\varphi}}(r-\delta)=(0,\varphi_1(r),\dots)^\top$. Note that in particular  $\phi(\cdot)=\phi_{\btheta^\star}(\cdot)$.

For a stochastic matrix $\mathbf{M} (\bX;\btheta)$ we use the notation $\mathbf{M}_{K_n} (\bX;\btheta) = \left(\mathbf{M} (\bX;\btheta)\right)_{i,j=0,\dots,K_n}$, $\mathbf{M} (\btheta) = \EE \left[ \mathbf{M} (\bX;\btheta) \right]$, $\mathbf{M}_{K_n} (\btheta) = \left(\mathbf{M} (\btheta)\right)_{i,j=0,\dots,K_n}$ and $\widebar{\mathbf{M}} (\btheta)=|W_n|^{-1} \mathbf{M} (\btheta)$. 

We may check using the GNZ formula and \cite[Lemma~3.1]{coeurjolly2013fast} that first $\mathbf{A}_{K_n} (\bX;\btheta)  = - \frac{\mathrm d}{\mathrm d \btheta} \; \boldsymbol{e}_{K_n}(\bX;\btheta)$, whence $\mathbf{A}_{K_n} (\btheta) = \EE [ \mathbf{A}_{K_n} (\bX;\btheta) ]$ is the sensitivity matrix and second that $\Var \left[ \boldsymbol{e}_{K_n}(\bX;\btheta^{\star}) \right]  = \mathbf{A}_{K_n} (\btheta^{\star}) + \mathbf{B}_{K_n} (\btheta^{\star})$.

We are now able to describe the general assumptions required by our results.

\begin{enumerate}[($\mathcal C$.1)]
\item  $(W_n)_{n \geq 1}$ is an increasing sequence of regular convex compact sets, such that  $W_n \to \mathbb{R}^d$ as $n \to \infty$ and there exists $p\ge 1$ such that $\sum_{n\ge 1} \left(\frac{|W_n\setminus W_n^-|}{|W_n|}\right)^p<\infty$ where $W_n^-=W_n\ominus(R+\delta)$ is the domain $W_n$ eroded by $R+\delta$. \label{C:Wn}
\item We assume that the Papangelou conditional intensity function has the log-linear specification given by~\eqref{condint} and that $g\in L^2_w([\delta,R+\delta])$ satisfies the assumptions of Corollary~\ref{cor:existence}.  \label{C:model}
\item We assume (i) $K_n \to \infty$ as $n \to \infty$; (ii) there exists $m \ge 0$ such that $\|\varphi_k\|_\infty=O(k^m)$;  (iii) $\sup_{r\in [0,R]}(r+\delta)^{d-1}/w(r)<\infty$; (iv) for any $n\ge 1$, $\hat{\btheta}_{K_n}, \, \btheta^{\star} \in \mathcal{S}_m$.
\label{C:anKn}
\item We assume  (i) $\sup_{r\in[0,R]} w(r)<\infty$; (ii)  $\psi = 1 + |\sigma_{d}| \, e^{\theta_0^\star} {\int_{\delta}^{R+\delta} (1 - \phi(r) ) \, r^{d-1} \, \mathrm{d}r }>0$ where $|\sigma_{d}|$ is the surface area of the unit sphere in $\mathbb{R}^d$; (iii) the set $[0,R]\setminus \Delta$ has zero Lebesgue measure where $\Delta=\{r\in [0,R]: (r+\delta)^{d-1}\phi(r+\delta) / w(r)>0 \}$.   \label{C:psi}
\end{enumerate}

Using the finite range property and the stationarity of $\bX$, we refer the reader to \cite{coeurjolly2013fast} where it is proved under conditions \ref{C:Wn}-\ref{C:model} that 
\begin{align*}
 \widebar{\mathbf{A}}(\btheta)& = \EE \left[ \bS(\mathrm{0}, \bX) \bS(\mathrm{0}, \bX)^\top \lambda (\mathrm{0},\bX;\btheta) \right] \\
\widebar{\mathbf{B}}(\btheta) & = \EE \left[ {\int_{\mathcal B(0)} \hspace{-.25cm}  \bS(\mathrm{0},\bX)\bS(v,\bX)^\top \lambda(\mathrm{0},\bX;\btheta)\lambda(v,\bX;\btheta)  \left(
1-\phi_{\btheta}( \|v-u\|)
\right)} \mathrm{d}v  \right] \\
+ & \,  \EE \left[ {\int_{\mathcal B(0)}  \widetilde{\boldsymbol{\varphi}}(\|v\|-\delta) \widetilde{\boldsymbol{\varphi}}(\|v\|-\delta)^\top \lambda(\{\mathrm{0},v\},\bX;\btheta)\mathrm{d}v}\right]
\end{align*}
where to ease notation $\mathcal B(u)$ stands for the ball (or annulus depending on the value of $\delta$)
$\left \{ v \in \mathbb{R}^d: \delta \leq \| v - u \| \leq R+\delta \right \}$.  
To get the exact expression of $\widebar{\mathbf{B}}(\btheta)$, we actually require that $\mathcal B(0) \subset W_n$, which is always true for $n$ large enough.

Condition~\ref{C:Wn} specifies our asymptotic framework. If for instance, \linebreak$W_n=[-(n/2)^{1/d},(n/2)^{1/d}]^d$,  then $|W_n|=n$ and $|W_n\setminus W_n^-|=O(n^{1-1/d})$ so the last part of \ref{C:Wn} is satisfied by choosing any $p>d$. Condition \ref{C:model} in particular ensures the existence of a stationary Gibbs measure.

Let us discuss Condition \ref{C:anKn}. The requirement that $K_n \to \infty$ and the existence of $m$ are natural. Note that $m$ is equal to 0 for the cosine basis, to $1/2$  for the Haar one and to $(d-1)/2$ (when $d\ge 2$) for the Fourier-Bessel one \cite[Section S5 in the supplementary material]{jalilian2019orthogonal}). (iii) is satisfied for the Haar and cosine basis for every $\delta>0$ since $w(r)=1$  and for the Fourier-Bessel one ($w(r)=r^{d-1}$) if $\delta=0$. This assumption is actually used to show that $\int_{0}^R \varphi_k(r)^2\mathrm d r=O(1)$ (used in Lemma~\ref{lem:Esk}). If (iii) is not valid (e.g. the Fourier-Bessel basis with $\delta>0$), we can only show that $\int_{0}^r \varphi_k(r)^2\mathrm d r=O(k^m)$. Theorem~\ref{THM:CONST} remains true but with different rates of convergence (see Section~\ref{sec:FB} in the Supplementary Material for a treatment of this case). By assuming \ref{C:anKn}, we implicitly assume that $g$ is lower and upper-bounded, which means that $\phi(r)$ is positive and upper-bounded for any $r\in [\delta,R+\delta]$ (this is the reason why we slightly modify the Diggle-Gratton pairwise interaction function, for which $\delta=0$, in such a way that $\phi(0)=0.05>0$).
To continue the discussion of Condition \ref{C:anKn}, let us define the sequence
\begin{equation} \label{eq:aK}
a_{K_n} = \sum_{k>K_n} k^m |\theta^\star_k|.
\end{equation}
As already noticed by \cite{hall1986rate} in the simpler context of orthogonal series density estimator, the rate at which the coefficients $\theta_k^\star$ converge to 0 determine the rate of convergence of density estimators. A similar behaviour is observed in our context since condition \ref{C:anKn} implies that $a_{K_n}\to 0$ as $n\to \infty$. Even more, we will show in Theorem~\ref{THM:CONST} that the consistency of $\hat \btheta_{K_n}$, the convergence of the ISE, etc strongly depend on the rate of convergence of $a_{K_n}$ to 0. We continue this discussion after the presentation of Theorem~\ref{THM:CONST}.


As seen in Lemma~\ref{lemma2} (ii), Condition \ref{C:psi} is used to ensure that the infinite matrix $\widebar{\mathbf A}(\btheta^\star) + \widebar{\mathbf B}(\btheta^\star)$ is positive definite, that is, from Lemma~\ref{lemma1}, to ensure that the normalized variance  $|W_n|^{-1}\Var[\mathbf{e}_{K_n}(\bX;\btheta^\star)] = |W_n|^{-1}( \mathbf{A}_{K_n}(\bX;\btheta^\star) + \mathbf{B}_{K_n}(\bX;\btheta^\star) )$ is positive definite for $n$ sufficiently large. The condition on $\psi$ depends only on the model. It is trivially satisfied when $\phi\le 1$, i.e. for purely repulsive PIPP (in particular for PIF1 and PIF2). We have  checked numerically  that this condition is satisfied for the two other examples PIF3 and PIF4 where $\phi$ can exceed 1. The condition related to $\Delta$ in \ref{C:psi} depends on the basis function. Remember that $\btheta^\star\in \mathcal S_m$ necessarily implies that $\phi(\cdot)>0$. So for the Fourier-Bessel basis for which $w(r)=r^{d-1}$, the condition is always satisfied. For the two other considered bases, for which $w(r)=1$, the condition is  obvious when $\delta>0$ and is also easily satisfied when $\delta=0$ since $r^{d-1}\phi(r)$ vanishes only at $r=0$.

\begin{remark} \label{remMat} As partly mentioned in the previous paragraph, Lemma~\ref{lemma2} (i)-(ii) ensures that the infinite matrices $\widebar{\mathbf{A}}(\btheta)$ and $\widebar{\mathbf{A}}(\btheta^\star)+\widebar{\mathbf{B}}(\btheta^\star)$ are positive definite. Those matrices play an important role in our asymptotic results as they correspond to some limit of the normalized sensitivity matrix and variance covariance matrix of the estimating equation $\mathbf{e}_{K_n}(\bX;\cdot)$. Even if this looks hidden in the proof of Lemma~\ref{lemma2}, assuming that the basis of functions $(\varphi_k)_{k\ge 1}$is orthonormal is crucial here.
\end{remark}
  
\subsection{Main results} \label{sec:result}

For a random vector $\bX_n$ and a sequence of real numbers $x_n$, the notation $\bX_n = O_{\mathrm P} (x_n)$ or $\bX_n = o_{\mathrm P} (x_n)$ means that $\| \bX_n \| =  O_{\mathrm P} (x_n)$ and $\| \bX_n \| =  o_{\mathrm P} (x_n)$. In the same way, the notation $\bV_n = O(x_n)$ and $\bM_n= O(x_n)$ mean that $\| \bV_n \| = O(x_n)$ and $\| \bM_n \| = O(x_n)$ for a vector $\bV_n$ and a squared matrix $\bM_n$. 
 
As detailed in the previous section, the nonparametric estimate of $g$ is based on the estimation of $\btheta_{K_n}^\star$, the first $K_n$ coefficients of its expansion on the basis of functions $\{\varphi_k\}_k$, which is done by minimizing the estimating equation~\eqref{en}. Proposition~\ref{PROP:BIAS} presented in the Supplementary Material provides conditions on the sequences $a_{K_n}$, $K_n$ and $|W_n|$ to get a control in probability of the estimating equation~\eqref{en} evaluated at $\btheta^\star_{K_n}$.

The estimator  of $\btheta_{K_n}^\star$, obtained as a minimization, is not explicit. As such, its bias and variance are hard to determine. Its rate of convergence in probability can nevertheless be derived from the one of the estimating equation~\eqref{en}. Then, the integrated squared error and the asymptotic normality of $\hat{g}_n(r;K_n)$ can be deduced. These results are summarized by Theorem~\ref{THM:CONST}. The pointwise convergence of $\hat{g}_n(r;K_n)$ and a multivariate extension of Theorem~\ref{THM:CONST}(ii) are provided in the Supplementary Material.
   
\begin{theorem} 
\label{THM:CONST}
Assume the conditions \ref{C:Wn}-\ref{C:anKn} hold.   
\begin{enumerate}[(i)]
\item As $n\to \infty$, asume that $a_{K_n} K_n^{1/2}\to 0$ and that $K_n^{m \vee 1}/|W_n| \to 0$. Then, there exists a maximizer $\hat{\btheta}_{K_n} \in \mathcal{S}_m$ of  $\, \mathrm{LPL}_n(\bX;\btheta_{K_n})$ such that
\begin{align} \label{eq:cons}
 { \| {\hat{\btheta}}_{K_n} -\btheta_{K_n}^{\star}\|} =  O_\mathrm{P} \left( {\frac{K_n^{(m\vee 1)/2}}{|W_n|^{1/2}}} + a_{K_n} {K_n}^{1/2} \right) 
\end{align}
and 
\[ \mathrm{ISE}(\hat{g}_n(\cdot\,;K_n)) =   O_\mathrm{P} \left( {\frac{K_n^{m\vee 1}}{|W_n|}} + a_{K_n}^2 {K_n} \right). 
\]
\item As $n\to \infty$, asume that $a_{K_n} K_n^{1/2}\to 0$, $a_{K_n}^2 K_n^{m+3}|W_n|^{-1/2}$ and that $K_n^{2/3(m+m\vee 1+2)}|W_n|^{-1} \to 0$. Then, for $r \in (\delta,R+\delta)$ 
\begin{align*}
s_n^{-1}(r) \left \{ \hat{g}_n(r;K_n) - g(r) \right \} \xrightarrow{d} \mathcal{N}(0,1)
\end{align*}
where 
\begin{align}
s_n^{2}(r) = & \widetilde{\boldsymbol{\varphi}}_{K_n}^\top(r - \delta) \,  \boldsymbol \Pi_{K_n} \, \widetilde{\boldsymbol{\varphi}}_{K_n}(r - \delta), \nonumber \\
\boldsymbol \Pi_{K_n} = & \mathbf{A}_{K_n}^{-1}(\btheta_{K_n}^\star) \,  \left \{ \mathbf{A}_{K_n}(\btheta^{\star}) + \mathbf{B}_{K_n}(\btheta^{\star}) \right \} \, \mathbf{A}_{K_n}^{-1}(\btheta_{K_n}^{\star}).\label{tauxn}
\end{align}
\end{enumerate}
\end{theorem}
Theorem~\ref{THM:CONST} sheds the light on the trade-off which must be found on the sequences $K_n$ and $|W_n|$ and the assumption on the underlying model through the rate of convergence of $a_{K_n}\to 0$ which implicitly imposes some conditions on $g$. However, the smoothness of $g$ cannot be directly related to the decrease of the partial sum $\sum_{k>K_n} k^m|\theta_k^\star|$. To discuss more this, let us assume that $\theta^\star_k= O(k^{-1-\varepsilon-\gamma})$ for some $\varepsilon>0$ and $\gamma>m$ which implies that $a_{K_n}= o(K_n^{m-\gamma})$. Let us also introduce the space $\mathcal H^s:=\{f\in L^2_w([\delta,R+\delta]): \sum_{k} (1+k)^{2s} |{\theta^\star_k}|^2 <\infty\}$, which for the the cosine basis, corresponds to the Sobolev space with regularity $s$. Note that $\theta_k^\star=O(k^{-1-\varepsilon-\gamma})$ implies that $g\in \mathcal H^{\gamma+1/2}$ (actually to any $\mathcal H^s$ for any $s\le \gamma+1/2$). Using this ``regularity'' space, assuming that $m\le 1$ and focusing on the convergence of the ISE to simplify the discussion, we see that the ISE tends to 0 if $\gamma \ge m+1/2$ which would imply that $g\in \mathcal H^{m+1}$. On the much simpler problem of orthogonal density estimators with iid random variables, \cite{hall1986rate} for instance, shows that the MISE is consistent as soon as the true density is squared integrable. We attribute our more restrictive assumptions to the strong nonlinearity of the current problem (we cannot estimate $\theta_k^\star$ without estimating all other coefficients) and to the complex  dependence characteristic of the Gibbs point process.\\ 
To continue the discussion, if one assumes in addition that $K_n=|W_n|^\alpha$ (with $\alpha>0$), the ISE tends to zero if $0<\alpha<1$ with optimal rate proportional to $|W_n|^{-r}$ with $ r=(\gamma-m+1/2)/(\gamma-m)$.


As already mentioned $\mathbf{A}_{K_n}(\btheta_{K_n}^{\star})$ and $\mathbf{A}_{K_n}(\btheta^{\star}) + \mathbf{B}_{K_n}(\btheta^{\star})$ can be seen as the sensitivity matrix and the variance covariance matrix of the score of the truncated pseudolikelihood. By using ergodic theorem \cite[see][]{nguyen1979ergodic} and the GNZ formula, \cite{coeurjolly2013fast} proposed fast and consistent estimates of such matrices for general Gibbs point processes. Applied to pairwise interaction point processes, the $(j,k)$-th elements $a_{jk} = |W_n|^{-1}(\mathbf{A}_{K_n}(\btheta_{K_n}^{\star}))_{jk}$ (or $|W_n|^{-1}(\mathbf{A}_{K_n}(\btheta^{\star}))_{jk}$) and $b_{jk}=|W_n|^{-1}(\mathbf{B}_{K_n}(\btheta^{\star}))_{jk}$ can respectively be estimated by
\begin{align*}
\hat{a}_{jk} = & |W_n|^{-1} \sum_{u \in \bX \cap W_n} S_j(u,\bX \setminus u) S_k(u,\bX \setminus u) \\
\hat{b}_{jk} = & |W_n|^{-1} {\sum_{u,v \in \bX \cap W_n}}^{\!\!\!\!\!\!\!\!\!\!\!\!\prime} \quad \Bigg \{ S_j(u,\bX \setminus \{u,v\}) S_k(v,\bX \setminus \{u,v\}) \left(\frac{1}{\phi_{\hat{\btheta}_{K_n}}(\|v-u\|)}-1 \right)  \\ 
& + \varphi_j(\|v-u\|-\delta) \varphi_k(\|v-u\|-\delta) \Bigg \}
\end{align*}
where to ease notation, the symbol $\sum^\prime$ means that the sum acts on pairs of distinct points such that $\delta \leq \| v - u\| \leq R+\delta$. These estimates depend only on data points and do not require any integrals discretization which explains their quick implementation within the function \texttt{vcov.ppm} from the \texttt{spatstat R} package. Since $\widetilde{\boldsymbol{\varphi}}_{K_n}(r - \delta)$ depends only on the basis system, estimates of $\boldsymbol \Pi_{K_n}$ and finally $s_n(r)$ follow easily.

By Theorem~\ref{THM:CONST} (ii), an approximate $100(1-\alpha)\%$ confidence interval for $g(r)$ is given by $\hat{g}_n(r;K_n) \mp z_{\alpha/2} \hat{s}_n(r)$ where $r \in (\delta,R+\delta)$ and $z_{\alpha/2}$ is the quantile of order $(1-\alpha/2)$ of a standard normal distribution. A pointwise confidence interval for $\phi(r)$ is easily deduced.

\section{Implementation and simulation study} \label{sec:sim}

\subsection{Full description of the methodology} \label{subsec:imp} ${ }$ \\

Given an orthonormal basis function, $\delta$, $R$ and a number of polynomials $K$, the problem of estimating $g$ (and thus $\phi$) is easily implemented since we bring back the estimation of the coefficients $\theta_0,\dots,\theta_K$ to the estimation of an exponential family Gibbs model. The estimation is done using the pseudolikelihood which is implemented in the \texttt{ppm} function of the \texttt{spatstat} \texttt{R} package. This function takes advantage of the analogy between the pseudolikelihood and the quasi-Poisson regression \cite{baddeley2000practical}.

We consider the estimation of $g$ on $[\delta,R+\delta]$ using an expansion over a basis of functions on $L^2_w([0,R])$, with $\delta \ge 0$ and $R<\infty$. We remind that to encompass cases (a)-(b) of Corollary~\ref{cor:existence} (see Remarks after Corollary~\ref{cor:existence}) we abuse notation naming $\delta\ge 0$ as the hard-core parameter: $\delta=0$ means no hard-core while a positive value means that two points at distance smaller than $\delta$ do not interact. Whatever the value of $\delta$, the parameter $R+\delta$ is related to the range of the GPP. Two points at distance larger than $R+\delta$ interact in the same way than under a homogenous Poisson point process.  If these parameters are unknown, which is most often the case for real datasets, we either assume we have at our disposal a lower-bound and upper-bound for these irregular parameters or we use a standard approach to estimate the hyperparameters $\delta$ and $R$. In this situation, we estimate $\delta$ by $\hat \delta= n/(n+1) d_{\min}$, where $d_{\min}$ is the observed smallest distance between two points. This estimate is the  \texttt{spatstat} version (known to be less biased) of the estimate proposed by~\cite{ripley1978quick}. To estimate $R+\delta$, we can use the profile pseudolikelihood (by setting first a reasonable flexible model) or take the value of $r$ such that the maximal deviation between the empirical Ripley's $K$-function and its expected value under the homogeneous Poisson process is observed. Both approaches are described in \cite[p.~518]{baddeley2015spatial}.  In the simulation study, we have first assumed that those two parameters are known and given by their values. We also show on a smaller example that misspecifying or estimating $\delta$ and $R$  does not deteriorate  the efficiency of the proposed methodology.

Let us now consider the crucial problem of selecting the number of polynomials $K$.
In the context of nonparametric estimation of the pair correlation function, \cite{jalilian2019orthogonal} and \cite{coeurjolly2019second} suggest to select it by minimizing an estimate of the MISE or by maximizing a composite likelihood cross-validation criterion respectively. We were unable to mimic these approaches in the context of the present paper, in particular because we do not have an explicit expression (nor asymptotic equivalent) of the bias and the variance of $\hat{g}(r;K)$. Instead, we propose to adapt an information criterion approach. Information criteria for composite likelihood were first introduced by~\cite{gao2010composite} and were then extended to spatial point processes by~\cite[][]{daniel2018penalized,choiruddin2021information,bainference}. Applied to the truncated pseudolikelihood, this criterion writes
\begin{equation}\label{eq:cAIC}
\mathrm{cAIC}(K,\bX) = - 2 \, \mathrm{LPL}(\bX;\hat{\btheta}_K) + 2 \, \text{trace} (\hat{\mathbf A}_K \hat{\boldsymbol \Pi}_K)
\end{equation}
where $\hat{\mathbf A}_K$ and $\hat{\boldsymbol{\Pi}}_K$ are estimates of $\mathbf{A}_K({\boldsymbol \theta}_{K}^\star)$ and ${\boldsymbol{\Pi}}_K$ given by~\eqref{tauxn}. As already mentioned in the previous section, both estimates  and thus the trace are easily implemented in particular in \texttt{R} and can be obtained using the function \texttt{vcov.ppm} from the \texttt{spatstat} package.

Note that the trace term in~\eqref{eq:cAIC} comes from the fact that the pseudolikelihood is only a composite likelihood and not a likelihood. If we were dealing with an inhomogeneous Poisson point processes ($\phi=1$ with $\beta=\beta(u)$ modelled through say $p$ covariates for instance), the pseudolikelihood would be a likelihood, the trace term $p$ and the criterion $\mathrm{cAIC}$  the standard Akaike's information criterion.

Finally, we suggest to select $K$ as 
\[
    \hat K^{\star}_{\mathrm{cAIC}} = \mathop{\mathrm{argmin}}_{1\le K\le K_{\max}} \mathrm{cAIC}(K,\bX)
\]
yielding a data-driven procedure. Note that we `view' $\hat K^{\star}_{\mathrm{cAIC}}$ as an estimate of $K^{\star}_{\mathrm{cAIC}} = \mathrm{argmin}_K \EE[\mathrm{cAIC}(K,\bX)]$ and we choose $K_{\max} = 15$ in the simulation study and the real datasets analysis. The whole procedure is summarized in Algorithm~\ref{alg:algorithm1}.

\begin{minipage}{\textwidth} 
\renewcommand*\footnoterule{}
\begin{algorithm}[H] 
\caption{Data-driven estimation of the pairwise interaction function $\phi(r)$ via series expansion (when $\delta$ and $R$ are known)}
\label{alg:algorithm1}
\begin{algorithmic}[1]
\Statex
     \State{Assume $\bx$ is a realization of $\text{Gibbs}(\phi(r),W)$ and consider the series expansion of $g=\log \phi$ on the basis of orthonormal polynomials $(\varphi_k)_{k\ge 1}$ }
     \State{Let $K_{\max} \ge 1$}
     \For{$k \gets 1$ to $K_{\max}$}
        \State {$S_k(u,\bx) \gets \sum_{v \in \bx} \varphi_k (\Vert v - u \Vert - \delta), u \in W$}
    \EndFor
    \For{$K \gets 1$ to $K_{\max}$ }
    \State {Form  $\bS_K(u,\bx)\gets (S_1(u,\bx),\dots,S_K(u,\bx))$} 
    \State {Compute $\hat{\btheta}_{K}$ (through the truncated pseudolikelihood) \texttt{fit} $\gets$ \texttt{ppm(}$\bX \sim 1$, interaction = $\bS_K$\texttt{,...)}}
    \State{Compute $\mathrm{LPL}(\bx;\hat{\btheta}_K) \gets$ \texttt{logLik.ppm(fit)} }
    \State{Compute $\hat{\mathbf A}_K \gets $ \texttt{vcov(fit,hessian=TRUE)}}
    \State{Compute $\hat{\boldsymbol \Pi}_K \gets$ \texttt{vcov(fit)}}
    \State{Compute $\mathrm{cAIC}(K,\bx)$ }
    \EndFor
    \State{Compute  $\hat K^\star = \hat K^{\star}_{\mathrm{cAIC}} \gets \mathrm{argmin}_{1\le K\le K_{\max}} \mathrm{cAIC}(K,\bx)$}
    \State{Define $\hat{\boldsymbol{\theta}}_{\hat K^{\star}}=(\hat \theta_1,\dots,\hat \theta_{\hat K^{\star}} ) $}
    \State {
    {Define $\hat{g}(r;\hat K^\star) \gets \sum_{k=1}^{K}\hat{\theta}_{k}\varphi_k(r - \delta) $} ; 
    {$\hat{\phi}(r;\hat K^\star) \gets \exp \hat{g}(r;\hat K^\star)$}
    }
\end{algorithmic}
\end{algorithm}
\end{minipage}

\subsection{Simulation set-up} \label{sim:ortho}

We investigate the performance of the orthogonal series estimators  through simulations of GPPs with pairwise interaction functions $\phi_k$ (referred to as PIF$k$ model) for $k=1,\dots,4$ defined in Section~\ref{sec:gpps} on the spatial domains $W_\ell=[0,\ell]^2$ for $\ell=1,2,3$. More precisely, we generate $500$ simulations from GPPs with true conditional intensity $\lambda(u,\bx) = \beta \prod_{v \in \bx} \phi_k (\Vert v - u \Vert)$ for $k=1,\dots,4$. It does make sense to set $\beta$ so that the model has a prescribed value of the average number of points per unit volume. The intensity of a GPP is not explicit in terms of its parameters but recently \cite{coeurjolly2018intensity} proposed an approximation using determinantal point processes. This approximation was shown to outperform the (best alternative) Poisson-saddlepoint approximation  proposed by~\cite{baddeley2012fast}, even for clustered patterns. We use this approximation to set the activity parameter such that the mean number of points in $W_1$ is $200$, for the four models. The values of $\beta$ and $\theta_0^\star=\log \beta$ as well as the average number of points based on the 500 simulations is summarized in Table~\ref{tab:nb}.

\begin{table}[!h]
\begin{center}
\begin{tabular}{rrrrrr}
\hline
& \multicolumn{2}{c}{Value of} & \multicolumn{3}{c}{Average number of points in}\\
Model & $\beta$ & $\theta_0^\star$ &$W_1=[0,1]^2$& $W_2=[0,2]^2$ & $W_3=[0,3]^2$\\
 \hline
PIF1 &1827 & 7.51  &183& 731& 1643\\
PIF2 &1240 & 7.12  &190& 759& 1708\\
PIF3 &119  & 4.78  &182&717&1562\\
PIF4 &1316 & 7.18 &217&868&1951 \\
\hline
\end{tabular}	
\end{center}
\caption{\label{tab:nb} Values of the activity parameter $\beta$ and $\theta_0^\star=\log \beta$ such that the approximation of the intensity is 200 on $[0,1]^2$, and average number of points on $W_\ell=[0,\ell]^2$ based on 500 simulations of PIF$k$ for $k=1,\dots,4$.}
\end{table}

For each simulated pattern, we first estimate $\btheta_K$ for $K=1,2,\ldots,K_{\max}$, with $K_{\max}=15$ using the \textsf{ppm} function of the \textsf{spatstat} package with the border correction. Note that this correction consists in replacing the observation domain in the estimation procedure by $W_\ell \ominus (R+\delta)$, that is $W_\ell$ eroded by $R+\delta$, for $\ell=1,2,3$. The number of dummy points $\textsf{nd}^2$, used to discretize the integral in~\eqref{lpl}, is set to $64n$, where $n$ is the observed number of points. Then, we apply Algorithm~\ref{alg:algorithm1} to select the value of $K$ and estimate $g_k$ and $\phi_k$ (for $k=1,\dots,4$) on $[\delta,R+\delta]$. We remind that this interval is $[0,0.08]$ for PIF1 and PIF2 and $[0.01,0.08]$ for PIF3 and PIF4 (which have a hard core $\delta=0.01$). We proceed in this way for the three basis of functions described in Section~\ref{subsec:orth}, that is the cosine, the Fourier-Bessel and the Haar bases. Section~\ref{sim:eip} investigates a simulation where $\delta$ and $R$ are unknown.

For each model and each observation window, we store the 500 values of $\hat K^{\star}_{\mathrm{cAIC}}$. We also evaluate via Monte-Carlo $K^\star_{\mathrm{MISE}}$ which minimizes the empirical MISE given by the expected value of~\eqref{eq:ise} and  $K^\star_{\mathrm{cAIC}}$ which minimizes the Monte-Carlo approximation of ${\EE} [ \mathrm{cAIC}(K,\bX) ]$. In addition, we compute the following relative losses
\begin{align*}
\widehat{\mathrm{Loss}}_{\mathrm{MISE}}&=500^{-1}\sum_{i} 
\frac{\int \{\hat g^{(i)}(r;K^{\star}_{\mathrm{MISE}})-g(r)\}^2 \mathrm{d} r}{\int 
g(r)^2 \mathrm{d} r} \\	
\widehat{\mathrm{Loss}}_{\mathrm{cAIC}}&=500^{-1}\sum_{i} 
\frac{\int \{\hat g^{(i)}(r;\hat K^{\star,(i)}_{\mathrm{cAIC}})-g(r)\}^2 \mathrm{d} r }{\int 
g(r)^2 \mathrm{d} r} 
\end{align*} 
which are Monte-Carlo estimates of the  relative $L^2$-error when we choose for each replication $i$, $K^{\star}_{\mathrm{MISE}}$ or the data-driven estimate $\hat{K}^{\star,(i)}_{\mathrm{cAIC}}$. We emphasize that $\widehat{\mathrm{Loss}}_{\mathrm{MISE}}$ exactly corresponds to the Monte-Carlo estimate of the relative MISE evaluated at $K^\star_{\mathrm{MISE}}$ 
for the cosine and Haar bases for which $w(r)=1$. We choose to define these criterions (and thus change the notation) so that we are able to compare the results between different bases, in particular defined with different weight functions. Finally, for the data-driven selected smoothing parameter, we also report the relative Monte-Carlo mean squared error of the parameter $\theta_0^\star=\log\beta$. Table~\ref{table:smoothing} summarizes our findings regarding the quality of the selection of the optimal number of orthonormal polynomials, while Figure~\ref{fig:envelopes.cAIC.pifs} and Figure~\ref{fig:envelopes.cAIC.gs} in the Supplementary Material depict the Monte Carlo means and 95\% confidence envelopes of estimates of $\phi$ and $g$ (based on $\hat K^{\star}_{\mathrm{cAIC}}$).

\subsection{Simulation results}

We first comment on Table~\ref{table:smoothing}. For all types of PIPPs, the integers $K^{\star}_{\mathrm{MISE}}$ and $K^{\star}_{\mathrm{cAIC}}$ increase most of the time as the spatial domain enlarges and the number of points increases.  The cutoffs $K^{\star}_{\mathrm{MISE}}$ and $K^{\star}_{\mathrm{cAIC}}$ are almost identical for all models under the three orthonormal bases. For instance, $K^{\star}_{\mathrm{MISE}}$ and $K^{\star}_{\mathrm{cAIC}}$ are the same for PIF1 and PIF3 models under the cosine and the Haar systems. The mean and standard deviation of 500 values of $\hat K^{\star}_{\mathrm{cAIC}}$ are also given in Table~\ref{table:smoothing}. As explained in the previous section, we remind that this column is based on data-driven estimate of the smoothing parameter, while the columns related to $K^{\star}_{\mathrm{MISE}}$ and $K^{\star}_{\mathrm{cAIC}}$ can only be computed using Monte-Carlo replications. For all four models and the three orthonormal bases, one observes that the mean of $\hat K^{\star}_{\mathrm{cAIC}}$ increases when the amount of data increases. Regarding its standard deviation, we observe that the variability in the values of $\hat K^{\star}_{\mathrm{cAIC}}$ is moderate in particular for the observation domains $W_2$ and $W_3$. The increasing of the different $K^\star$ is in agreement with the theoretical requirement  that $K_n\to \infty$ with the volume of the observation domain. It is also satisfactory that the means of the $\hat K^\star_{\mathrm{cAIC}}$ are quite close to $K^\star_{\mathrm{MISE}}$. From the two penultimate columns, we can first observe that the relative $L^2$-errors (empirical loss functions) get  smaller with the amount of data. Even if we did not seek at minimizing the MISE, it is interesting to observe that the second last two columns are not too far from each other, except maybe for PIF1 in $W_1$. From these different comments, we come to the conclusion that the use of the composite AIC is a pertinent data-driven procedure to select the smoothing parameter $K$.

Finally, if we specifically focus on the penultimate column which allows us to compare the different bases, one observes that the Fourier-Bessel basis (resp. Haar basis) seems the most appropriate for PIF1 (resp. PIF2) except on W1. The performances of the three bases functions are quite similar for PIF3 while the Haar basis provides worse results for PIF4 than the two other ones.  

Figure~\ref{fig:envelopes.cAIC.pifs} and Figure~\ref{fig:envelopes.cAIC.gs} in the Supplementary Material show the result of this procedure on estimates of $\phi$ and $g$ respectively using the three orthonormal bases. The comments below apply to estimates of $\phi$ and $g$. Overall, these results indeed show the efficiency of the procedure. 
Whatever the nature of the pairwise interaction function (repulsive, attractive or a mixture of repulsion and attraction) and whatever the orthonormal basis function used, on average, the true function ($\phi$ or $g$) is well recovered at least for domains $W_2$ and $W_3$. We also observe that the widths  of envelopes  are get smaller as the amount of data increases, which intrinsically means that the variance of the estimates tends to 0 as $|W_\ell|$ increases. We may mention that the variability at short distances for estimates of models PIF2 and PIF3 is very large even for the domain $W_3$ making the pattern indistinguishable from a Poisson pattern.

The approximation of $\phi_1$ (and $g_1$) is a success story for the cosine and the Fourier-Bessel bases. Even the approximation in $W_1$ gives a fair visualization of $\phi_1$, and the approximations in $W_2$ and $W_3$ provide an almost perfect fit. The Haar basis does not approximate $\phi_1$ as good as the cosine and the Fourier-Bessel bases. In particular, the estimates tend to be negatively biased when $r$ gets close to $\tilde R=0.08$. The interaction function $\phi_2$ is a piecewise constant function. Thus, there is no surprise that $\phi_2$ (and $g_2$) is perfectly fitted by the Haar system-after all. Note that $\phi_2$ and $g_2$ are also fairly fitted by the cosine and the Fourier-Bessel bases. The purely attractive interaction function $\phi_3$ (and $g_3$) looks much more well-fitted with the cosine and Fourier-Bessel bases than the Haar one, especially when $r$ is close to $\delta=0.01$ or $\tilde R=R+\delta=0.08$. Finally, the interaction functions $\phi_4$ and $g_4$ are fairly fitted by the three orthonormal systems. 

In terms of variability, these figures bring a new information compared to Table~\ref{table:smoothing}. For the four pairwise interaction functions, estimates of $\phi$ and $g$ using the Fourier-Bessel basis tend to have a larger (resp. smaller) variance than the ones using the cosine and Haar bases when $r$ is close to $\delta$ (resp. close to $R+\delta$). These behaviors are clearly accountable for the weight function ($r^{d-1}$ for the Fourier-Bessel basis and 1 for the two other ones) but we do not have a theoretical explanation for this.

From this simulation study, we  recommend the use of the Fourier-Bessel basis (resp. the Haar basis) if one suspects the pairwise interaction to be smooth (to be piecewise constant). 

To supplement our asymptotic normality result in Section~\ref{sec:result}, we consider the distribution of $z_n(r):=s_n^{-1}(r) \left \{ \hat{g}_n(r;K^{\star}_{\mathrm{cAIC}}) - g(r) \right \}$ for $r=0.035$ and $r=0.05$ in case of PIF$k$ model, $k=1,\dots,4$ and using the three orthonormal bases. Results for spatial domain $W_\ell$, $\ell=1,2,3$ are summarized in Figure~\ref{fig:proof.theorem21} and Figure~\ref{fig:proof.theorem22} in the Supplementary Material. Overall, one can observe that the distributions of $z_n(r)$ get closer to the one of a standard Gaussian random variable as $|W|$ increases.

\subsection{An example when $\delta$ and $R$ are estimated or misspecified} 
\label{sim:eip}

The purpose of this section is to show that misspecifying or estimating $\delta$ and $R$ does not deteriorate the efficiency of the estimators. To this end, we consider a simulation simpler study  than that of Section~\ref{sim:ortho} and generate $500$ realizations from GPPs with pairwise interaction functions $\phi_k$, $k=1,\dots,4$ on the spatial domain $W=[0,2]^2$. The simulation parameters are unchanged: in particular $\delta=0$ and $R=0.08$ for PIF1 and PIF2 models while $\delta=0.01$ and $R=0.07$ for PIF3 and PIF4 models. We consider only the Fourier-Bessel basis. Three scenarii are considered:
\begin{enumerate}[ \textsf{Scenario}~1.]
\item Same values as in the simulation set-up, that is $\delta=0$ and $R=0.08$ for the PIF1 and PIF2 models while $\delta=0.01$ and $R=0.07$ for the PIF3 and PIF4 models. \label{orth:scen1} 
\item $\delta$ and $R+\delta$ are replaced by their estimated values. We estimate $\delta$ ($\text{if} >0$) by $\hat \delta= n/(n+1) d_{\min}$ where $n$ is the number of points and $d_{\min}$ is the smallest observed distance between two points. To estimate $R+\delta$, we use the $K$-function i.e. we estimate the irregular parameter $R+\delta$ by the distance $r$ of maximal squared deviation of the border-corrected estimate of the Ripley's $K$-function. \label{orth:scen2} 
\item $R+\delta$ is replaced by an upper bound $R_{\max}=0.12$. 
\end{enumerate}

Results are summarized in Figure~\ref{fig:scenario.pifs} and Figure~\ref{fig:scenario.gs} in the Supplementary Material. For all type of PIPPs, it can be seen that estimations of $\phi$ and $g$ in  \textsf{Scenarii}~$2$ and~$3$ perform favorably compared to  \textsf{Scenario}~$1$.  As seen from the first two columns and PIF3-PIF4 models, the estimation of the hard core parameter has no influence on the estimates of $\phi$ and $g$ close to $\delta$.  The most interesting observation is on the behaviors of estimates of $\phi$ and $g$ for \textsf{Scenarii} $2$-$3$ when $r>R+\delta$. For both scenarii, we indeed observe that estimates of $\phi$ (resp. $g$) are very close to 1 (resp. 0), as expected. The counterpart is that the envelopes tend to be wider for \textsf{Scenarii} $2$-$3$ compared to the ones of the benchmark \textsf{Scenario}~$1$. However, this increase of variability is very moderate. Overall, this simulation study shows that the estimation results are very stable when the irregular parameters are estimated or misspecified.

\section{Conclusion} \label{sec:conclusion}  

In this paper, we develop an orthogonal series estimation of the log of the pairwise interaction function. The coefficients involved in the expansion of the function are obtained by maximizing a modified version of the pseudolikelihood function of the GPP.  We propose to use the composite Akaike information criterion for the selection of the  smoothing parameter. The implementation of the estimation procedure is quite simple thanks to an analogy with the pseudolikelihood for exponential family models. Our large simulation study confirms our  results and shows the practical efficiency for large classes of pairwise interaction functions and for different orthonormal basis functions. These simulation results also confirm the use of the composite AIC criterion we suggest as a data-driven procedure to select the number of orthonormal polynomials used in the series expansion. Finally, we apply our methodology to three real datasets. The results provide a fine description of  the interaction between the points for each dataset.

The use of orthonormal bases seems to be crucial to verify Lemmas~\ref{lemma1}-\ref{lemma2} which essentially show that the sensitivity matrix and variance covariance matrix of the score function are definite positive for $n$ large enough. Even if several orthonormal bases could be used in practice, we have not explored optimality criteria for the choice of basis system. In the same vein, we do not know if our results remain valid for expansions not based on orthonormal basis. Investigating such an extension is an interesting  perspective. There are a few other directions that are worth exploring in the future. First, it would be interesting from a practical point of view to extend the method to pairwise interaction point process with conditional intensity of the form $\lambda(u,\bx) = \beta(u) \prod_{v \in \bx} \phi (\Vert v - u \Vert)$ where $\beta(u)$ is modelled through say $p$ spatial covariates for instance or is also non parametrically estimated. Second, as an alternative of the data-driven selection of the tuning parameter $K$, we could consider a regularized version of the truncated pseudolikelihood (with a penalty term such as lasso, fused lasso, etc). Third, the methodology and the results are restricted to finite-range models. Relaxing this assumption seems quite straightforward from a methodological point of view. We could replace $\tilde R$ by a sequence $\tilde R_n$ which increases with the amount of data. Handling this extension from a theoretical point of view is however definitely challenging.   

\section*{Acknowledgements}

The authors are grateful to the editor, associate editor and reviewers for their suggestions and comments which led to a significant improved version of the manuscript. The authors also would like to sincerely thank Abdollah Jalilian for sharing their code implementing the Fourier-Bessel basis and Juha Heikkinen for fruitful discussions on the Bayesian smoothing approach developed in~\cite{heikkinen1999bayesian}. 

\bibliographystyle{imsart-number}
\bibliography{OrthogonalGPP}

\setlength{\tabcolsep}{1pt}
\renewcommand{\arraystretch}{1.2}
\begin{table}[p]
\caption{Smoothing parameters $K^{\star}_{\mathrm{MISE}}$ and $K^{\star}_{\mathrm{cAIC}}$, Monte Carlo means and standard deviations of  $\widehat{K}^{\star}_{\mathrm{cAIC}}$ based on $500$ replications of PIPP with interaction functions $\phi_l$, $l=1,\ldots,4$. The values (in $10^{-3}$) of the mean integrated squared error with $w=1$, that is ($\widehat{\mathrm{Loss}}$) evaluated at the smoothing parameters $K^{\star}_{\mathrm{MISE}}$ and $\widehat{K}^{\star}_{\mathrm{cAIC}}$ are provided in columns 7-8 respectively. The last column $\widehat{\mathrm{rMSE}}_{\mathrm{cAIC}}(\theta_0^\star)$ (in $10^{-4}$) reports the corresponding relative Monte-Carlo mean squared error of estimates of $\theta_0^\star=\log\beta$.}
\label{table:smoothing} 
\centering
\begin{tabular}{ ccc || c@{\hskip 0.2cm}c@{\hskip 0.2cm}c@{\hskip 0.2cm}c@{\hskip 0.2cm}c@{\hskip 0.2cm}c }
\hline
\hline 
 \multicolumn{1}{c}{Function} &\multicolumn{1}{c}{Basis} & \multicolumn{1}{c}{Spatial domain \;\;}    & $K^{\star}_{\mathrm{MISE}}$ & $K^{\star}_{\mathrm{cAIC}}$ & $\widebar{\widehat K}^{\star}_{\mathrm{cAIC}}\left(\hat{\sigma}_{{\widehat K}^{\star}_{\mathrm{cAIC}}}\right)$ 
 & $\widehat{\mathrm{Loss}}_{\mathrm{MISE}}$
 & $\widehat{\mathrm{Loss}}_{\mathrm{cAIC}}$ & 
 $\widehat{\mathrm{rMSE}}_{\mathrm{cAIC}}(\theta_0^\star)$
 \vspace{0.1cm} \\ 
 \hline
  \hline
 \multirow{9}{*}{PIF1} & \multirow{3}{*}{Cosine} & $W_1$    & 2 & 2 & 2.7(1.6) & 38.5 & 780.5 & 334.4\\ 
 &  & $W_2$    & 4 & 4 &  3.7(2.2) & 13.8 & 105.5 & 12.3\\ 
  & & $W_3$    & 4 & 4 & 4.5(2.2) & 6.4 & 15.3 & 5.0 \vspace{0.1cm} \\ 
  
& \multirow{3}{*}{FB} & $W_1$   & 2& 2 & 2.4(1.7) & 38.2 & 733.9 & 42.5\\ 
 &  & $W_2$  & 3 & 2& 2.8(1.7) & 13.4 & 29.7 & 10.2\\ 
  & & $W_3$  & 4 & 2 & 3.2(2) & 7.8 & 16.1 & 4.5 \vspace{0.1cm}\\ 
 
  & \multirow{3}{*}{Haar} & $W_1$   & 2& 2 & 4.2(1.5) & 159.9 & 90.8 & 43.3\\ 
  & & $W_2$  & 4 & 4 & 8.5(2.4) & 54.2 & 53.4 & 12.3 \\ 
  & & $W_3$   & 8 & 8 & 11.4(3.1) & 33.9 & 40.2 & 5.1 \\ 
  
 \hline
  \multirow{9}{*}{PIF2} & \multirow{3}{*}{Cosine} & $W_1$ & 5 & 5 & 4.8(2.4)& 151.9 & 211.5 & 48.8\\ 
   & & $W_2$ &  5& 10 & 9.9(2.7) & 55.0 & 71.7 & 8.7\\ 
   & & $W_3$  &  11 & 15& 12.2(2.4) & 43.1 & 47.1 & 3.8 \vspace{0.1cm}\\
   
  & \multirow{3}{*}{FB} & $W_1$ & 4& 4 &4.1(2.3)& 156.5 & 263.0 & 41.5\\ 
   & & $W_2$ & 9& 8& 9(2.5) & 65.1 & 74.7 & 8.6\\ 
  & & $W_3$  &  14 &15&11.4(2.8)& 44.9 & 50.4 & 4.6 \vspace{0.1cm}\\ 
  
  & \multirow{3}{*}{Haar} & $W_1$ & 2 & 2 & 3.7(1.7) & 134.8 & 138.7 & 46.0\\ 
   & & $W_2$ & 5 & 5 & 5.7(1.6) & 23.4 & 33.3 & 8.6\\ 
  & & $W_3$  & 5 & 5 & 6.3(2.2) & 12.6 & 20.8 & 3.8\\ 
   
   \hline
  \multirow{9}{*}{PIF3} & \multirow{3}{*}{Cosine} & $W_1$ & 1 & 1 & 1.5(1.2)& 329.3 & 656.7 & 18.8\\ 
   & & $W_2$ & 1& 1 & 2.7(2.3) & 199.3 & 341.9 & 3.3\\ 
   & & $W_3$  & 4 & 4& 4.1(2.6)& 138.4 & 220.3 & 2.4 \vspace{0.1cm}\\
   
  & \multirow{3}{*}{FB} & $W_1$ &  1& 1 & 1.7(1.1) & 607.0 & 998.6 & 19.6\\ 
   & & $W_2$ & 3& 3& 3.1(1.8) & 221.1 & 364.8 & 3.4\\ 
  & & $W_3$  & 3 & 3 & 4(2.0)& 114.6 & 193.2 & 2.2 \vspace{0.1cm}\\ 
  
  & \multirow{3}{*}{Haar} & $W_1$ & 1 & 1 & 1.4(1.2) & 329.3 & 657.0 & 18.9\\ 
  & & $W_2$ & 1 & 1& 2.1(2.1) & 199.3 & 335.5 & 3.3\\ 
  & & $W_3$  & 1 & 1 & 4(3.4) & 186.4 & 292.7 & 2.5\\ 
   \hline
   \multirow{9}{*}{PIF4} & \multirow{3}{*}{Cosine} & $W_1$ & 4 & 4 & 4.5(1.2)& 67.7 & 89.6 & 32.0\\ 
   & & $W_2$ &  5& 5 & 5.4(1.8) & 20.7 & 28.9 & 23.8\\ 
   & & $W_3$  &  6 & 6& 6.1(2.0)  & 17.3 & 19.3 & 22.2 \vspace{0.1cm}\\
   
  & \multirow{3}{*}{FB} & $W_1$ &  3& 2 & 3.4(1.8) & 62.8 & 96.6 & 28.8\\ 
   & & $W_2$ & 5& 5& 5.2(1.5) & 24.1 & 28.7 & 22.5\\ 
  & & $W_3$  & 5 & 5& 5.7(1.7)& 16.5 & 18.9 & 21.3 \vspace{0.1cm} \\ 
  
  & \multirow{3}{*}{Haar} & $W_1$ & 8 & 8 & 7.2(2.0) & 176.1 & 194.5 & 32.2 \\ 
  & & $W_2$ & 8 &8 & 11(2.9) & 75.7 & 82.6 & 23.8 \\ 
  & & $W_3$ & 15 & 15 & 14.2(1.7) & 57.1 & 58.4 & 22.4\\ 
  \hline
 \end{tabular}
\end{table}

\begin{figure}[p]
\begin{center}
\renewcommand{\arraystretch}{0}
\includegraphics[width=1\textwidth]{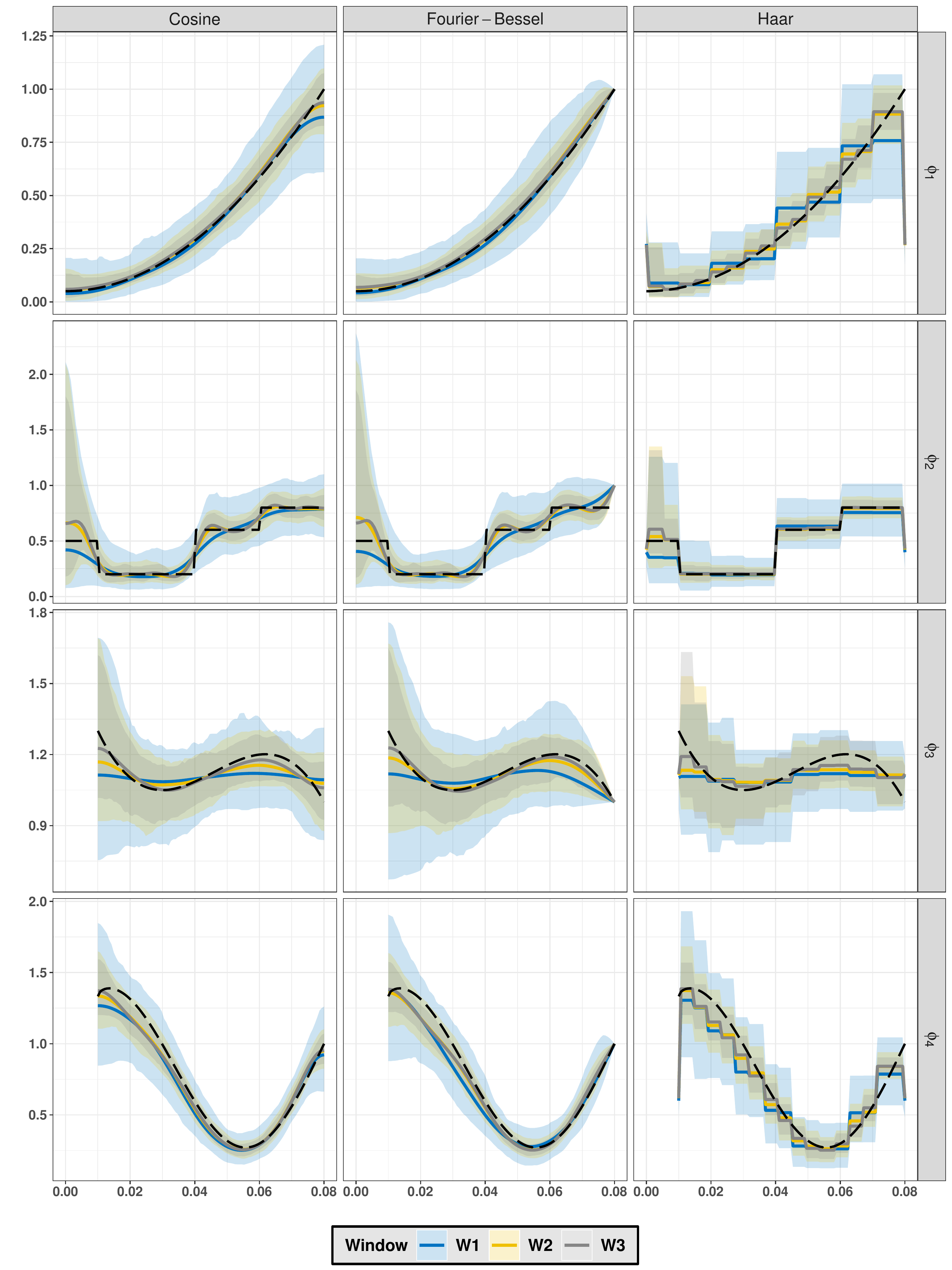} 
\caption{For $l=1,\ldots,4$, theoretical pairwise interaction functions $\phi_l$ (dashed black curves), Monte Carlo means of $\hat{\phi}_l(\cdot,\hat K_{\mathrm{cAIC}}^{\star})$ in $W_j$ ($j=1,2,3$). Envelopes correspond to 95\% Monte-Carlo pointwise confidence intervals.}
\label{fig:envelopes.cAIC.pifs}
\end{center}
\end{figure}

\begin{figure}[p]
\begin{center}
\renewcommand{\arraystretch}{0}
\includegraphics[width=1\textwidth]{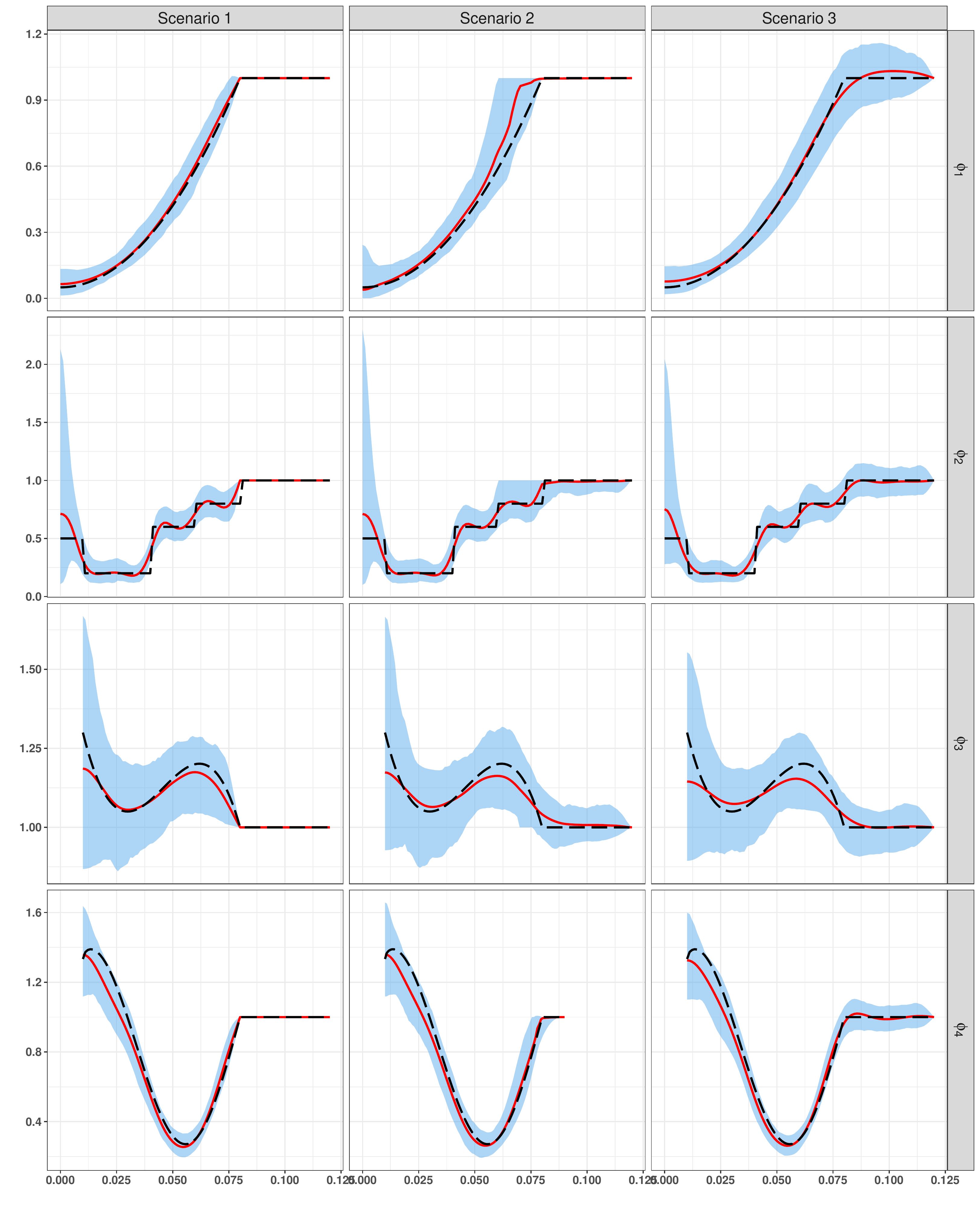} 
\caption{For $l=1,\ldots,4$, theoretical pairwise interaction functions $\phi_l$ (dashed black curves), Monte Carlo means of $\hat{\phi}_l(\cdot,\hat K_{\mathrm{cAIC}}^{\star})$ (red curves) using the Fourier-Bessel basis in $W_2$ when $\delta$ and $R$ are known (first column), estimated (second column), or when $R+\delta$ is set to $R_{\max}=0.12$ (third column). Envelopes correspond to 95\% Monte-Carlo pointwise confidence intervals.}
\label{fig:scenario.pifs}
\end{center}
\end{figure}






\newpage

\centerline{\sc \large SUPPLEMENTARY MATERIAL}
\appendix

\section{Application to real datasets} \label{sec:app} ${}$
 
To illustrate the  procedure on real datasets, we consider point patterns of locations of ``light-on''  displaced amacrine cells in the retina of a rabbit ($152$ points observed in $[0,1060]\times[0,662] \;(\mu m^2)$), longleaf pine trees ($584$ points observed in $[0,200]\times[0,200] (m^2)$) and electrical breakdown spots in a microelectronic capacitor ($964$ points observed on a circular electrode enclosed in the rectangle $[-282,282]\times[-282,282] \;(\mu m^2)$). These datasets are accessible from \texttt{R} under the \texttt{spatstat} package. Hereafter, we refer respectively to these datasets as $D_1$, $D_2$ and $D_3$. Note that in our analysis of $D_2$, we use the unmarked point pattern i.e. the three types of points are now of a single type. We omit the duplicated points of $D_3$. The estimates of $\delta$ are $21$, $0.2$ and $1.1$ respectively. These values give upper bounds $R+\delta$: $120$, $20$ and $15$ for $D_1$, $D_2$ and $D_3$. We apply Algorithm~\ref{alg:algorithm1} with the Fourier-Bessel basis. The  cut-offs $\hat{K}^{\star}=\hat K^{\star}_{\mathrm{cAIC}} $ respectively equal to $2$, $6$ and $3$ for the three point patterns and we use the result in Theorem~\ref{THM:CONST}(iii) to determine the $95\%$ pointwise confidence envelopes for the pairwise interaction function $\phi$ and its logarithm $g$. The computational times (with a single cluster) for finding the above cut-offs are respectively $108$ sec, $1560$ sec and $1080$ sec which can be significantly reduced by considering more clusters. For instance, with Compute Canada servers, it takes a few seconds to find these values. We have also looked at the computational times in a single cluster for the estimation of the coefficients (so of $g$ and $\phi$ functions) and the variance covariance matrices; these are respectively $16$ sec, $104$ sec and $34$ sec for $D_1$, $D_2$ and $D_3$. Furthermore, we include estimates of the summary function L together with envelopes from the fitted pairwise interaction models.  

Results are reported in Figure \ref{fig:app.cAIC}. For the dataset $D_1$, the estimate of the pairwise interaction function is below $1$ for all distances $r$, indicating repulsion as suggested by the representation of $D_1$. This dataset was also analyzed by~\cite{heikkinen1999bayesian} and~\cite{berthelsen2003likelihood} using Bayesian approaches. The posterior mode distribution obtained by~\cite{heikkinen1999bayesian} can be found in their Figure~3, while the posterior mean distribution and its 95\% credibility interval obtained by~\cite{berthelsen2003likelihood} can be found in their Figure~8. Our nonparametric estimate is in agreement with the estimates provided by both Bayesian approaches: the estimate is small or zero in our case when $r\le 21$, increases and reaches values in $[0.9,1]$ when $r\ge 100$. The lower envelope of the orthogonal series estimate is also quite similar to the corresponding one obtained by~\cite{berthelsen2003likelihood}. Our upper envelope (truncated to 1 for a better comparison with previous works) is however larger in particular when $r\le 60$. The pointwise confidence envelopes we propose are based on the asymptotic normality and should be interpreted with caution for this dataset which has only 152 points. Overall, the results we obtain are very satisfactory. 
It is worth repeating the advantages of our nonparametric approach when we compare it to Bayesian approaches: stability and speed of the algorithm (a few seconds to get the estimate and envelopes for this dataset $D_1$), the absence of tuning parameters (except the smoothing parameter $K$ which is estimated using a data-driven procedure), the framework that the number of points is indeed random (which is considered as fixed by~\cite{heikkinen1999bayesian}) and the adaptivity of the method to estimate any type of pairwise interaction function (not restricted to increasing bounded by~1 pairwise interaction functions as in \cite{berthelsen2003likelihood}).

Regarding $D_2$, the estimation of the pairwise interaction function suggests clustering at all scales which is in perfect concordance with the behavior of the points of $D_2$. \cite{wong2021poles} found the same result in their analysis of $D_2$ using the pair correlation function. For $D_3$, we observe repulsive clusters that is at some distances (between approximatively $7.13 \, \mu m$ and $12 \, \mu m$) the point pattern shows a moderate degree of clustering while at other distances it displays some degree of spatial regularity. This result is in agreement with the configuration of the points of $D_3$. 

As for the envelopes of the estimated functions $\hat{\phi}(r)$ and $\hat{g}(r)$, they are tighter for $D_3$ which is expected according to the precision of the estimation in an asymptotic framework. Notice that the scenario in $D_3$ is more suitable to an asymptotic setting than those in $D_1$ and $D_2$. 

To see how well the fitted pairwise interaction models fit to the data, we perform global envelopes for composite hypotheses~\cite{myllymaki2017global} with the summary function L using the \texttt{GET.composite} function of the \texttt{GET} \texttt{R} \texttt{package}~\cite{myllymaki2019get}. We generated $s=499$ replications of the fitted model and for each of these simulations $s_2=199$ refitted replications were used, so in total $10^5$ simulations. The smoothing parameter was kept fixed and we used the studentized envelope test. 

Figure~\ref{fig:app.cAIC} (fourth row) show the test results for the three datasets. For $D_2$ and $D_3$, one can see that the fitted pairwise interaction models are not rejected at the significance level $\alpha=5\%$. As for $D_1$, the test result suggests a rejection of the fitted pairwise interaction point process at the significance level $\alpha=5\%$: this might be explained by the fact that we do not have enough points in this setting to reduce the bias induced by the nonparametric approach (recall that the model under the null hypothesis is obtained using such an approach). These global envelope tests should be interpreted with caution (in particular when the number of points is moderate). The considered null hypothesis is not $H_0$: $\bX$ is a PIPP but more $H_0$: $\bX$ is a PIPP with parametric logarithm of PIF of the form $\sum_{k=1}^K \theta_k\varphi_k(r-\delta)$ with $K=2,3,6$ for $D_1,D_2$ and $D_3$ respectively. Deriving an appropriate global envelope test is an interesting question which is left for a future research.

\begin{figure}[!h]
\centering
\renewcommand{\arraystretch}{0}
\includegraphics[width=1\textwidth]{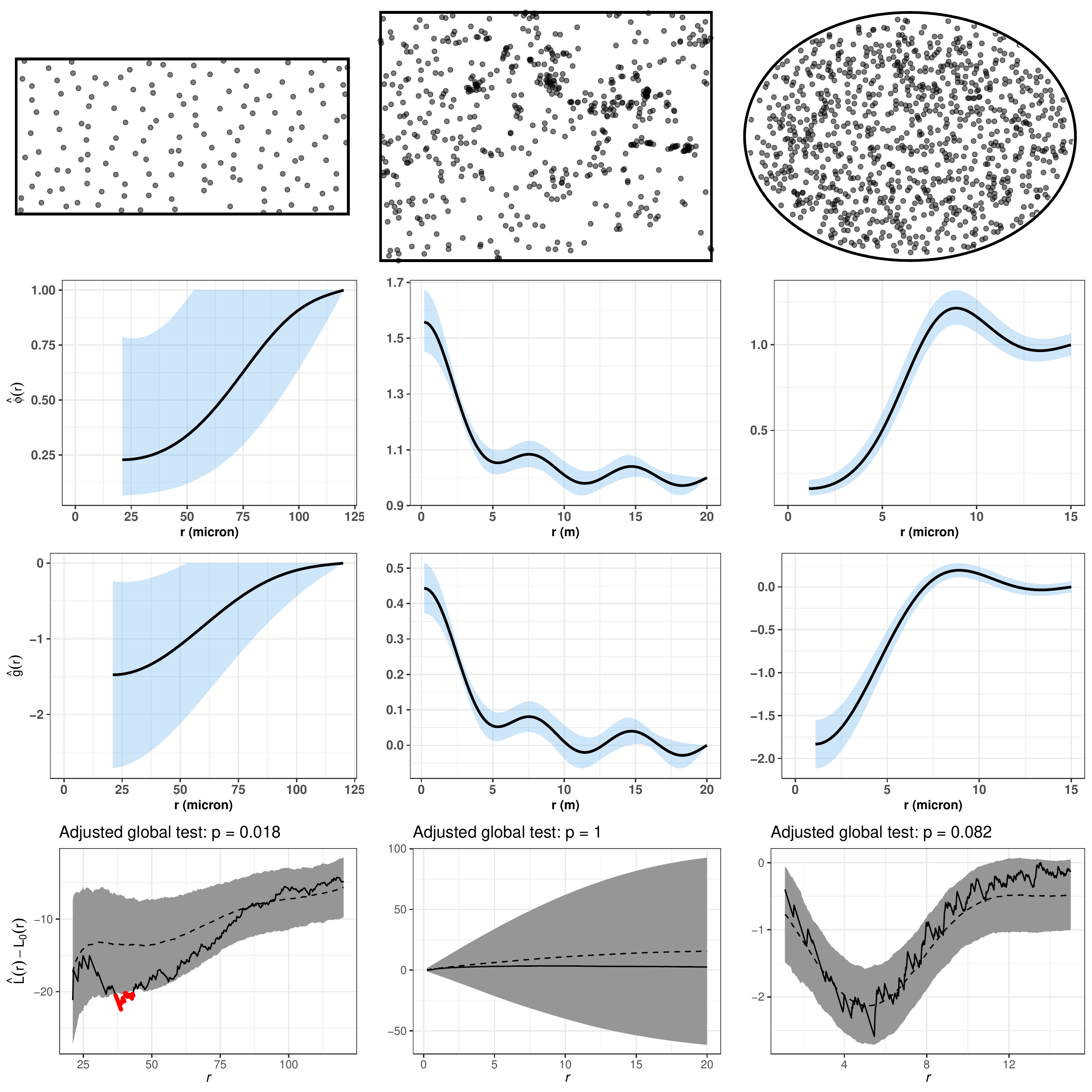}  
\caption{From left to right: $152$ locations of amacrine cells in the retina of a rabbit, $584$ longleaf pine trees, and $964$ locations of electrical breakdown spots in a microelectronic capacitor (first row), estimated pairwise interaction functions (second row), estimated log-pairwise interaction functions  (third row) and global envelope tests with $T(r)=\hat{L}(r)$ used to test the fitted pairwise interaction models for the three datasets where the solid black line denotes the data function and the dashed line represents the expectation $T_{0}(r)$ (fourth row). The selected smoothing parameter $\hat{K^{\star}}$ used in the estimation is obtained using the composite AIC. We have  $\hat{K^{\star}}=2$ for the amacrine cells data, $\hat{K^{\star}}=6$ for the longleaf pine trees and $\hat{K^{\star}}=3$ for the electrical breakdown spots. The steel blue areas represent the $95\%$ pointwise confidence intervals for the estimates and the grey areas represent the $95\%$ global (`st') envelopes.}
\label{fig:app.cAIC}
\end{figure}

\newpage 

\section{Auxiliary Lemmas}

\begin{lemma} \label{lem:Esk}
Let $p, k\in \mathbb N$ with $p\ge2$. Under conditions \ref{C:model}-\ref{C:anKn}, there exists $\kappa<\infty$ such that
\begin{enumerate}[(i)]
\item \begin{equation}\label{eq:Esk}
\EE \left[ |S_k(0,\bX)|^p\right]	 \le \kappa^p \|\varphi_k\|_\infty \EE[Z^p] = O(k^{m(p-2)})
\end{equation}
where $Z$ is a Poisson random variable with mean 1.\\
\item
\[
	\EE \left[ \|\bS_{K_n} (0,\bX) \|^p\right] = O(K_n^{m(p-2)+p/2}).
\]
\item For $p=2,4$
\begin{equation}\label{eq:rKn}
\EE \left[ |r_{K_n}(0,\bX)|^p\right] = O(a_{K_n}^p) \quad \text{ where } \quad r_{K_n}(u,\bX)= 1 - \exp\left( -\sum_{k>K_n} \theta_k^\star S_k(u,\bX)\right).
\end{equation}
\end{enumerate}
\end{lemma}

\begin{proof}
(i) First, note that $|S_k(0,\bX)|^p\le h(\bX)$ where $h(\bX)=\big(\sum_{v\in \bX\cap \mathcal B(0)} |\varphi_k|(\|v\|-\delta) \big)^p$. Since $\bx \mapsto h(\bx)$ is an increasing function (i.e. $h(\bx)\le h(\mathbf y)$ for any $\bx \subset \mathbf y$) and since $\bX \cap \mathcal B(0)$ is stochastically dominated by a Poisson point process with intensity $\rho=\bar \lambda |\mathcal B(0)|$ by condition \ref{C:model} (which ensures \eqref{ls}), we have (see e.g. \cite{bassan1990moments,privault2012moments}) that 
\[
	\EE \left[ |S_k(0,\bX)|^p\right] \le \sum_{r_1,\dots,r_p \in \mathcal I(p)} K_p(r_1,\dots,r_p) \prod_{\ell=1}^p \left( \int_{\mathcal B(0)}\varphi_k^\ell (\|v\|-\delta)\rho \mathrm d v\right)^{r_\ell}
\]
where 
\[
	\mathcal I(p) = \left\{ (r_1,\dots,r_p)\in \mathbb N^p: \sum_{\ell=1}^p \ell r_\ell=p\right\}
	\quad \text{ and } \quad
	K_p(r_1,\dots,r_p)= p! \prod_{\ell=1}^p \frac{1}{(\ell !)^{r_\ell} r_\ell!}.
\]
Now, observe that using \ref{C:anKn}(iii) 
\begin{align*}
	\left( \int_{\mathcal B(0)}\varphi_k^\ell (\|v\|-\delta)\rho \mathrm d v\right)^{r_\ell} &=
	\left( \int_{0}^R\varphi_k^\ell (r)w(r) \frac{(r+\delta)^{d-1}}{w(r)} \rho \mathrm d r\right)^{r_\ell} \\
	&\le \|\varphi_k\|_\infty^{r_\ell (\ell-2)\vee 0} \kappa^{r_\ell} 
\end{align*}
with $\kappa=\rho \sup_{r\in [0,R]} (r+\delta)^{d-1}/w(r)$. Therefore, since $\kappa^{\sum_\ell r_\ell}\le \kappa^p$ for any $(r_1,\dots,r_p)\in \mathcal I(p)$ we have
\[
	\EE \left[ |S_k(0,\bX)|^p\right] \le \, \kappa^p \sum_{r_1,\dots,r_p \in \mathcal I(p)} K_p(r_1,\dots,r_p)  \|\varphi_k\|_\infty^{f(r_1,\dots,r_p)}
\]
with $f(r_1,\dots,r_p) = \sum_{\ell=1}^p (r_\ell(\ell-2))\vee 0$ which rewrites as $f(r_1,\dots,r_p)=p-\tilde f(r_1,\dots,r_p)$ with $\tilde f(r_1,\dots,r_p)= r_1 +2\sum_{\ell=2}^p r_\ell$. Since $\sum_{\ell=1}^pr_{\ell}\ge 1$, if $r_1=0$ then $\sum_{\ell=2}^p r_\ell\ge 1$ and $\tilde f(r_1,\dots,r_p)\ge 2$. The same conclusion is also obviously true if $r_1\ge 2$. Finally if $r_1=1$, since $p> 1$, $\sum_{\ell=2}^p r_\ell\neq0$ otherwise $\sum_{\ell=1}^p \ell r_\ell =1\neq p$. Therefore, for any $(r_1,\dots,r_p)\in \mathcal I(p), f(r_1,\dots,r_p)\le p-2$ whereby we deduce that
\begin{align*}
\EE \left[ |S_k(0,\bX)|^p\right] & \le \kappa^p \|\varphi_k\|_\infty^{p-2} 	\sum_{r_1,\dots,r_p \in \mathcal I(p)} K_p(r_1,\dots,r_p) \\
&\le  \kappa^p \|\varphi_k\|_\infty^{p-2} \EE[Z^p] \quad \text{ where } Z\sim \mathcal P(1).
\end{align*}

(ii) From Hölder's inequality
\begin{align*}
\|\bS_{K_n}(0,\bX)\|^p &\le  \left( \sum_{k=1}^{K_n}|S_k(0,\bX)|^{2\alpha}\right)^{p/2\alpha} \left( K_n\right)^{p/2\beta}	 \quad \text{ with } 1/\alpha+1/\beta=1 \\
&\le K_n^{p/2-1}  \sum_{k=1}^{K_n}|S_k(0,\bX)|^{p} \quad \text{ with } p=2 \alpha.
\end{align*}
Hence, from (i)
\[
	\EE \left[ \|\bS_{K_n}(0,\bX)\|^p \right]= O\left(K_n^{p/2-1} \sum_{k=1}^{K_n} k^{m(p-2)} \right)= O\left(K_n^{m(p-2)+p/2}\right).
\]
(iii)	Let $p=2$. We have using a series expansion of the exponential function
\begin{align*}
r_{K_n}(0,\bX)^2 &= 1+ \mathrm{e}^{-2\sum_{k>K_n} \theta_k^\star S_{k}(0,\bX)} - 2 	\mathrm{e}^{-\sum_{k>K_n} \theta_k^\star S_{k}(0,\bX)} \\
&= \sum_{j=2}^\infty \frac{(-1)^j}{j!} \left\{\left( 2\sum_{k>K_n} \theta_k^\star S_{k}(0,\bX)\right)^j
-2 \left( \sum_{k>K_n} \theta_k^\star S_{k}(0,\bX)\right)^j 
\right\}\\
&\le \sum_{j\ge 2} \frac{2^j}{j!} \left|\sum_{k>K_n} \theta_k^\star \|\varphi_k\|_\infty \right|^j  N(\mathcal B(0))^j \\
&\le \sum_{j\ge 2} \frac{(2a_{K_n})^j }{j!}   N(\mathcal B(0))^j \\
&\le 4a_{K_n}^2 N(\mathcal B(0))^2 \sum_{j\ge 2} \frac{1}{(j-2)!}   (2N(\mathcal B(0)))^{j-2}\\
&\le 4 a_{K_n}^2 N(\mathcal B(0))^2 \exp(2N(\mathcal B(0))).
\end{align*}
Thus, $\EE(r_{K_n}(0,\bX)^2)= O(a_{K_n}^2)$ from~\eqref{ls} which implies that for any bounded $A\subset \mathbb R^d$, $\EE[N(A)^p \exp(cN(A))]<\infty$ for any $p\ge 1,c>0$.


For the case $p=4$, we follow the same strategy and leave the reader to check that
\begin{align*}
 r_{K_n}(0,\bX)^4	&= 1+ \sum_{j\ge 0} \frac{(-1)^j}{j!}\left( -4+6 2^j -4  3^j + 4^j\right) \left( \sum_{k>K_n} \theta_k^\star S_k(0,X)\right)^j \\
&= \sum_{j\ge 4} \frac{(-1)^j}{j!}\left( -4+6  2^j -4  3^j + 4^j\right) \left( \sum_{k>K_n} \theta_k^\star S_k(0,X)\right)^j \\
& \le 12 \sum_{j\ge 4} \frac{(4a_{K_n})^j}{j!} N(\mathcal B(0))^j \\
&\le 3072 \, a_{K_n}^4 N(\mathcal B(0))^4 \exp(4N(\mathcal B(0)))
\end{align*}
which yields $\EE[r_{K_n}(0,\bX)^4]=O(a_{K_n}^4)$ using again~\eqref{ls}.
\end{proof}


\begin{lemma}
\label{lemma1}
Under conditions \ref{C:Wn}-\ref{C:anKn}, the following results hold for any $\by, \, \btheta \in \mathcal{S}_m$
\begin{enumerate}[(i)] 
\item $|W_n|^{-1} \, \by^\top  \mathbf{A}(\bX;\btheta) \by \xrightarrow[n \to \infty]{\mbox{a.s.}} \by^\top \widebar{\mathbf{A}}(\btheta) \by$, \\
\item $|W_n|^{-1} \left(
\by_{K_n}^\top  \mathbf{A}_{K_n}(\bX;\btheta) \by_{K_n} -
\by^\top  \mathbf{A}(\bX;\btheta) \by
\right) \xrightarrow[n\to \infty]{\mbox{a.s.}} 0$,
\item 
$|W_n|^{-1} \, \by^\top  \mathbf{B}(\bX;\btheta) \by \xrightarrow[n \to \infty]{\mbox{a.s.}}  \by^\top \widebar{\mathbf{B}}(\btheta) \by$,
\item 
$|W_n|^{-1} \left(
 \by_{K_n}^\top  \mathbf{B}_{K_n}(\bX;\btheta) \by_{K_n} 
-\by^\top  \mathbf{B}(\bX;\btheta) \by
 \right) \xrightarrow[n \to \infty]{\mbox{a.s.}} 0$.
\end{enumerate}
\end{lemma}

\begin{proof}
The proof uses extensively Theorem 4~\cite{nguyen1979ergodic}, referred to as the ergodic theorem hereafter.

\begin{enumerate}[(i)]
\item Let 
\begin{align*}
\alpha_{W_n}(\by,\bX;\btheta) &= \by^\top \mathbf{A}(\bX; \btheta) \by ={\int_{W_n} \left\{ \sum\limits_{k=0}^{\infty} y_k \, S_k(u,\bX) \right\}^2 \lambda (u,\bX;\btheta) \mathrm{d}u }.
\end{align*}
We have that $({\alpha}_\Lambda)_\Lambda$ is additive and covariant \cite[see][]{nguyen1979ergodic}. Let $W_0=[0,1]^d$ and $ D \subset W_0$. Under the local stability ensured by \ref{C:model}, we have
\begin{align*}
| {\alpha}_D (\by,\bX;\btheta) | \leq {\int_{D}  \bar{\lambda} \, \left\{ \sum\limits_{k=0}^{\infty} |y_k| \, S_k(u,\bX) \right\}^2  \mathrm{d}u }.
\end{align*}
Under the finite range property ensured by \ref{C:model}, we have that almost surely for any $k\ge 1$
\begin{align*} 
S_k(u,\bX) = \sum_{v \in \bX \cap \mathcal B(u)} \varphi_k(\|v-u\|-\delta) \leq & \, N(\mathcal B(u)) \| \varphi_k \|_{\infty} 
\leq  \, c \, k^m N(\mathcal B(u))
\end{align*}
for some constant $c>0$. Hence, it follows that 
\begin{align*}
| \bar{\alpha}_D (\by,\bX;\btheta) | \leq & \, c {\int_{D} (1+N(\mathcal B(u))  )^2 \|\by\|_m^2  \mathrm{d}u } \\
\leq & \, c  \|\by\|_m^2 \, ( 1 + N(W_0 \oplus (R+\delta)))^2.
\end{align*}
 Since $\by \in \mathcal{S}_m$ and since $\EE [ N(A)^2 ] < \infty$ for any bounded $A \subset \mathbb{R}^d$, $| {\alpha}_D (\by,\bX;\btheta) |$ is bounded by an integrable random variable uniform in D. We are therefore allowed  to apply the ergodic theorem:  $|W_n|^{-1} {\alpha}_{W_n} (\by,\bX;\btheta) \xrightarrow[n \to \infty]{a.s.} |W_0|^{-1} \EE [{\alpha}_{W_0} (\by,\bX;\btheta) ]$ and from the stationarity of $\bX$, this reduces to
\[
|W_n|^{-1} {\alpha}_{W_n} (\by,\bX;\btheta) \xrightarrow[n \to \infty]{a.s.} \EE \left[ \left \{ \sum\limits_{k=0}^{\infty} y_k \, S_k(\mathrm{0},\bX) \right\}^2 \lambda(\mathrm{0},\bX;\btheta) \right]
\]
for any $\by, \, \btheta \in \mathcal{S}_m$. 

\item Let $\alpha^\prime_{W_n}(\bX)= 
\by_{K_n}^\top  \mathbf{A}_{K_n}(\bX;\btheta) \by_{K_n} -
\by^\top  \mathbf{A}(\bX;\btheta) \by$. Using algebraic identities and the local stability of $\bX$, we have
\begin{align*}
\alpha^\prime_{W_n}(\bX)&= \int_{W_n} 
\left[
\left\{\sum_{k=0}^{K_n} y_k S_k(u,\bX)  \right\}^2- 
\left\{\sum_{k=0}^{\infty} y_k S_k(u,\bX)  \right\}^2 \right] 
\lambda(u,\bX;\btheta) \mathrm d u \\
&\le 2 \bar \lambda \|\by\|_m \,\sum_{k>K_n} k^m|y_k| \;  \alpha^{\prime\prime}_{W_n}(\bX)
\end{align*}
where $\alpha^{\prime\prime}_{W_n}(\bX) = \int_{W_n} (1+N(\mathcal B(u)))^2 \mathrm d u$. Using the same arguments as in (i), we can apply the ergodic and then use the stationarity of the process to prove that as $n\to \infty$, $|W_n|^{-1}\alpha^{\prime\prime}_{W_n}(\bX)\to \EE[(1+N(\mathcal B(0)))^2]$ almost surely. Therefore almost surely 
\[
	|W_n|^{-1} \alpha^\prime_{W_n}(\bX) = o_{\mathrm{a.s.}}\left(\sum_{k>K_n} k^m|y_k|\right) = o_{\mathrm{a.s.}}(1)
\]
since $\by\in \mathcal{S}_m$.
\item Let 
\[
\beta_{W_n}(\bX) = \by^\top  \mathbf{B}(\bX;\btheta) \by = \int_{W_n}\int_{W_n} f(u,v,\bX;\btheta,\by) \mathrm du \mathrm dv	
\]
with
\begin{align*}
f(u,v,\bX;\btheta,\by) =&	\big\{ \by^\top \bS(u,\bX)\by^\top \bS(v,\bX) (1-\phi_{\btheta}(\|v-u\|)) \\
&+ ({\by_{-0}}^\top \boldsymbol{\varphi}(\|v-u\|))^2 \phi_{\btheta}(\|v-u\|) \big\} \lambda(u,\bX;\btheta)\lambda(v,\bX;\btheta).
\end{align*}
By definition of the basis functions and the finite range property, we have that
\begin{align*}
\beta_{W_n}(\bX) &= \by^\top  \mathbf{B}(\bX;\btheta) \by = \int_{W_n}\int_{W_n \cap \mathcal B(u)} f(u,v,\bX;\btheta,\by) \mathrm du \mathrm dv \\
&= \beta_{1,W_n^-}(\bX) + \beta_{2,\partial W_n}(\bX)	
\end{align*}
where $W_n^-=W_n\ominus (R+\delta)$ and $\partial W_n = W_n \setminus W_n^-$ and where
\begin{align*}
\beta_{1,W_n^-}(\bX) &= \int_{W_n^-}\int_{\mathcal B(u)} f(u,v,\bX;\btheta,\by) \mathrm du \mathrm dv \\
\beta_{2,\partial W_n}(\bX) &= \int_{\partial W_n}\int_{W_n\cap \mathcal B(u)} f(u,v,\bX;\btheta,\by) \mathrm du \mathrm dv
\end{align*}
Using the notation $\mathcal B_2(u)$ as the ball or annulus centered at $u$ with radii $\delta$ and $2(R+\delta)$ and since $\phi_{\theta}$ and $1-\phi_{\btheta}$ are bounded, it is left to check that for any $u \in W_n$ and $v \in \mathcal B(u)$
\begin{equation} \label{eq:upf}
|f(u,v,\bX;\btheta,\by)| \le c \nu(u) \quad \text{where} \quad \nu(u) = (1+N(\mathcal B(u)))(1+N(\mathcal B_2(u)))
\end{equation}
On the one hand, $(\beta_{1,\Lambda})_\Lambda$ is additive and covariant. Then, by using~\eqref{eq:upf} it is clear that for any $D\subset W_0=[0,1]^d$, $|\beta_{1,D}(\bX)|$ is bounded by an integrable random variable uniformly in $D$. Hence, the ergodic may be applied which by using the stationarity of $\bX$ and the fact that $|W_n|/|W_n^-|\to 1$ by \ref{C:Wn} leads exactly to $|W_n|^{-1} \beta_{1,W_n^-}(\bX) \to\by^\top \widebar{\mathbf{B}}(\btheta) \by$ almost surely as $n\to \infty$. On the other hand, from~\eqref{eq:upf}
\[
|\beta_{2,\partial W_n}(\bX)| \le c \int_{\partial W_n} \nu(u) \mathrm d u.
\]
Now, for any $p\ge 1$ from Hölder's inequality and the stationarity of $\bX$
\begin{align}
\EE \left[ |W_n^{-1} \beta_{2,\partial W_n}(\bX)|^p\right] \le& c^p |W_n|^{-p} \int_{(\partial W_n)^p}\EE [
\nu(u_1)\dots \nu(u_p)
]
\mathrm du_1 \dots \mathrm du_p  \nonumber\\
\le & c^p |W_n|^{-p}  
\left(\int_{\partial W_n}  
\EE[\nu(u)^p]^{1/p} \mathrm d u
\right)^p \nonumber\\
&\le c^p \EE[\nu(0)^p] \left( \frac{|\partial W_n|}{|W_n|}\right)^p. \nonumber 
 \end{align}
whereby we deduce using condition~\ref{C:Wn} and Borel-Cantelli's lemma  that $|W_n|^{-1} \beta_{2,\partial W_n}(\bX) \to 0$ almost surely as $n\to \infty$, which ends the proof.
\item Let $\beta^\prime_{W_n}(\bX)= 
\by_{K_n}^\top  \mathbf{B}_{K_n}(\bX;\btheta) \by_{K_n} -
\by^\top  \mathbf{B}(\bX;\btheta) \by = \int_{W_n}\int_{W_n\cap \mathcal B(u)} h(u,v,\bX;\btheta,\by) \mathrm d u \mathrm d v$ where
\begin{align*}
h(u,v,\bX;\btheta,\by) =& \Big\{\Big(
-\sum_{k>K_n} y_k S_k(u,\bX) \sum_{\ell=0}^\infty y_k S_k(v,\bX) 
-\sum_{k>K_n} y_k S_k(v,\bX) \sum_{\ell=0}^\infty y_\ell S_\ell(u,\bX) 
\Big) \\
&\times (1-\phi_{\btheta}(\|v-u\|))	-\Big( \sum_{k>K_n} y_k \varphi_k(\|v-u\|) \sum_{\ell=1}^\infty y_\ell \varphi_\ell(\|v-u\|) \Big) \\
&\times  \phi_{\btheta}(\|v-u\|) \Big\} \lambda(u,\bX;\btheta) \lambda(v,\bX;\btheta).
\end{align*}
Again, we leave the reader to check that for any $u\in W_n$ and $v\in W_n\cap \mathcal B(u)$ 
\begin{align*}
|h(u,v,\bX;\btheta,\by) | 
&\le c \Big( \sum_{k>K_n}k^m|y_k|\Big) \nu(u)
\end{align*}
whereby we deduce that
\[
	|\beta^\prime_{W_n}(\bX)| \le c \Big( \sum_{k>K_n}k^m|y_k|\Big) \int_{W_n} \nu(u) \mathrm du.
\]
Arguments similar to those developed in (i)-(iii) show that $|W_n|^{-1}|\beta^\prime_{W_n}(\bX)|\to 0$ almost surely as $n \to \infty$.
\end{enumerate}
\end{proof}

The next result is fundamental in  our paper. Combined with Lemma~\ref{lemma1}, it first states that the sensitivity matrix $\mathbf{A}_{K_n}(\bX;\btheta)$ and the variance covariance matrix $\mathbf{A}_{K_n}(\bX;\btheta^\star)+\mathbf{B}_{K_n}(\bX;\btheta^\star)$ of the estimating equation $\mathbf{e}_{K_n}(\bX;\btheta^\star)$ are positive definite for $n$ large enough. Those results are key-ingredients to apply a general central limit theorem (CLT) (from \cite{coeurjolly2017parametric}) to $\mathbf{e}_{K_n}(\bX;\btheta^\star)$.


\begin{lemma}
\label{lemma2}
 Assume \ref{C:Wn}-\ref{C:anKn}. Then,

\begin{enumerate}[(i)] 
\item 
$\underset{\substack { \by \in \mathbb{R}^{\mathbb{N}}, \| \by \| =1, \\ \btheta \in \mathcal{S}_m}}{\inf} \,  \by^\top  \widebar{\mathbf{A}}(\btheta) \, \by > 0$, \\

\item 
$\| \widebar{\mathbf{A}}_{K_n}^{-1}(\btheta_{K_n}^{\star}) \| = O(1)$ where $O(1)$ is uniform in $n$. \\

If in addition, condition \ref{C:psi} holds, then \\

\item $\underset{\substack { \by \in \mathcal{S}_m \\ \| \by \| =1}}{\inf} \, \by^\top \left \{ \widebar{\mathbf{A}}(\btheta^{\star}) + \widebar{\mathbf{B}}(\btheta^{\star}) \right \} \by > 0$ \\

\item
$ \bigg(\by_{K_n}^\top \left \{ \mathbf{A}_{K_n}(\btheta^{\star}) + \mathbf{B}_{K_n}(\btheta^{\star}) \right \} \by_{K_n} \bigg)^{-1/2} \by_{K_n}^\top \boldsymbol{e}_{K_n}(\bX;\btheta^{\star})  \xrightarrow{d} \mathcal{N}(0,1)$.

\item For any $\btheta,\btheta^\prime \in \mathcal{S}_m$ (eventually random)
 \[
 	\|\mathbf A_{K_n} (\bX;\btheta) -  \mathbf A_{K_n} (\bX;\btheta^{\prime})\| = O_{\mathrm P}(|W_n| K_n^{m+3/2} \|\btheta-\btheta^\prime\|).
 \]
 \item $\|\mathbf A_{K_n} (\bX;\btheta^\star) - \mathbf A_{K_n} (\btheta^\star)\| = o_{\mathrm P}(|W_n|).$
 \end{enumerate}
\end{lemma}

\begin{proof}

\begin{enumerate}[(i)]
\item Let $\by \in \ell^2(\mathbb{N})$ with $\|\by\|<\infty$ and $\btheta\in \mathcal{S}_m$ be such that $\by^\top \widebar{\mathbf{A}}(\btheta) \by =0$. We aim at proving that necessarily this leads to $\by=0$. Let $\by=(y_0,\by_{-0}^\top)^\top$\\
\underline{Step 1: Proof that $y_0=0$}. Let $E_0=\{\bx \in N_{lf}: N(\mathcal B(0))=0\}$ where we remind that $\mathcal B(u)$ stands for the ball or annulus centered at $u$ with radii $\delta$ and $R+\delta$. Now it is clear that 
\begin{align*}
0&= 	\by^\top \widebar{\mathbf{A}}(\btheta) \by  = \EE \left[ \left\{ \by^\top \mathbf{S}(0,\bX)\right\}^2 \lambda(0,\bX;\btheta)\right] \\
&\ge \EE \left[ \left\{ \by^\top \mathbf{S}(0,\bX)\right\}^2 \lambda(0,\bX;\btheta)
\mathbf{1}_{E_0}(\bX)\right].
\end{align*}
The trick is to observe first that $\left\{ \by^\top \mathbf{S}(0,\bX)\right\}^2 \lambda(0,\bX;\btheta)\mathbf{1}_{E_0}(\bX) = y_0^2\mathrm{e}^{\theta_0} \mathbf{1}_{E_0}(\bX)$ and second that $\mathrm{P}(E_0)>0$ since in particular $\phi(r+\delta)>0$ for any $r\in [0,R]$. These different arguments yield $y_0=0$. \\
\underline{Step 2: Proof that $\by_{-0}=0$}. We are reduced to show that $\EE\left[ \left\{ \by_{-0}^\top \mathbf{S}_{-0}(0,\bX)\right\}^2 \lambda(0,\bX;\btheta)\right]=0$ implies that $\by_{-0}=0$. Let $r\in [\delta+\varepsilon,R+\delta-\varepsilon]$ for some (small) $\varepsilon>0$ and define 
\[
	E_{r,\varepsilon} =\{\bx \in N_{lf}:  N(\mathcal B(0))=N(\beta(r,\varepsilon))
	=1\}
\]
where $\beta(r,\varepsilon)=B((0,r),\varepsilon)$. The set $E_{r,\varepsilon}$ is thus the set of locally finite configurations of points with at most one point close to the $d$-dimensional point $(r,0,\dots,0)^\top$ in $\mathcal B(0)$. Now, since  any $\bx \in E(r,\varepsilon)$ contains at most one point in $\mathcal B(0)$, we have 
\begin{align*}
&\left\{ \by_{-0}^\top \mathbf{S}_{-0}(0,\bX)\right\}^2 \lambda(0,\bX;\btheta) \mathbf{1}_{E_{r,\varepsilon}}(\bX)\\
&= \sum_{j,k=1}^\infty y_jy_k\sum_{u \in \bX \cap \beta(r,\varepsilon)} \varphi_j(\|u\|-\delta) \sum_{v\in \bX\cap \beta(r,\varepsilon)} \varphi_k(\|v\|-\delta) \prod_{w\in \bX \cap \beta(r,\varepsilon)} \phi_{\btheta}(\|w\|)\mathbf{1}_{E_{r,\varepsilon}}(\bX)\\ 
&= \sum_{j,k=1}^\infty y_jy_k\sum_{u \in \bX \cap \beta(r,\varepsilon)} \varphi_j(\|u\|-\delta) \varphi_k(\|u\|-\delta) \phi_{\btheta}(\|u\|)\mathbf{1}_{E_{r,\varepsilon}}(\bX) \\
&= \sum_{u\in \bX \cap \beta(r,\varepsilon)} \{\by_{-0}^\top \boldsymbol{\varphi}(\|u\|-\delta)\}^2 \phi_{\btheta}(\|u\|) \mathbf{1}_{E_{r,\varepsilon}}(\bX)\\
&= \{\by_{-0}^\top \boldsymbol{\varphi}(r-\delta)\}^2 \phi_{\btheta}(r) \mathbf{1}_{E_{r,\varepsilon}}(\bX)+ R(\bX)
\end{align*}
with 
\[
	R(\bX) = \left\{\sum_{u\in \bx \cap \beta(r,\varepsilon)} \{\by_{-0}^\top \boldsymbol{\varphi}(\|u\|-\delta)\}^2 \phi_{\btheta}(\|u\|) -\{\by_{-0}^\top \boldsymbol{\varphi}(r-\delta)\}^2 \phi_{\btheta}(r) \right\}\mathbf{1}_{E_{r,\varepsilon}}(\bX).
\]
Since $|R(\bX)| \le \sup_{u\in \beta(r,\varepsilon)} |\{\by_{-0}^\top \boldsymbol{\varphi}(\|u\|-\delta)\}^2 \phi_{\btheta}(\|u\|) -\{\by_{-0}^\top \boldsymbol{\varphi}(r-\delta)\}^2 \phi_{\btheta}(r) |$, it is clear that $R(\bX)$ tends to 0 in $L^1$ as $\varepsilon \to 0$. Therefore, there exists $\varepsilon_1$ such that for any $\varepsilon<\varepsilon_1$
\begin{align*}
0 &= 	\EE\left[ \left\{ \by_{-0}^\top \mathbf{S}_{-0}(0,\bX)\right\}^2 \lambda(0,\bX;\btheta)\right] \\
&\geq \EE\left[ \left\{ \by_{-0}^\top \mathbf{S}_{-0}(0,\bX)\right\}^2 \lambda(0,\bX;\btheta) \mathbf 1_{E_{r,\varepsilon}}(\bX)\right] \\
&=\mathrm{P}(E_{r,\varepsilon}) \{\by_{-0}^\top \boldsymbol{\varphi}(r-\delta)\}^2 \phi_{\btheta}(r) + \EE[R(\bX)] \\
&\ge \frac12 \mathrm{P}(E_{r,\varepsilon}) \{\by_{-0}^\top \boldsymbol{\varphi}(r-\delta)\}^2 \phi_{\btheta}(r) .
\end{align*}
Multiplying by $w(r-\delta)$ and integrating over $[\delta+\varepsilon,R+\delta-\varepsilon]$ leads to
\begin{align*}
0 &= 	\int_{\delta+\varepsilon}^{R+\delta-\varepsilon} \left\{ \by_{-0}^\top \boldsymbol{\varphi}(r-\delta)\right\}^2 \lambda(0,\bX;\btheta) w(r-\delta)\phi_{\btheta}(r)\mathrm{P}(E_{r,\varepsilon}) \mathrm d r  \\
&\ge \inf_{r\in [\varepsilon,R-\varepsilon]} \phi_{\btheta}(r+\delta)\mathrm{P}(E_{r,\varepsilon}) \int_{\varepsilon}^{R-\varepsilon} \left\{ \by_{-0}^\top \boldsymbol{\varphi}(r)\right\}^2 w(r) \mathrm{d}r.
\end{align*}
Using the orthonormality of the basis function
\[
\int_{\varepsilon}^{R-\varepsilon} \left\{ \by_{-0}^\top \boldsymbol{\varphi}(r)\right\}^2 w(r) \mathrm{d}r \le \int_{0}^{R} \left\{ \by_{-0}^\top \boldsymbol{\varphi}(r)\right\}^2 w(r) \mathrm{d}r = \|\by_{-0}\|^2.	
\]
Hence, using the dominated convergence theorem, as $\varepsilon\to 0$
\[
\int_{\varepsilon}^{R-\varepsilon} \left\{ \by_{-0}^\top \boldsymbol{\varphi}(r)\right\}^2 w(r) \mathrm{d}r \to  \int_{0}^{R} \left\{ \by_{-0}^\top \boldsymbol{\varphi}(r)\right\}^2 w(r) \mathrm{d}r = \|\by_{-0}\|^2.
\]
In other words there exists $\varepsilon_2>0$ such that for any $\varepsilon<\min(\varepsilon_1,\varepsilon_2)$
\begin{equation}\label{eq:step2}
0 \ge \inf_{r\in [\varepsilon,R-\varepsilon]} \mathrm{P}(E_{r,\varepsilon}) \inf_{r\in [0,R]} \phi_{\btheta}(r+\delta) \frac{\|\by_{—0}\|^2}{2} \ge 0.
\end{equation}
	
By condition \ref{C:model}, we have already mentioned (see Section~\ref{sec:result}) that $\inf_{r\in [0,R]} \phi_{\btheta}(r+\delta)>0$. Moreover, for some fixed value of $0<\varepsilon<\min(\varepsilon_1,\varepsilon_2)$, $\mathrm{P}(E_{r,\varepsilon})>0$. These arguments and~\eqref{eq:step2} yield $\by_{-0}=0$ and thus $\by=0$.
\item Let ${\by}_{K_n}\in \mathbb R^{K_n+1}$ be such that $\|\by_{K_n}\|=1$ and let $\tilde \by=({\by}_{K_n}^\top,0^\top)^\top$. Since from (i)
\begin{align*}
\by_{K_n}^\top \widebar{\mathbf{A}}_{K_n}(\btheta_{K_n}^{\star}) \by_{K_n}&
= {\tilde{\by}}^\top  \widebar{\mathbf{A}}(\btheta_{K_n}^{\star}) \tilde \by\ge \underset{\substack { \by \in \mathbb{R}^{\mathbb{N}}, \| \by \| =1, \\ \btheta \in \mathcal{S}_m }}{\inf} \by^\top  \widebar{\mathbf{A}}(\btheta) \by>0
\end{align*}  
we deduce that
\begin{align*}
\| \widebar{\mathbf{A}}_{K_n}^{-1}(\btheta_{K_n}^{\star}) \|  = \left \{ \underset{\substack { \by_{K_n} \in \mathbb{R}^{K_n+1} \\ \| \by_{K_n} \| =1}} {\inf} \, \by_{K_n}^\top \widebar{\mathbf{A}}_{K_n}(\btheta_{K_n}^{\star}) \by_{K_n} \right \}^{-1} =O(1).
\end{align*} 
\item Let $\mathcal{F}_0=\{ \bx \in N_{lf}: N(\mathcal{B}_2)=0 \}$ where $\mathcal{B}_2(u)$ is the ball (or annulus) centered at $u$ with radii $\delta$ and $2(R+\delta)$. It is clear that $\pi_0=P(\mathcal{F}_0)>0$. For all $v \in \mathcal B(0) \cup \{ \mathrm{0} \}$, we have
\begin{align*}
\lambda(v,\bX;\btheta^\star) \, \mathbf{1}_{\mathcal{F}_0}(\bX) = & \, e^{\theta_0^\star}  \, \mathbf{1}_{\mathcal{F}_0}(\bX) \\
\by^\top \bS(v,\bX) \, \mathbf{1}_{\mathcal{F}_0}(\bX) = & \, y_0 \, \mathbf{1}_{\mathcal{F}_0}(\bX).
\end{align*}
Thus with similar arguments used in (i), we have $\by^\top \left \{ \widebar{\mathbf{A}}(\btheta^{\star}) + \widebar{\mathbf{B}}(\btheta^{\star}) \right \} \by \geq T_1 + T_2 + T_3$ where
\begin{align*}
T_1 = & \, \pi_0 \, y_0^2 \, e^{\theta_0^\star}, \\
T_2 = & \, \pi_0 {\int_{\mathcal{B}(0)} y_0^2 \, e^{2\theta_0^\star} \left(1 -\phi(\|v\|) \right) \mathrm{d}v } \\
       = & \, \pi_0 \, y_0^2 \, e^{2\theta_0^\star} \, | \sigma_{d} | {\int_{\delta}^{R+\delta} (1 - \phi(r) ) \, r^{d-1} \, \mathrm{d}r }, \\
T_3 = & \, \pi_0 \, \int_{\mathcal B(0)} \left \{ 
\by_{-0}^\top {\boldsymbol{\varphi}}(\|v\|-\delta) \right \}^2 
\, e^{2\theta_0^\star} \phi(\|v\|) 
\mathrm{d}v  \\
       = & \, \pi_0 \, |\sigma_{d}| \,  e^{2\theta_0^\star} {\int_{\delta}^{R+\delta}  \left \{ \by_{—0}^\top {\boldsymbol{\varphi}}(r-\delta) \right \}^2 \, \phi(r) \, r^{d-1} \,  \mathrm{d}r }.
\end{align*}
By the definition of $\psi$ in condition \ref{C:psi}, $T_1+T_2= \pi_0 \, y_0^2 \, e^{\theta_0^\star} \, \psi$. The term $T_3$ can be rewritten as
\begin{align*}
T_3 = & \, \pi_0 \, |\sigma_{d}| \,  e^{2\theta_0^\star} {\int_{0}^{R}  \left \{ \by_{-0}^\top {\boldsymbol{\varphi}}(r) \right \}^2 w(r) \, \phi(r) \, \frac{(r+\delta)^{d-1}}{w(r)} \, \mathrm{d}r }. \\
\end{align*} 
Now, let $\by$ be such that $\by^\top \left \{ \widebar{\mathbf{A}}(\btheta^{\star}) + \widebar{\mathbf{B}}(\btheta^{\star}) \right \} \by=0$, then by~condition \ref{C:psi}, there exists $c>0$ such that
\[
	0 =\by^\top \left \{ \widebar{\mathbf{A}}(\btheta^{\star}) + \widebar{\mathbf{B}}(\btheta^{\star}) \right \} \by \ge c y_0^2 + c{\int_{0}^{R}  \left \{ \by_{-0}^\top {\boldsymbol{\varphi}}(r) \right \}^2 w(r) \, \tilde \phi(r) \mathrm{d}r } \ge 0
\]
where $\tilde \phi(r)= (r+\delta)^{d-1} \phi(r) /w(r)$.
We deduce on the one hand that $y_0=0$ and on the other hand that for any $r \in [0,R]$, $\left \{ \by_{-0}^\top {\boldsymbol{\varphi}}(r) \right \}^2 w(r) \, \tilde \phi(r)=0$. In particular, over $\Delta=\{r\in [0,R]: \tilde \phi(r)>0\}$, $\left \{ \by_{-0}^\top {\boldsymbol{\varphi}}(r) \right \}^2 w(r) =0$. If we integrate this over $\Delta$ which by~condition \ref{C:psi} is such that $[0,R]\setminus \Delta$ has zero Lebesgue measure, we obtain 
\begin{align*}
\int_{\Delta} 	\left \{ \by_{-0}^\top {\boldsymbol{\varphi}}(r) \right \}^2 w(r) \mathrm d r &=
\int_{0}^R 	\left \{ \by_{-0}^\top {\boldsymbol{\varphi}}(r) \right \}^2 w(r) \mathrm d r = \|\by_{-0}\|^2 =0. 
\end{align*}
Hence $\by=0$, which ends the proof of (iii).

\item From the GNZ formula (\ref{gnz}), we have $\EE \, \left[ \boldsymbol{e}_{K_n}(\bX;\btheta^{\star}) \right] = 0$ whereby we deduce that $\EE \left[ \by_{K_n}^\top \boldsymbol{e}_{K_n}(\bX;\btheta^{\star}) \right] = 0$. Let $\bSigma_{K_n}(\btheta^{\star})=  \mathbf{A}_{K_n}(\btheta^{\star}) + \mathbf{B}_{K_n}(\btheta^{\star})$. Since by Lemma~\ref{lemma1} for $n$ large enough,  ${\displaystyle  |W_n|^{-1} \by_{K_n}^\top \bSigma_{K_n}(\btheta^{\star}) \, \by_{K_n} \geq \frac12 \inf_{\by \in  \mathcal{S}_m }  \by^\top \widebar{\bSigma}(\btheta^{\star}) \, \by }$, we deduce from Lemma~\ref{lemma2}(iii) ${ \liminf_{n\to \infty} \;   |W_n|^{-1} \by_{K_n}^\top \bSigma_{K_n}(\btheta^{\star}) \, \by_{K_n} >0}$. We can now apply a central limit theorem (CLT) to $\by_{K_n}^\top \boldsymbol{e}_{K_n}(\bX;\btheta^{\star})$ proved by \cite{coeurjolly2017parametric} for innovation type statistics (condition \ref{C:model} is a particular case and implies that the sequences $\alpha_n$ and $R_n$ in \cite{coeurjolly2017parametric} are such that $\alpha_n=R_n=R+\delta$).

\item Let $\boldsymbol{\Psi}_{K_n}(\bX)=\mathbf{A}_{K_n}(\bX;\btheta) - \mathbf{A}_{K_n}(\bX;\btheta^\prime)$. We have 
\begin{align*}
\boldsymbol{\Psi}_{K_n}(\bX) = {\int_{W_n} \bS_{K_n}(u, \bX) \bS_{K_n}(u, \bX)^\top  \left \{ \lambda (u,\bX;\btheta) - \lambda (u,\bX;\btheta^\prime) \right \} \mathrm{d}u}.
\end{align*}
Using first-order Taylor expansion, there exists $s \in (0,1)$ such that 
\[
\lambda (u,\bX;\btheta) - \lambda (u,\bX;\btheta^\prime) = (\btheta - \btheta^\prime)^\top \bS_{K_n}(u, \bX) \lambda (u,\bX;\check{\btheta})
\]
where $\check{\btheta} = \btheta + s (\btheta^\prime - \btheta)$. Hence, 
\begin{align*}
\boldsymbol{\Psi}_{K_n}(\bX)  =  {\int_{W_n}  \bS_{K_n}(u, \bX) \bS_{K_n}(u, \bX)^\top (\btheta - \btheta^\prime)^\top \bS_{K_n}(u, \bX) \lambda (u,\bX;\check{\btheta}) \mathrm{d}u}.
\end{align*}
By~\eqref{ls},
\begin{align*}
\| \boldsymbol{\Psi}_{K_n}(\bX) \| \leq c\| \btheta - \btheta^\prime \|   {\int_{W_n} \| \bS_{K_n}(u, \bX) \|^3  \mathrm{d}u}.
\end{align*}
Using Lemma~\ref{lem:Esk} and the ergodic theorem (see e.g. proof of Lemma~\ref{lemma1}) we have
\begin{align*}
 {\int_{W_n} \| \bS_{K_n}(u, \bX) \|^3  \mathrm{d}u} = O_{a.s.}(K_n^{m+3/2} \, |W_n|)
\end{align*}
whereby we deduce that $\| \boldsymbol{\Psi}_{K_n}(\bX) \| = O_\mathrm{P}(K_n^{m+3/2}|W_n| \, \|\btheta-\btheta^\prime\|)$. 
\item Following the proof of Lemma~\ref{lemma1}(i), we could also prove that 
\[
 \xi_n(\bX):=	|W_n|^{-1}\sup_{\substack { \by \in \, \mathcal{S}_m \\ \| \by \| =1}}  \left \{ \by^\top  \left(
	{\mathbf{A}}(\bX;\btheta^\star) -{\mathbf{A}}(\btheta^\star) 
	\right)
	\by \right \} \to 0
\]
almost surely as $n\to \infty$. 
\end{enumerate}
\end{proof}

\section{Consistency of the estimating equation $\boldsymbol{e}_{K_n}(\bX;\btheta_{K_n}^{\star})$}  \label{sec:prop1}

\begin{proposition} 
 \label{PROP:BIAS} 
Let
\begin{align*}
	\boldsymbol{e}_{K_n}(\bX;\btheta_{K_n}^{\star}) = \boldsymbol{e}_{K_n}(\bX;\btheta^{\star}) + \boldsymbol{\delta}_{K_n}(\bX,\btheta^\star)
\end{align*}
where
\begin{equation}\label{eq:bdelta}
\boldsymbol{\delta}_{K_n}(\bX,\btheta^\star) = 	
\int_{W_n} \bS_{K_n}(u,\bX) \lambda(u,\bX;\btheta^{\star}) r_{K_n}(u,\bX) \mathrm{d}u 
\end{equation}
and where $r_{K_n}$ is given by~\eqref{eq:rKn}. Under the assumptions of Theorem~\ref{THM:CONST}(i), as $n\to~\infty$
\begin{align*}
\boldsymbol{e}_{K_n}(\bX;\btheta^{\star})&=  O_{\mathrm{P}}({K_n^{1/2}|W_n|^{1/2}}) 	\\
\boldsymbol{\delta}_{K_n}(\bX,\btheta^\star) &= O_{\mathrm{P}}(a_{K_n} K_n^{1/2}|W_n|+ a_{K_n}K_n^{(m+1)/2}|W_n|^{1/2}) 
\end{align*}
whereby we deduce that
\begin{align}
\boldsymbol{e}_{K_n}(\bX;\btheta_{K_n}^{\star}) = O_{\mathrm P} \left( {K_n^{(m\vee 1)/2}|W_n|^{1/2}} +a_{K_n} K_n^{1/2}|W_n| \right) \label{eq:score}.
\end{align}
\end{proposition}

In the following $c$ may stand for a generic non-negative constant independent of $\btheta, \btheta^\star, \bX, n,\dots$ and which may vary from line to line. 

\begin{proof}

We split the proof into two parts focusing on the control of $\boldsymbol{e}_{K_n}(\bX;\btheta^{\star})$ and $ \boldsymbol{\delta}_{K_n}(\bX,\btheta^\star)$.

\noindent\underline{Control of $\boldsymbol{e}_{K_n}(\bX;\btheta^{\star})$}.
The random vector $ \boldsymbol{e}_{K_n}(\bX;\btheta^{\star})$, which is a specific innovation statistic sudied in particular by~\cite{baddeley2005residual,coeurjolly2013fast}, is centered with variance given by $\Var [ \boldsymbol{e}_{K_n}(\bX;\btheta^{\star})] = |W_n| (\mathbf{\bar A}_{K_n}(\btheta^\star) + \mathbf{\bar B}_{K_n}(\btheta^\star))$. From the stationarity of the Gibbs point process~\eqref{ls}, the boundedness of $\phi_{\btheta}$ for any $\btheta$ and Lemma~\ref{lem:Esk}
\begin{align*}
\|\mathbf{\bar A}_{K_n}(\btheta^\star)\| & \le \bar \lambda \, \EE	\left[ \|\mathbf{S}_{K_n}(0,\bX)\|^2\right] = O(K_n) \\
\|\mathbf{\bar B}_{K_n}(\btheta^\star)\| & \le 
c   \int_{\mathcal B(0)} \EE \left[ \|\mathbf{S}_{K_n}(0,\bX)\| \|\mathbf{S}_{K_n}(v,\bX)\|\right] \mathrm dv+ 
c \int_{\mathcal B(0)} \|\tilde{\boldsymbol{\varphi}}(\|v\|-\delta)\|^2 \mathrm dv \\
&= O\left( \EE	\left[ \|\mathbf{S}_{K_n}(0,\bX)\|^2\right]\right) 
+ O \left( \sum_{k=1}^{K_n} \int_0^R \varphi_k(r)^2 w(r) \frac{(r+\delta)^{d-1}}{w(r)}\mathrm dr\right) =O(K_n) 
\end{align*}
whereby we deduce that $\Var [ \boldsymbol{e}_{K_n}(\bX;\btheta^{\star})] = O(K_n |W_n|)$ and that $\boldsymbol{e}_{K_n}(\bX;\btheta^{\star})=O_{\mathrm{P}}(\sqrt{K_n|W_n|})$. \\

\noindent\underline{Control of $\boldsymbol{\delta}_{K_n}(\bX;\btheta^{\star})$}.
First, $\EE\left[ \boldsymbol{\delta}_{K_n}(\bX;\btheta^{\star})\right]= |W_n|\EE \left[ \mathbf S_{K_n}(0,\bX) r_{K_n}(0,\bX) \lambda(0,\bX; \btheta^\star\right]$. From~\eqref{ls}, Cauchy-Schwarz inequality and Lemma~\ref{lem:Esk}, we have
\begin{align}
\|\EE\left[ \boldsymbol{\delta}_{K_n}(\bX;\btheta^{\star})\right]\| & \le c|W_n| \EE\left[ \|\mathbf S_{K_n}(0,\bX) \| \, |r_{K_n}(0,\bX)|\right]\nonumber\\
& \le c|W_n| \EE\left[ \|\mathbf S_{K_n}(0,\bX) \|^2\right]^{1/2} \EE\left[ r_{K_n}(0,\bX)^2\right]^{1/2} \nonumber\\
&= O( |W_n| a_{K_n} K_n^{1/2} ). \label{eq:biasdelta}
\end{align}
Second, $\Var \left[ \boldsymbol{\delta}_{K_n}(\bX;\btheta^{\star})\right] = T_1-T_2$ with $T_1=\EE\left[ \boldsymbol{\delta}_{K_n}(\bX;\btheta^{\star})\boldsymbol{\delta}_{K_n}(\bX;\btheta^{\star})^\top\right]$ and $T_2=\EE\left[ \boldsymbol{\delta}_{K_n}(\bX;\btheta^{\star})\right]\EE\left[ \boldsymbol{\delta}_{K_n}(\bX;\btheta^{\star})\right]^\top$. Equation~\eqref{eq:biasdelta} yields $\|T_2\|=O(\|\EE[\boldsymbol{\delta}_{K_n}(\bX;\btheta^{\star})]\|^2)=O(|W_n|^2a_{K_n}^2K_n)$. Regarding the term $T_1$, using the stationarity of the Gibbs point process~\eqref{fr}, we have for $n$  large enough (ensuring that $\mathcal B(0)\subset W_n$)
\begin{align*}
T_1 &= 
\EE \left[ \int_{W_n}\int_{W_n} \mathbf{S}_{K_n}(u,\bX)\mathbf{S}_{K_n}(v,\bX)^\top 
r_{K_n}(u,\bX)r_{K_n}(v,\bX) 
\lambda(u,\bX;\btheta^\star)\lambda(v,\bX;\btheta^\star)
\right] \mathrm du\mathrm dv\\
&= |W_n| \int_{\mathcal B(0)} 
\EE \left[ 
\mathbf{S}_{K_n}(0,\bX)\mathbf{S}_{K_n}(v,\bX)^\top 
r_{K_n}(0,\bX)r_{K_n}(v,\bX) 
\lambda(0,\bX;\btheta^\star)\lambda(v,\bX;\btheta^\star)
\right] \mathrm dv.
\end{align*}
Now, using Hölder's inequality and again the stationarity of the Gibbs point process
\begin{align*}
\|T_1\| & \le c|W_n| \EE \left[ \|\mathbf S_{K_n}(0,\bX\|^4\right]^{1/2}  
\EE \left[ r_{K_n}(0,\bX)^4\right]^{1/2}	\\
&= O(|W_n| K_n^{(2m+2)/2} (a_{K_n}^4)^{1/2}) = O(|W_n| a_{K_n}^2 K_n^{m+1}). 
\end{align*}
Hence
\[
	\Var \left[ \boldsymbol{\delta}_{K_n}(\bX;\btheta^{\star})\right] = O 
	\left( 
	a_{K_n}^2 K_n |W_n|^2   + a_{K_n}^2 K_n^{m+1} |W_n| 
	\right). 
\]
The control in probability of $\boldsymbol{\delta}_{K_n}(\bX;\btheta^{\star})$ is easily deduced whereby we deduce that
\begin{align*}
	\boldsymbol{e}_{K_n}(\bX;\btheta^{\star}_{K_n})&=O_{\mathrm{P}}\left(K_n^{1/2}|W_n|^{1/2}
	+a_{K_n} K_n^{1/2} |W_n|   + a_{K_n} K_n^{(m+1)/2} |W_n|^{1/2}
	\right)\\
	&=O_{\mathrm{P}}\left( |W_n| \left(
	\left(\frac{K_n}{|W_n|}\right)^{1/2}
	+a_{K_n} K_n^{1/2}    + a_{K_n} K_n^{1/2} \left(\frac{K_n^m}{|W_n|}\right)^{1/2}
	\right)\right)\\
	&=O_{\mathrm{P}}\left(K_n^{(m\vee 1)/2}|W_n|^{1/2}
	+a_{K_n} K_n^{1/2} |W_n|  	\right)
\end{align*}
if $a_{K_n}K_n^{1/2}\to 0$ and $K_n^{m\vee 1}|W_n|^{-1}\to 0$ as $n\to \infty$.
\end{proof}

\section{Proof of Theorem~\ref{THM:CONST}}  \label{sec:theorem1}
\begin{proof}
\begin{enumerate}[(i)]
\item Using first-order Taylor expansion, there exists $t \in (0,1)$ such that
\begin{align*}
\boldsymbol{e}_{K_n} \big( \bX;\hat{\btheta}_{K_n} \big) - \boldsymbol{e}_{K_n} \big(\bX;\btheta_{K_n}^{\star} \big) = \boldsymbol{e}_{K_n}^{(1)} \big( \bX;\widetilde{\btheta}_{K_n} \big) \big( \hat{\btheta}_{K_n} - \btheta_{K_n}^{\star} \big),
\end{align*} 
where $\widetilde{\btheta}_{K_n}  = \btheta_{K_n}^{\star} + \, t \big( \hat{\btheta}_{K_n} - \btheta_{K_n}^{\star} \big)$.
Since $\hat{\btheta}_{K_n}=\underset{\btheta_{K_n} \in \mathcal{S}_m }{\text{argmax}} \, \mathrm{LPL}_n(\bX;\btheta_{K_n})$ and \linebreak $\boldsymbol{e}_{K_n}^{(1)} (\bX;\btheta_{K_n}^{\star})=-\mathbf{A}_{K_n}(\bX;\btheta_{K_n}^{\star})$, we therefore have
\begin{align}
\boldsymbol{e}_{K_n}(\bX;\btheta_{K_n}^{\star})  =  \mathbf{A}_{K_n}(\bX;\widetilde{\btheta}_{K_n}) \big( \hat{\btheta}_{K_n} - \btheta_{K_n}^{\star} \big). \label{eq:defthetaEst}
\end{align}
Rewrite the last equation as $L=R$ where 
\begin{align*}
L = & \, \bt_n^\top \mathbf{A}_{K_n}(\bX;\widetilde{\btheta}_{K_n}) \, \bt_n  \cdot  \| \hat{\btheta}_{K_n} - \btheta_{K_n}^{\star} \|^2 \\
R = & \, (\hat{\btheta}_{K_n} - \btheta_{K_n}^{\star})^\top \boldsymbol{e}_{K_n}(\bX;\btheta_{K_n}^{\star})
\end{align*}
with $\bt_n = ( \hat{\btheta}_{K_n} - \btheta_{K_n}^{\star} ) / \| \hat{\btheta}_{K_n} - \btheta_{K_n}^{\star} \|$. We have that $\bt_n \in \mathcal{S}_m$ and $\| \bt_n \| = 1$. Note that
\begin{align*}
|W_n|^{-1} \bt_n^\top  \mathbf{A}_{K_n}(\bX;\widetilde{\btheta}_{K_n}) \, \bt_n \geq \underset{\substack { \by_{K_n}, \btheta \in \mathcal{S} \\ \| \by_{K_n} \| =1}}{\inf} \, |W_n|^{-1} \by_{K_n}^\top  \mathbf{A}_{K_n}(\bX;\btheta) \by_{K_n}.
\end{align*}
By Lemma~\ref{lemma1}(i)-(ii), for $n$ large enough a.s.
\begin{align*}
\underset{\substack { \by_{K_n}, \btheta \in \mathcal{S}_m \\ \| \by_{K_n} \| =1}}{\inf} \, |W_n|^{-1} \by_{K_n}^\top  \mathbf{A}_{K_n}(\bX;\btheta) \by_{K_n} &\geq \frac12 \,\underset{\substack { \by, \btheta \in \mathcal{S}_m \\ \| \by \| =1}}{\inf} \,  \by^\top  \widebar{\mathbf{A}}(\btheta) \, \by \\
&\ge 
\frac12 \,\underset{\substack { \by \in \mathbb{R}^{\mathbb{N}}, \| \by \| =1 \\ \btheta \in \mathcal{S}_m}}{\inf} \,  \by^\top  \widebar{\mathbf{A}}(\btheta) \, \by 
=: \frac1c > 0
\end{align*}
where the latter inequality ensues from Lemma~\ref{lemma2}(i). Hence, for $n$ large enough $ L \geq c^{-1} \, |W_n| \,  \| \hat{\btheta}_{K_n} - \btheta_{K_n}^{\star} \|^2$ almost surely. On the other hand, we have  \linebreak $| R | \leq  \| \hat{\btheta}_{K_n} - \btheta_{K_n}^{\star} \| \, \| \boldsymbol{e}_{K_n}(\bX;\btheta_{K_n}^{\star}) \|$ whereby we deduce that  
\begin{align*}
\| \hat{\btheta}_{K_n} - \btheta_{K_n}^{\star} \| \leq c \, |W_n|^{-1} \, \| \boldsymbol{e}_{K_n}(\bX;\btheta_{K_n}^{\star}) \|.
\end{align*}
The results are easily deduced from Proposition~\ref{PROP:BIAS}.

(ii) We start first with the following equation
\[
	\mathbf{A}_{K_n}(\btheta_{K_n}^\star)(\hat{\btheta}_{K_n} - \btheta_{K_n}^{\star}) - \boldsymbol{e}_{K_n}(\bX; \btheta^{\star})= \sum_{i=1}^4 T_i
\]
where
\begin{align*}
T_1 &= (\mathbf{A}_{K_n}(\btheta_{K_n}^\star)  - \mathbf{A}_{K_n}(\bX;\btheta_{K_n}^\star)  - ) (\hat{\btheta}_{K_n} - \btheta_{K_n}^{\star}) \\
T_2 &= (\mathbf{A}_{K_n}(\bX;\btheta_{K_n}^\star)  - \mathbf{A}_{K_n}(\bX;\widetilde{\btheta}_{K_n})) (\hat{\btheta}_{K_n} - \btheta_{K_n}^{\star})\\
T_3 &= \mathbf{A}_{K_n}(\bX;\widetilde{\btheta}_{K_n}) (\hat{\btheta}_{K_n} - \btheta_{K_n}^{\star}) - \boldsymbol{e}_{K_n}(\bX; \btheta^{\star}_{K_n})\\
T_4 &=	\boldsymbol{e}_{K_n}(\bX; \btheta^{\star}_{K_n}) - \boldsymbol{e}_{K_n}(\bX; \btheta^{\star}).
\end{align*}
From Lemma~\ref{lemma2}(vi), $\|T_1\|=o_{\mathrm P}(|W_n| \|\hat{\btheta}_{K_n} - \btheta_{K_n}^{\star})\|)= o_{\mathrm P}(a_{K_n} K_n^{1/2}|W_n|+K_n^{(m\vee 1)/2}|W_n|^{1/2} )$. Lemma~\ref{lemma2}(v) and Theorem~\ref{THM:CONST}(i) show that 
\begin{align*}
T_2&=O_{\mathrm{P}}\left( |W_n|K_n^{m+3/2} \left( K_n^{(m\vee 1)/2}{|W_n|}^{-1/2}+a_{K_n}{K_n}^{1/2}\right)^2\right)	\\
&=O_{\mathrm{P}}\left( K_n^{m+m\vee 1+3/2}+   a_{K_n}^2K_n^{m+5/2} |W_n| \right).	
\end{align*}
By definition $T_3=0$ and by Proposition~\ref{PROP:BIAS}, $T_4=\boldsymbol{\delta}_{K_n}(\bX;\btheta^\star)=O_{\mathrm{P}}(a_{K_n}K_n^{1/2}|W_n|)$. Hence,
\[
	\sum_i T_i = O_{\mathrm P}\left( 
K_n^{(m\vee 1)/2} |W_n|^{1/2}+K_n^{m+m\vee 1+3/2}+a_{K_n} K_n^{1/2} |W_n| + a_{K_n}^2K_n^{m+5/2}|W_n|
	\right).
\]
Hence, using Lemma~\ref{lemma2}(ii) we obtain
\begin{align}
\hat{\btheta}_{K_n} - \btheta_{K_n}^{\star} =& \mathbf{A}_{K_n}^{-1}(\btheta_{K_n}^\star)  \boldsymbol{e}_{K_n}(\bX; \btheta^{\star})  + O_\mathrm{P}(x_n) \label{eq:tmp}
\end{align}
with
\begin{align}
\label{eq:xn}
x_n &= 
\frac{K_n^{(m\vee 1)/2}}{|W_n|^{1/2}} + \frac{K_n^{m+m\vee 1+3/2}}{|W_n|} + a_{K_n}K_n^{1/2} + a_{K_n}^2 K_n^{m+5/2} \nonumber \\
&= O \left( \frac{K_n^{m+m\vee 1+3/2}}{|W_n|} + a_{K_n}K_n^{1/2} + a_{K_n}^2 K_n^{m+5/2} \right).
\end{align}
Now, let $r \in (\delta, R+\delta)$. Using~\eqref{eq:tmp} and since $\|\widetilde{\boldsymbol{\varphi}}_{K_n}(r - \delta)\|=\sqrt{K_n}$ , we have
\begin{align*} 
\hat{g}_n(r;K_n) - g(r) = & \sum_{k=1}^{K_n} (\hat{\theta}_{k,n}-\theta_k^{\star}) \varphi_k(r-\delta) - \sum\limits_{k>K_n} \theta_k^{\star} \varphi_k(r-\delta) \\
= &  \, \widetilde{\boldsymbol{\varphi}}_{K_n}(r - \delta)^\top \left( \hat{\btheta}_{K_n} - \btheta_{K_n}^{\star} \right) - \sum\limits_{k>K_n} \theta_k^{\star} \varphi_k(r-\delta) \\
= & \,  \widetilde{\boldsymbol{\varphi}}_{K_n}(r - \delta)^\top \mathbf{A}_{K_n}^{-1}(\btheta_{K_n}^\star)  \boldsymbol{e}_{K_n}(\bX; \btheta^{\star}) + O_\mathrm{P}(\sqrt{K_n} x_n) + O(a_{K_n})
\end{align*}
Let $\bz=\widetilde{\boldsymbol{\varphi}}_{K_n}(r - \delta)^\top  \mathbf{A}_{K_n}^{-1}(\btheta_{K_n}^\star)$. From Lemma~\ref{lemma2}(ii) and Cauchy-Schwarz inequality, we have
\[
\sum_{k=0}^{K_n} k^m \, |z_k| \leq K_n^{m+1/2} \| \bz \| 
\le c \| \widetilde{\boldsymbol{\varphi}}_{K_n}(r - \delta)^\top \| \, \| \mathbf{A}_{K_n}^{-1}(\btheta_{K_n}^\star)\|
= O(K_n^{m+1} |W_n|^{-1}) =o(1)\]
by assumption. Hence, $\bz \in \mathcal{S}_m$ and by Lemma~\ref{lemma2}(iv),
\[
s_n^{-1}(r) \,\widetilde{\boldsymbol{\varphi}}_{K_n}(r - \delta)^\top \mathbf{A}_{K_n}^{-1}(\btheta_{K_n}^\star)  \boldsymbol{e}_{K_n}(\bX; \btheta^{\star}) \xrightarrow{d} \mathcal{N}(0,1)
\]
where $s_n(r)$ is given by (\ref{tauxn}). Now, take $\by_{K_n} =  \mathbf{A}_{K_n}^{-1}(\btheta_{K_n}^\star) \, \widetilde{\boldsymbol{\varphi}}_{K_n}(r - \delta)^\top$. We then deduce from Lemma~\ref{lemma2}(iii) that ${\displaystyle \liminf_{n\to \infty} \;   |W_n|^{-1} \Var \left( \by_{K_n}^\top \boldsymbol{e}_{K_n}(\bX; \btheta^{\star}) \right) >0}$, i.e. \linebreak${ \liminf_{n\to \infty} \;   |W_n|^{-1} s_n^{2}(r) >0}$ which implies that $s_n^{-1}(r) = O(|W_n|^{-1/2})$. Finally
\begin{align*}
s_{n}^{-1}(r) \sqrt{K_n} x_n &= 
 \frac{K_n^{m+m\vee 1+2}}{|W_n|^{3/2}} + a_{K_n} K_n^{1/2} \left( \frac{K_n}{|W_n|}\right)^{1/2} + a_{K_n}^2 \frac{K_n^{m+3}}{|W_n|^{1/2}}.
\end{align*}
The second term tends to 0 by assumptions of Theorem~\ref{THM:CONST}(i). The first and third one tend to 0 by assumption of Theorem~\ref{THM:CONST}(ii). Finally, using Slutsky's Theorem, we conclude that
\[
s_n^{-1}(r) \left \{ \hat{g}_n(r;K_n) - g(r) \right \} \xrightarrow{d} \mathcal{N}(0,1).
\]  
\end{enumerate}
\end{proof}

\section{Version of Theorem~\ref{THM:CONST} when $\sup_{r\in[0,R]} (r+\delta)^{d-1}/w(r)=\infty$}
\label{sec:FB}

As mentioned in the discussion of Condition~\ref{C:anKn}, the point (iii) of this condition is used in Lemma~\ref{lem:Esk} to prove that $\int_0^R \varphi_k(r)^2 \mathrm dr = O(1)$. If $\sup_{r\in[0,R]} (r+\delta)^{d-1}/w(r)=\infty$ (a situation which occurs with the Fourier-Bessel basis and $\delta>0$), then we only have $\int_0^R \varphi_k(r)^2\mathrm d r=O(k^m)$. We claim that in this situation Theorem~\ref{THM:CONST} remains true but with different assumptions. Here are the changes. The proofs are omitted. In Lemma~\ref{lem:Esk}, $\EE(S_k(0,\bX)^p)=O(k^{mp})$ and $\EE(\|\mathbf S_{K_n}(0,\bX)\|^p)=O(k^{mp+p/2})$. Lemmas~\ref{lemma1}-\ref{lemma2} remain unchanged. In Proposition~\ref{PROP:BIAS}, $\mathbf e_{K_n}(\bX;\btheta_{K_n}^\star)= O_{\mathrm{P}}\big( K_n^{m+1/2}|W_n|^{1/2}+a_{K_n}K_n^{m+1/2}|W_n|\big)$. Theorem~\ref{THM:CONST} is reformulated as follows\\
(i) if $K_n^{m+1/2}|W_n|^{-1/2}\to 0$ and $a_{K_n}K_n^{m+1/2} \to 0$ as $n\to \infty$
\begin{align*} 
\mathrm{ISE}(\hat{g}_n(\cdot\,;K_n)) =   O_\mathrm{P} \left( \| {\hat{\btheta}}_{K_n} -\btheta_{K_n}^{\star}\|^2 \right) = O_\mathrm{P} \left( {\frac{K_n^{2m+1}}{|W_n|}} + a_{K_n}^2 K_n^{2m+1} \right) 
\end{align*}
(ii) \eqref{tauxn} is true if $a_{K_n} K_n^{m+1/2}\to 0$, $a_{K_n}^2K_n^{6m+3}|W_n|^{-1/2}\to 0$ and $K_n^{10m/3+2}|W_n|^{-1}\to 0$ as $n\to \infty$.

\section{Pointwise convergence of $\hat g_n(r;K_n)$}

\begin{corollary}
Assume the conditions \ref{C:Wn}-\ref{C:anKn} hold. As $n\to \infty$, assume also that $a_{K_n} K_n\to 0$ and that $K_n^2/|W_n| \to 0$. Then, for $r \in (\delta,R+\delta)$ 
\[ \hat{g}_n(r;K_n) - g(r) 
=O_\mathrm{P} \left( \sqrt{\frac{K_n^{2}}{|W_n|}} +  a_{K_n}K_n \right).
\]
	
\end{corollary}

\begin{proof}
For $r \in (\delta,R+\delta)$, we have
\begin{align*}
 | \hat{g}_n(r;K_n) - g(r)  | \leq \left | \sum_{k=1}^{K_n} (\hat{\theta}_{k,n} - \theta_{k}^{\star}) \varphi_k(r-\delta) \right | + \left | \sum\limits_{k>K_n} \theta_{k}^{\star} \varphi_k(r-\delta) \right |
\end{align*}
which implies by Cauchy-Schwarz inequality and since $\|\hat{\btheta}_{-0,K_n} - \btheta_{-0K_n}^{\star}\| \le \|\hat{\btheta}_{K_n} - \btheta_{K_n}^{\star}\|$ that
\[
 | \hat{g}_n(r;K_n) - g(r)  | \leq  \| \hat{\btheta}_{K_n} - \btheta_{K_n}^{\star} \| \times \| \widetilde{\boldsymbol{\varphi}}_{K_n}(r - \delta) \| + a_{K_n}.
\]
The result is deduced from Theorem~\ref{THM:CONST}(i) and Lemma~\ref{lem:Esk} which ensures that $\| \widetilde{\boldsymbol{\varphi}}_{K_n}(r - \delta) \|=O(\sqrt{K_n})$.
\end{proof}

\section{A multivariate extension of Theorem~\ref{THM:CONST}$\mathrm{(ii)}$}  \label{sec:theorem3}

The next result is a multivariate extension of Theorem~\ref{THM:CONST}(ii). It may be used for instance to build a statistical test of a specific parametric form for the pairwise interaction function. We did not investigate this in this paper. 

 \begin{proposition}
\label{THM:MULTASYMPT}
Let $\hat{\bg}_n(\br; K_n)=(\hat{g}_n(r_i;K_n))_{i=1,\dots,q}$ and $\bg(\br)  =  (g(r_i))_{i=1,\dots,q}$ where $\br=(r_1,\dots,r_q)^\top \in \mathbb{R}^q$ and $r_i \in (\delta,R+\delta)$, $1 \leq i \leq q$. Suppose that the assumptions of Theorem~\ref{THM:CONST}(ii) are satisfied. Then,
\begin{align*}
\boldsymbol{\Sigma}_{K_n}^{-1/2}(\br) \Big \{ \hat{\bg}_n(\br; K_n) - \bg(\br) \Big \} \xrightarrow{d} \mathcal{N}(\mathrm{0},\mathbf{I}_q)
\end{align*}
\end{proposition}
where
\begin{align}
\boldsymbol{\Sigma}_{K_n}(\br) = & \widetilde{\boldsymbol{\varphi}}_{K_n}^\top(\br - \delta) \,  \boldsymbol \Pi_{K_n} \, \widetilde{\boldsymbol{\varphi}}_{K_n}(\br - \delta), \nonumber \\
\widetilde{\boldsymbol{\varphi}}_{K_n}(\br - \delta) = & \left ( \widetilde{\boldsymbol{\varphi}}_{K_n}(r_i - \delta) \right)_{1 \leq i \leq q}.\label{sigma:multivariate}
\end{align} 
Note that the $(K_n+1,q)$ matrix $\widetilde{\boldsymbol{\varphi}}_{K_n}(\br - \delta)$ can be equivalently written in the following form
\[
  \widetilde{\boldsymbol{\varphi}}_{K_n}(\br - \delta) =
  \left[ {\begin{array}{cccc}
    0 & 0 & \cdots & 0\\
    \varphi_1(r_1-\delta) & \varphi_1(r_2-\delta) & \cdots & \varphi_1(r_q-\delta)\\
    \vdots & \vdots &  & \vdots\\
    \varphi_{K_n}(r_1-\delta) & \varphi_{K_n}(r_2-\delta) & \cdots & \varphi_{K_n}(r_q-\delta)\\
  \end{array} } \right].
\]

\begin{proof}
Let $\bt \in \mathbb{R}^d$. We have
\begin{align*}
\bt^\top \Big \{ \hat{\bg}_n(\br; K_n) - \bg(\br) \Big \} = & \, \sum_{i=1}^q t_i \, \Big \{ \hat{g}_n(r_i;K_n) - g(r_i) \Big \} \\
= & \, \bt^\top \widetilde{\boldsymbol{\varphi}}_{K_n}(\br - \delta)^\top \mathbf{A}_{K_n}^{-1}(\btheta_{K_n}^\star)  \boldsymbol{e}_{K_n}(\bX; \btheta^{\star}) + O_{\mathrm P}(x_n)
\end{align*}
where $x_n$ is given by~\eqref{eq:xn}. We may prove similarly to the proof of Theorem~\ref{THM:CONST}(ii) that $\bz =  \bt^\top \widetilde{\boldsymbol{\varphi}}_{K_n}(\br - \delta)^\top \mathbf{A}_{K_n}^{-1}(\btheta_{K_n}^\star) \in \mathcal S_m$.
 Hence, by Lemma~\ref{lemma2}(iv)
\begin{align*}
\sigma_\bt^{-1}  \, \bt^\top \widetilde{\boldsymbol{\varphi}}_{K_n}(\br - \delta)^\top \mathbf{A}_{K_n}^{-1}(\btheta_{K_n}^\star)  \boldsymbol{e}_{K_n}(\bX; \btheta^{\star}) \xrightarrow{d} \mathcal{N}(0,1)
\end{align*} 
where
$\sigma_\bt  =  \Big ( \bt^\top\boldsymbol{\Sigma}_{K_n}(\br) \, \bt \Big )^{1/2}$ (note that $\boldsymbol{\Sigma}_{K_n}(\br)$ is given by (\ref{sigma:multivariate})). Using Lemma~\ref{lemma2}(iii), we have ${\displaystyle \liminf_{n\to \infty} \;   |W_n|^{-1} \sigma_\bt^2 >0}$ whereby we deduce that $\sigma_\bt^{-1} = O(|W_n|^{-1/2})$. With similar arguments used in the proof of Theorem~\ref{THM:CONST}(ii), we have that $\sigma_t K_n^{1/2} x_n=o(1)$. Hence, by Slutsky's Theorem
\begin{align*}
\sigma_\bt^{-1}  \, \bt^\top \Big \{ \hat{\bg}_n(\br; K_n) - \bg(\br) \Big \}  \xrightarrow{d} \mathcal{N}(0,1).
\end{align*} 
Finally, using Lemma 2.1 of \cite{biscio2018note}, we deduce that 
\[
\boldsymbol{\Sigma}_{K_n}^{-1/2}(\br) \Big \{ \hat{\bg}_n(\br; K_n) - \bg(\br) \Big \} \xrightarrow{d} \mathcal{N}(\mathrm{0},\mathbf{I}_d).
\]
\end{proof}

\section{Additional Figures}

\begin{figure}[H]
\begin{center}
\renewcommand{\arraystretch}{0}
\includegraphics[width=1\textwidth]{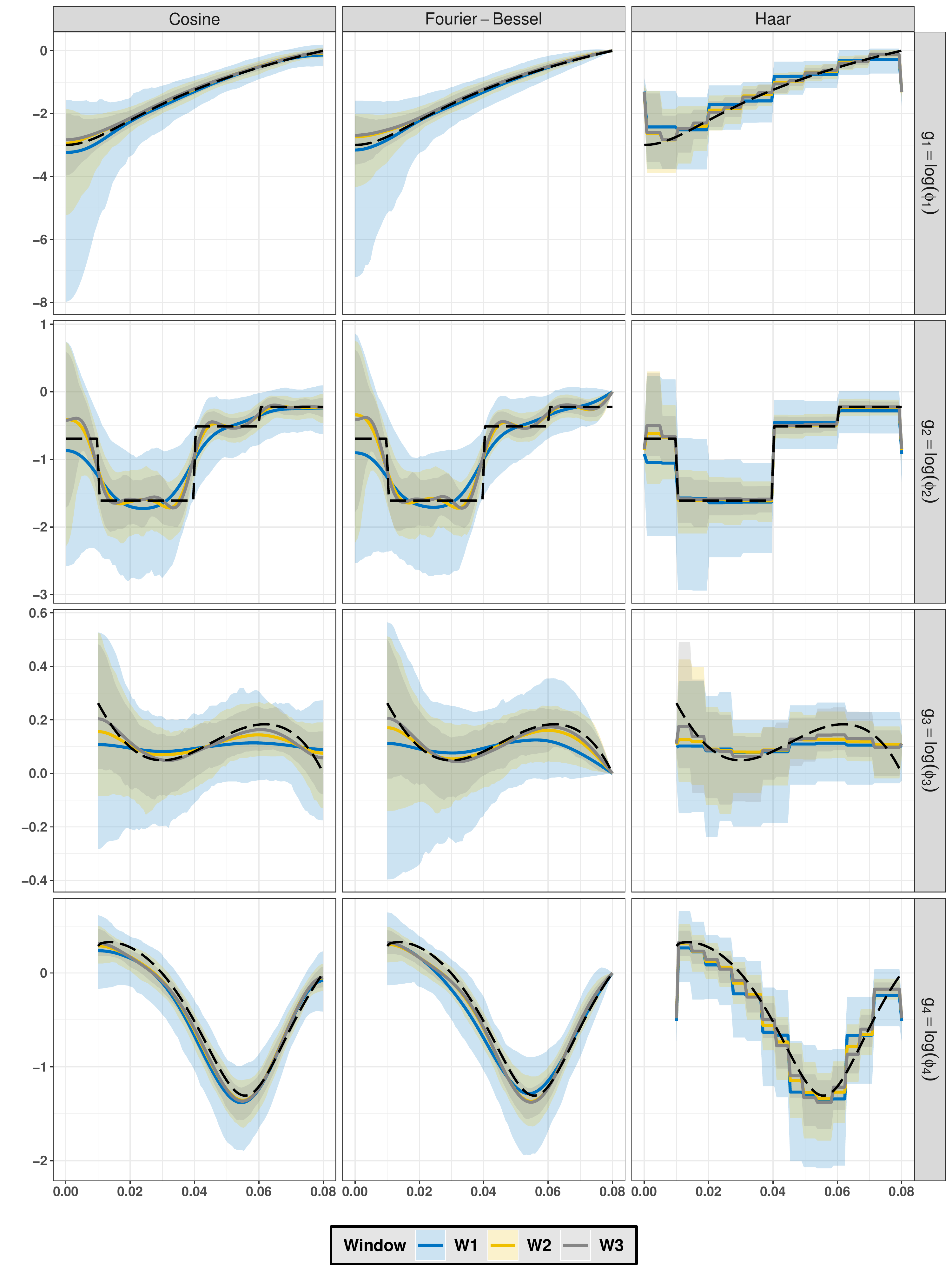} 
\caption{For $l=1,\ldots,4$,  theoretical logarithm of pairwise interaction functions $g_l=\log \phi_l$ (dashed black curves), Monte Carlo means of $\hat{g}_l(\cdot,\hat K_{\mathrm{cAIC}}^{\star})$ in $W_j$ ($j=1,2,3$). Envelopes correspond to 95\% Monte-Carlo pointwise confidence intervals.}
\label{fig:envelopes.cAIC.gs}
\end{center}
\end{figure}

\begin{figure}[H]
\begin{center}
\renewcommand{\arraystretch}{0}
\includegraphics[width=1\textwidth]{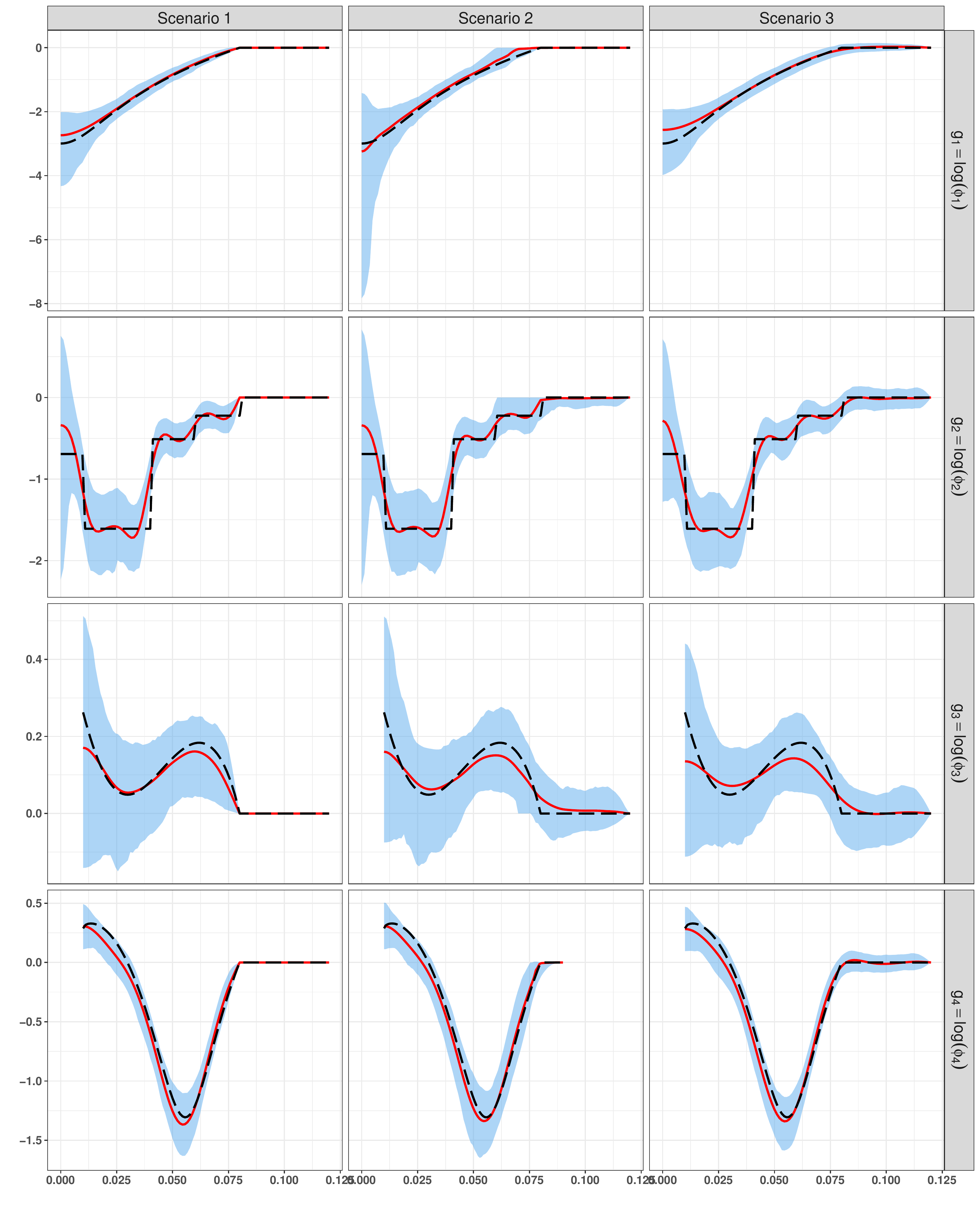} 
\caption{For $l=1,\ldots,4$, theoretical logarithm of pairwise interaction functions $g_l=\log \phi_l$ (dashed black curves), Monte Carlo means of $\hat{g}_l(\cdot,\hat K_{\mathrm{cAIC}}^{\star})$ (red curves) using the Fourier-Bessel basis in $W_2$ when $\delta$ and $R$ are known (first column), estimated (second column), or when $R+\delta$ is set to $R_{\max}=0.12$ (third column). Envelopes correspond to 95\% Monte-Carlo pointwise confidence intervals.}
\label{fig:scenario.gs}
\end{center}
\end{figure}

\begin{figure}[H]
\begin{center}
\renewcommand{\arraystretch}{0}
\includegraphics[width=1\textwidth]{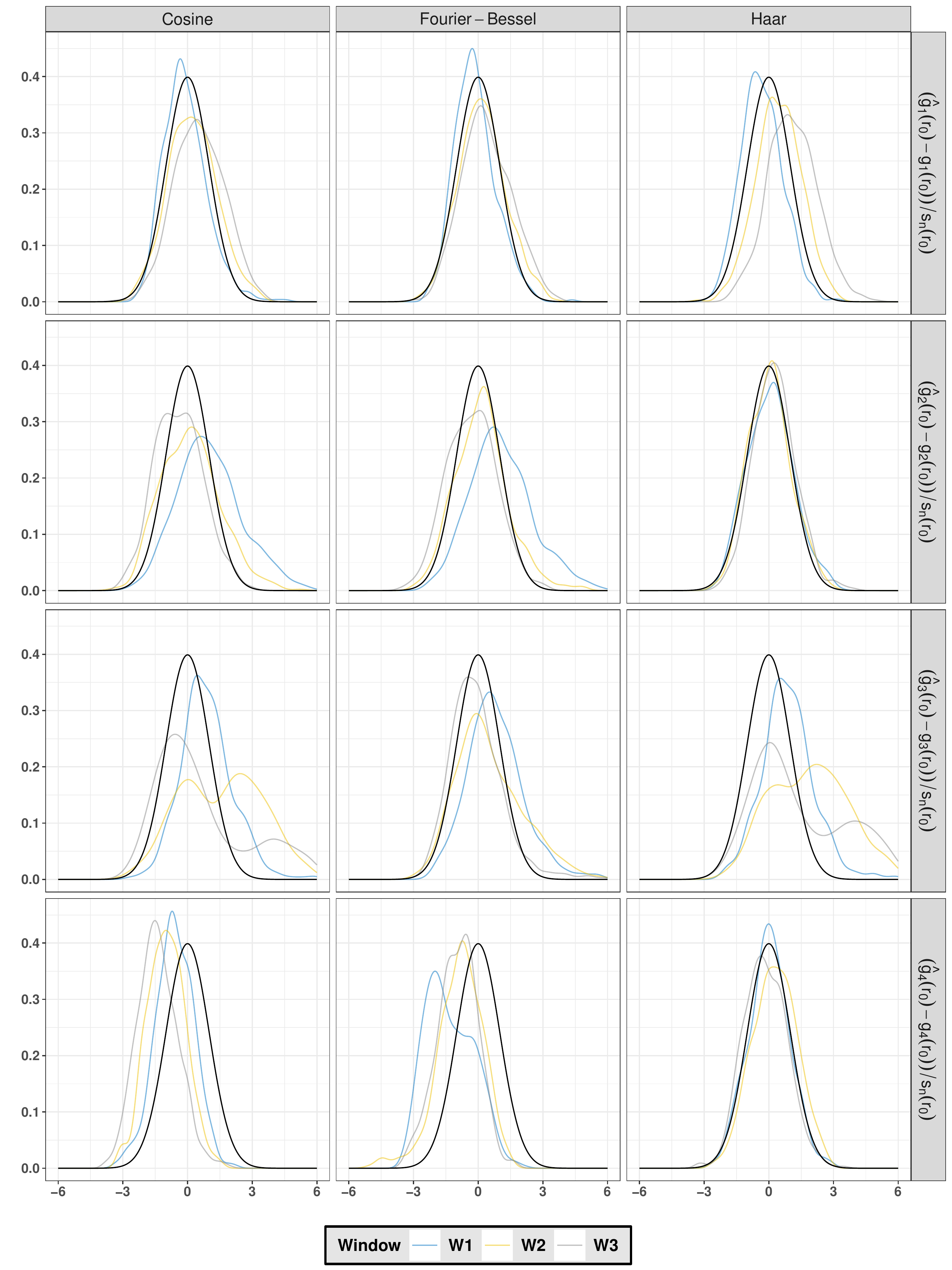} 
\caption{For $l=1,\ldots,4$, curves of $s_n^{-1}(r_0) \left \{ \hat{g}_l(r_0;\hat K_{\mathrm{cAIC}}^{\star}) - g_l(r_0) \right \}$ at $r_0=0.035$ in $W_j$ ($j=1,2,3$). The black curve is the standard normal distribution.}
\label{fig:proof.theorem21}
\end{center}
\end{figure}

\begin{figure}[H]
\begin{center}
\renewcommand{\arraystretch}{0}
\includegraphics[width=1\textwidth]{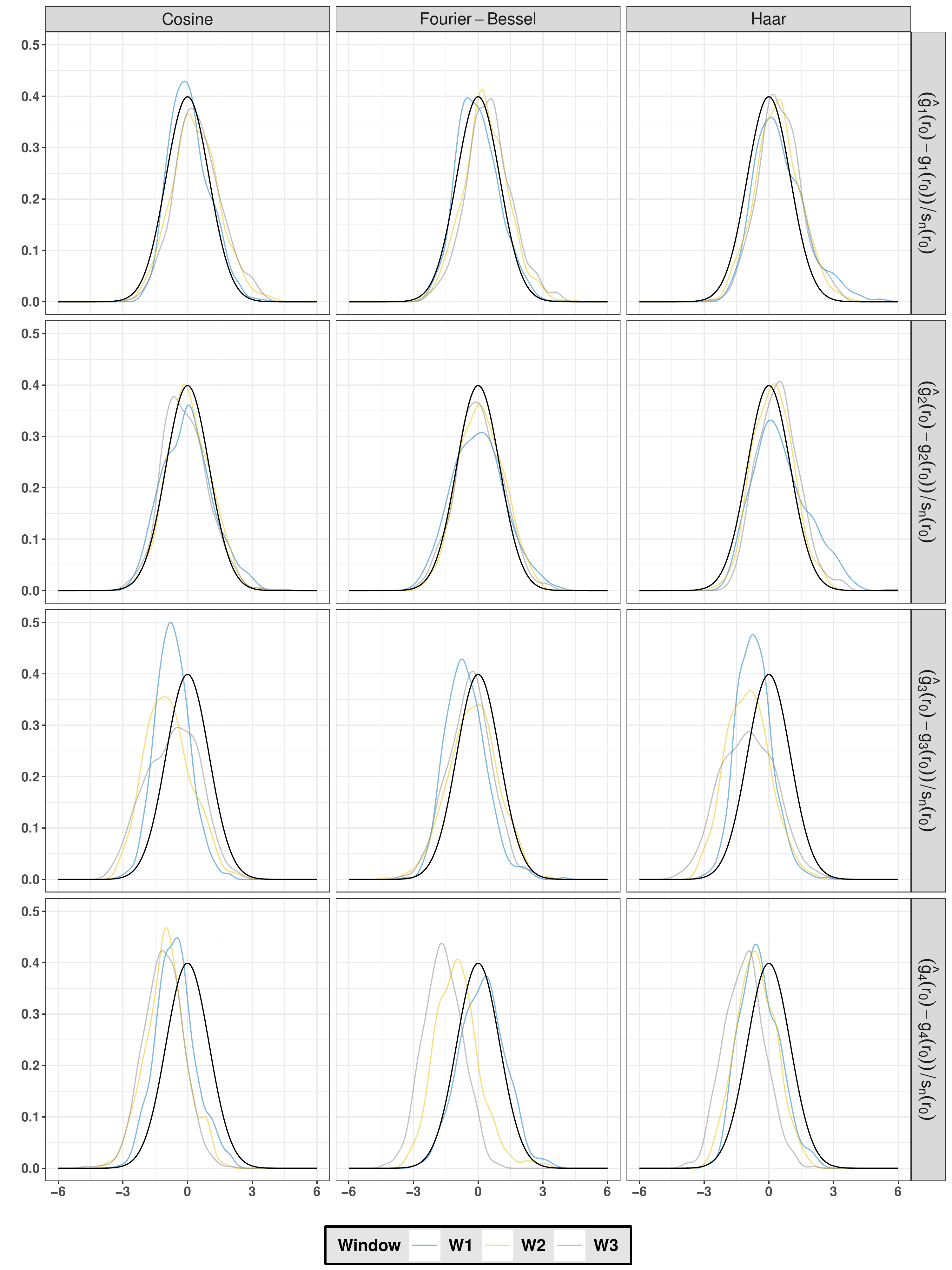} 
\caption{For $l=1,\ldots,4$, curves of $s_n^{-1}(r_0) \left \{ \hat{g}_l(r_0;\hat K_{\mathrm{cAIC}}^{\star}) - g_l(r_0) \right \}$ at $r_0=0.05$ in $W_j$ ($j=1,2,3$). The black curve is the standard normal distribution.}
\label{fig:proof.theorem22}
\end{center}
\end{figure}

\end{document}